\documentclass[11pt]{amsart}
\usepackage{mathtools}

\title[Integrable systems approach to the Schottky problem]
{Integrable systems approach to the Schottky problem and related questions}

\author{Samuel Grushevsky}
\address{Department of Mathematics and Simons Center for Geometry and Physics, Stony Brook University, Stony Brook, NY 11794-3651}
\email{sam@math.stonybrook.edu}
\thanks{Research of the first author was supported in part by NSF grant DMS-21-01631}

\author{Yuancheng Xie} 
\address{Center for Computational Mathematics and Control Research, Shenzhen MSU-BIT University, Shenzhen, Guangdong 518172, China}
\email{xieyuancheng@smbu.edu.cn}
\thanks{The work of the second author was partially supported by the National Key Research and Development Program of China (No. 2021YFA1002000), and by the National Natural Science Foundation of China under the Grant No. 12301304. }

\date{\today}

\usepackage{graphicx}
\usepackage{appendix}
\usepackage{graphics}
\usepackage{MnSymbol}
\usepackage{tikz-cd}
\usepackage{hyperref}
\usepackage{cleveref}
\numberwithin{equation}{section}
\usepackage{amsmath, amsthm, amsbsy,amsfonts,amscd, mathrsfs}

\newtheorem{theorem}{Theorem}[section]
\newtheorem{claim}[theorem]{Claim}
\newtheorem{fact}[theorem]{Fact}
\newtheorem{proposition}[theorem]{Proposition}
\newtheorem{propdef}[theorem]{Proposition-Definition}
\newtheorem{lemma}[theorem]{Lemma}
\newtheorem{problem}[theorem]{Problem}
\newtheorem{corollary}[theorem]{Corollary}
\newtheorem{conjecture}[theorem]{Conjecture}

\theoremstyle{definition}
\newtheorem{definition}[theorem]{Definition}
\newtheorem{example}[theorem]{Example}

\theoremstyle{remark}
\newtheorem{remark}[theorem]{Remark}
\newtheorem{exercise}[theorem]{Exercise}
\newtheorem*{hint}{Hint}
\newtheorem*{solution}{Solution}

\newcommand{\ttt}{\mathbf t}
\newcommand{\ZZ}{\mathbb Z}
\newcommand{\RR}{\mathbb R}

\newcommand{\NN}{\mathbb N}
\newcommand{\CC}{\mathbb C}
\newcommand{\PP}{\mathbb P}

\newcommand{\HH}{\mathbb H}

\newcommand{\cW}{\mathcal{W}}
\newcommand{\cA}{\mathcal{A}}
\newcommand{\cB}{\mathcal{B}}
\newcommand{\cE}{\mathcal{E}}

\newcommand{\cF}{\mathcal{F}}
\newcommand{\cL}{\mathcal{L}}
\newcommand{\cM}{{\mathcal M}}

\newcommand{\cO}{{\mathcal O}}
\newcommand{\cV}{{\mathcal V}}
\newcommand{\cD}{{\mathcal D}}

\newcommand{\cJ}{{\mathcal J}}
\newcommand\half{{\textstyle\frac12}}
\newcommand\ve{\varepsilon}

\DeclareMathOperator{\const}{const}

\DeclareMathOperator{\AJ}{AJ}
\DeclareMathOperator{\BA}{BA}
\DeclareMathOperator{\Res}{Res}
\renewcommand{\Im}{\operatorname{Im}}
\renewcommand{\Re}{\operatorname{Re}}

\DeclareMathOperator{\Jac}{Jac}
\DeclareMathOperator{\Pic}{Pic}
\DeclareMathOperator{\Div}{Div}
\DeclareMathOperator{\Kum}{Kum}

\DeclareMathOperator{\Sing}{Sing}
\DeclareMathOperator{\SL}{SL}
\DeclareMathOperator{\Sp}{Sp}
\DeclareMathOperator{\Mat}{Mat}

\DeclareMathOperator{\mult}{mult}

\DeclareMathOperator{\ord}{ord}
\renewcommand{\div}{\operatorname{div}}

\newcommand{\thmrefer}[1]{\renewcommand\thetheorem
  {\protect\ref{#1}}\addtocounter{theorem}{-1}}

\begin{document}

\begin{abstract}
We give a somewhat informal introduction to the integrable systems approach to the Schottky problem, explaining how the theta functions of Jacobians can be used to provide solutions of the KP equation, and culminating with the exposition of Krichever's proof of Welters' trisecant conjecture in the most degenerate (flex line) case.
\end{abstract}

\dedicatory{Dedicated to the memory of Igor Krichever}

\maketitle

\setcounter{tocdepth}{1}
\tableofcontents

\section{Introduction}

These lectures will focus on the relations between integrable systems and algebraic geometry. Let us first give some motivations for why we put these two subjects together. 
Consider the following naive question: How do we solve a given differential equation? This is a very vague question, but clearly of central mathematical importance. When we ask such a question, we need to know what it really means to {\em solve} a differential equation. For linear differential equations with constant coefficients, this is clear and well-known --- all solutions are given by linear combinations of suitable exponentials and polynomials, and are completely explicit. But what if we have linear differential equations with nonconstant coefficients? In a usual course of ordinary differential equations, something may be mentioned about such equations, but usually not much; perhaps some approximation techniques come to mind. In general if we write a complicated differential equation, there is no reason for it to have a nice solution, and almost surely no solutions will be elementary functions. For example, hypergeometric functions satisfy the so-called hypergeometric equations, and it took hundreds of years to have a decent understanding of them (hypergeometric series was first used by John Wallis in his 1655 book Arithmetica Infinitorum \cite{Wallis1972}, and the understanding slowly progressed until Riemann resolved the connection problem for hypergeometric functions in 1857 \cite{Riemann1857a}). Sometimes when we try to solve a differential equation, it could happen that the best way to describe the function that solves it is just to say that it is a solution of this differential equation --- and one may not necessarily hope to get a better direct description. So here `solve' would mean getting an alternative description of this function beyond just saying that it solves a particular differential equation.

There are in fact systems of differential equations which, while impossible to solve explicitly in elementary functions, are completely integrable in the sense that all integrals of motion can be determined. And then there is something that will lead us to algebro-geometric solutions, by which we mean solutions constructed from algebro-geometric data; perhaps hypergeometric functions are the first example of such solutions. The modern study of this, and the main results we discuss, is largely due to the works by Igor Krichever from the 1970s to 2010s.

\smallskip
So why does an algebraic curve come into play in the story of solving differential equations, in the first place? A simple example may explain the matter. Consider the integral
\begin{equation}\label{eq:trigonometricintegral}
y(x) \coloneqq \int_{0}^x \frac{dt}{\sqrt{1 - t^2}}, \quad -1 \le x \le 1\,.
\end{equation}
This integral can be calculated by making the `magic' substitution 
\begin{equation}\label{eq:tsintheta}
t \coloneqq \sin(\theta)\,, 
\end{equation}
which gives
\[y = \arcsin{x}, \qquad -1 \le x \le 1\,,\]
in the original variables. An important goal in the study of integrable systems is to find a magic substitution like \eqref{eq:tsintheta}, so that we can calculate the desired integrals or solve the desired equations. The deeper reason why the integral \eqref{eq:trigonometricintegral} can be calculated is because the semi-circle defined by $s = \sqrt{1 - t^2}$ can be rationally uniformized, that is there is a way to simultaneously represent $t$ and $s$ as rational functions of a third variable $z$: $t = \frac{2z}{1 + z^2}$, $s = \frac{1 - z^2}{1 + z^2}$ with $|z| \le 1$. But such elementary parameterization in general does not even exist for the seemingly slightly more general integral
\begin{equation}\label{eq:ellipticintegral}
y(x) \coloneqq \int_0^x \frac{dt}{\sqrt{4t^3 - g_2t - g_3}}, \qquad g_2, g_3 \in \CC\,.
\end{equation}
This is due to the fact that while the equation $s^2=1-t^2$ defines a rational curve (that is, a genus zero Riemann surface), the equation
\begin{equation}\label{eq:elliptic}
s^2 = 4t^3 - g_2t - g_3
\end{equation}
defines an elliptic curve (a genus one Riemann surface), which thus cannot be rationally parameterized. It is clear that to uniformize the elliptic curve we need to introduce some new functions. It was Gauss who made the important observation in his famous book Disquisitiones Arithmeticae \cite{Gauss1995} that analogously to the trigonometric integral \eqref{eq:trigonometricintegral}, in integral \eqref{eq:ellipticintegral} we should view $x$ as a function of $y$ to define this new function. This proposal was noticed and carried out by Abel, and the results are systematically exposed in Jacobi's Fundamenta Nova \cite{Jacobi2012}. This new function is the famous Weierstrass  function $\wp(x)$, which can be used to uniformize the elliptic curve defined by equation \eqref{eq:elliptic}, since it satisfies the differential equation
\[(\wp'(x))^2 = 4\wp(x)^3 - g_2\, \wp(x) - g_3\,.\]
The integral \eqref{eq:ellipticintegral} defines the so-called Abel-Jacobi map on the elliptic curve \eqref{eq:elliptic}, and can be generalized to the case of Riemann surfaces of higher genus, which was done by Riemann \cite{Riemann1857b}. It took another century for the realization that these Abel-Jacobi maps are exactly the magic transformations needed to obtain so-called finite-gap solutions of many integrable systems. 

\smallskip
Now we turn to the geometry of algebraic curves. In the following by a curve we always mean a complex projective, i.e.~algebraic and compact, reduced curve, which is the same as a compact Riemann surface. One can associate to a genus~$g$ curve~$C$, which is fundamentally a geometric object, its Jacobian, denoted $\Jac(C)$. The Jacobian is a $g$-dimensional complex torus, and in fact a principally polarized abelian variety. That is, the Jacobian is an algebraic variety that can be embedded into the complex projective space, and which has the structure of an abelian group. The Jacobian can be written as the quotient $\CC^g/\Lambda$ by a full rank lattice $\Lambda\cong \ZZ^{2g}$, with $\Lambda\otimes_\ZZ\RR=\CC^g$ (all of these we will define properly in due course), and as such is an object of a more arithmetic flavor. The map that sends a curve to its Jacobian is called the Torelli map, because Torelli's theorem is the statement that a curve can be recovered from its Jacobian. More precisely, this is the statement that the map
\[\begin{array}{rcl}
J: \cM_g & \rightarrow & \cA_g\\
C & \mapsto & \Jac(C)
\end{array}\]
defines an embedding of the moduli space~$\cM_g$ of genus~$g$ curves into the moduli space~$\cA_g$ of genus~$g$ (principally polarized, and as this will be the only type that occur in these lectures, we will henceforth drop these words) abelian varieties. \footnote{There is a caveat here that at the level of stacks this map is in fact 2-to-1 onto its image; we will discuss this in a little more detail later on.}

The Torelli map thus relates the geometric data of a curve with the arithmetic data of its Jacobian. There are, however, more abelian varieties than Jacobians: more precisely $\dim_\CC \cM_g=3g-3<\dim_\CC\cA_g=\tfrac{g(g+1)}{2}$ for $g\ge 4$. The 150-year old Schottky problem is to determine the image of the Torelli map.

\begin{problem}[Schottky problem]
    Determine the image of the Torelli map~$J$, i.e.~determine which abelian varieties arise as Jacobians of curves.
\end{problem}
There are many interpretations of this problem, depending on what we mean by `determine' --- one could deduce special geometric properties of Jacobians of curves and advance the understanding of their geometry, one can write equations for the Jacobian locus $J(\cM_g)$, \dots There are many approaches to the Schottky problem, some of which will be mentioned below, and essentially each of these approaches has been fundamental in the development of a mathematical theory: (Siegel) modular forms, Hodge theory, and integrable systems. 

There is also a well-known easier problem, which for lack of a standard term we will call 
\begin{problem}[Restricted Schottky problem]
    Given a curve $C$ embedded into a (principally polarized) abelian variety~$A$, is~$A$ the Jacobian~$\Jac(C)$? 
\end{problem}
\noindent
This is a much easier problem because given an abelian variety~$A$, the full Schottky problem asks if there exists a curve (which can be taken inside~$A$) such that~$A$ 
is its Jacobian, while in the restricted version the curve is already magically given to us. Still, the restricted Schottky problem is highly nontrivial, and various results in this direction will appear in these lectures in due course. Our main focus in these notes is Krichever's complete solution to the (full) Schottky problem.

On the one hand, we are going to solve suitable differential equations via some functions constructed starting from an arbitrary curve (and using its Jacobian). On the other hand, this same set of differential equations will provide a solution to the Schottky problem from pure algebraic geometry. This is why we put these two stories together. Starting from the curve, we will be able to construct some solutions of some differential equation, and the fact that we have been able to do this will characterize Jacobians of curves among all abelian varieties. This means that if we try to perform such a construction starting from any abelian variety, we will succeed if and only if this abelian variety is a Jacobian. The capstone result in this direction is  Krichever's proof of Welters' trisecant conjecture characterizing Jacobians among abelian varieties by the property of their Kummer variety having a trisecant line. The aim of these lectures is to present, following Krichever \cite{Krichever2006, Krichever2010, Krichever2013, Krichever2023}, the outline of the proof of the case of this characterization where the trisecant is most degenerate --- i.e.~when the 3 points of secancy come together to form a flex line of the Kummer variety.

\subsection*{The aim and the structure of these notes}
These lecture notes are loosely based on a set of 5 lectures given by the first author at BIMSA in June 2024, however, they include lots more material, additional background, details of the constructions, many more references, etc. We have tried to preserve the more informal style of exposition, giving more motivation and philosophy for the ideas and constructions, sometimes at the expense of more details of the arguments. These notes are not meant to be fully self-contained nor to be a textbook survey for the precise elements of the proofs. Rather, we hope that the more eclectic material selection, and the exposed interplay between constructions from classical algebraic geometry of curves and abelian varieties, theta functions, and integrable systems, may serve as a worthwhile introduction to this circle of ideas for mathematicians with a background either in integrable systems or in algebraic geometry. The original lectures were meant for intermediate-to-advanced graduate students, and this is the level of the background generally assumed.

The material we cover brings together notions and constructions from classical algebraic geometry and integrable systems, culminating in Krichever's proof of the characterization of Jacobians by the flex line of the Kummer variety. Our sections follow the lectures in alternating between topics: \Cref{sec:CommDO} exposes the basics of formal solutions of differential operators in one variable, of commuting differential operators, and ends with the appearance of the spectral curve; \Cref{sec:Jac} introduces Jacobians of curves and abelian varieties, from the analytic and algebraic viewpoint, and discusses Riemann's theta singularity theorem; in \Cref{sec:theta} we discuss the geometry of theta divisors of Jacobians, Weil's reducibility, and end by the discussion of trisecants of Kummer varieties of Jacobians, including Fay-Gunning's theorem, Shiota's theorem, and the statement of Welters' conjecture; in \Cref{sec:BAfunction} we describe Krichever's construction of the Baker-Akhiezer functions starting from the algebro-geometric data, showing how these can be used to construct solutions of the KP hierarchy; finally in \Cref{sec:flex} we explain Krichever's characterization of Jacobians by the existence of a flex line of the Kummer variety.

There are certainly many other wonderful surveys and expositions on a related circle of ideas, often more focused on either the integrable systems or on the algebraic geometry aspects, including those by Krichever himself: \cite{Mumford1975, vanderGeer1985, Donagi1988, Debarre1995, Marini1998, vanderGeer-Oort1999, Grushevsky2012, Dubrovin1981,Taimanov1990, Taimanov1997, Arbarello-DeConcini1990, mulase, Krichever2006, Krichever2013, Krichever-Shiota2013, Krichever2023}. We cannot supplant these, but hope that these notes may raise the readers' interest in learning more of all aspects of the story. Our exposition of Krichever's proof of the flex line case of the Welters' conjecture in spirit follows the outline given in \cite{Krichever2023}, but we provide more background and detail (hopefully without introducing extra mistakes).

\section*{Acknowledgments}
We are very grateful to the organizers of the first Beijing Summer Workshop in Mathematics and Mathematical Physics, at BIMSA, in June 2024, for the invitation for the first author to give lectures and for the second author to serve as the TA for this lecture course, and for very strong encouragement for us to prepare these lecture notes. We thank the audience for these lectures for active participation, many interesting questions and discussions that led us to further refine and enhance the presentation.

The second author gave a similar short course at the Great Bay University and Shenzhen University with these lecture notes. He thanks Fudong Wang and Chengfa Wu for the invitation. The second author also would like to thank Yuji Kodama for reading these notes and commenting on the construction of Baker-Akhiezer function on the singular spectral curves. We are thankful to Harry Braden for useful comments and corrections on the first version of these notes. We thank Motohico Mulase for bringing the reference \cite{mulase} to our attention, and for further comments on the multivariable KP hierarchy.

\smallskip
We dedicate this survey to the memory of Igor Krichever whose work and life created and shaped much of the mathematics discussed here. The first author will always cherish the memory of the collaboration and conversations with Igor, and whatever he knows about integrable systems is certainly owed to Igor. Of course any mistakes in the current text are our own, though.

\section{Commuting differential operators}\label{sec:CommDO}
In the first lecture we discuss commuting differential operators in one variable. While quite formal, this theory underpins the further developments, and we will see the algebraic curve naturally arising from a pair of commuting differential operators.

Let $L$ be a general differential operator in one variable $x$, where for now $x$ can be either a real or a complex number. More precisely, suppose
\[L = \sum \limits_{i = 0}^n u_i(x)\frac{d^i}{dx^i}\,.\]
Here the functions $u_i(x)$ can be holomorphic or smooth functions depending on the context.

We want to find eigenfunctions of $L$, i.e.~solutions of the equation
\[
L \psi(x)  = E \cdot \psi(x)
\]
for a constant $E\in\CC$ or $E\in\RR$ (denoted this way because we are thinking of the energy). For $\psi(x)$, we are interested in all `formal eigenfunctions', i.e.~these will be suitable series with no assumptions on convergence. We first note that we have some freedom to normalize~$L$ a little bit to put it in a simpler standard form.
\begin{exercise}\label{ex:1}
Show that by a suitable change of variables, it is `enough' to consider 
\begin{equation}\label{eq:Lcanonical}
L = \frac{d^n}{dx^n} + \sum \limits_{i = 0}^{n - 2}u_i(x)\frac{d^i}{dx^i},
\end{equation}
where the top coefficient is $1$, and the subleading coefficient vanishes.
\end{exercise}
We give some hints and outlines of solutions of the exercises in \Cref{sec:hints}.

The point of the exercise is that there is some naturally defined equivalence relation and it is convenient to find a canonical form in each equivalence class. The canonical form \eqref{eq:Lcanonical} of $L$ has the advantage that we can view it as some kind of perturbation of the leading term $\frac{d^n}{dx^n}$, and the eigenvalue problem for the leading term $\frac{d^n}{dx^n}$ is easily solved:
\[\frac{d^n}{dx^n} e^{kx} = k^n e^{kx}\,,\]
where $E = k^n$ is the corresponding eigenvalue. Then one can hope that the eigenfunction $\psi(x)$ for general~$L$ would be some kind of perturbation of $e^{kx}$. 

\smallskip
Motivated by the above consideration and observation that the $E$-eigenspace of $\tfrac{d^n}{dx^n}$ is $n$-dimensional, it will be more useful for us to write $E=k^n$, and to solve the equation $L \psi(x)= k^n \cdot \psi(x)$. This is of course the same equation as before, but crucially for the whole story, we will now be studying the dependence of (suitably normalized) solutions on~$k$, and to anticipate this we will write this in the following form:
\begin{equation}\label{eq:normalizedeigenvalueproblem}
L\psi(x, k) = k^n \psi(x, k)\,,
\end{equation}
where of course~$L$ is still the same differential operator, in~$x$ only.

The first result is the existence and uniqueness of suitably normalized solutions, given as formal power series in~$k^{-1}$.
\begin{theorem}
For any fixed $x_0 \in \CC$ (resp.~$x_0\in\RR$) there exists a unique formal solution $\psi$ of \eqref{eq:normalizedeigenvalueproblem} of the form
\begin{equation}\label{eq:normalizedeigenfunctionpsi}
\psi(x, k) = \left(\sum \limits_{s = 0}^{\infty} \xi_s(x) k^{-s}\right) e^{k(x - x_0)},
\end{equation}
where $\xi_s$ are holomorphic (resp.~smooth) functions of~$x$, 
and such that the solution is normalized at the point $x_0$ in the sense that~$\xi_0(x) \equiv 1$ and $\xi_s(x_0) = 0$  for all $s > 0$.
\end{theorem}
We denote the above eigenfunction normalized at $x = x_0$ by $\psi(x, k; x_0)$. 
\begin{remark}
Here by saying $\psi(x)$ is a formal eigenfunction we simply mean that we do not worry about any convergence of the series. Note that the condition $\xi_0(x) \equiv 1$ does not depend on $x_0$, which means the lowest order term is just $e^{k(x - x_0)}$, and the condition $\xi_s(x_0) = 0$ means that there are no corrections at higher orders at the point $x = x_0$. 
\end{remark}
\begin{proof}
The proof is basically a direct formal computation. We want our $\psi$ to be
\[\psi(x) = e^{k(x - x_0)} \left(1 + \xi_1(x)k^{-1} + \xi_2(x)k^{-2} + \dots\right)\]
where all $\xi_s(x)$ vanish at $x = x_0$. Applying $L$ to $\psi$, and noting that
\begin{align*}\frac{d^i}{dx^i}\left(\xi_s(x)k^{-s}e^{k(x - x_0)}\right) &= k^{-s}\left(\sum \limits_{l = 0}^i\binom{i}{l}\frac{d^{i - l}\xi_s}{dx^{i - l}} \cdot k^l\right)e^{k(x - x_0)}\\ &= k^{-s}\left(\frac{d^i\xi_s}{dx^i} + i\frac{d^{i - 1}\xi_s}{dx^{i - 1}} \cdot k + \dots \right)e^{k(x - x_0)}\,,\end{align*}
we get
\begin{align*}\sum \limits_{s = 0}^{\infty} \sum \limits_{j = 0}^n \binom{n}{j}\frac{d^j\xi_s}{dx^j}k^{n - j - s}e^{k(x - x_0)} &+ \sum \limits_{s = 0}^{\infty} \sum \limits_{i = 0}^{n - 2} u_i(x) \sum \limits_{j = 0}^i \binom{i}{j}\frac{d^j\xi_s}{dx^j}k^{i - j - s} e^{k(x - x_0)} \\ &= \left(\sum \limits_{s = 0}^{\infty} \xi_s(x)k^{n-s}\right) e^{k(x - x_0)}.\end{align*}
Canceling the nonzero factor $e^{k(x - x_0)}$, we get a differential equation for the coefficients $\xi_s(x)$. As this equation is valid for any $k$, the coefficient  of each $k^{n-l}$ must vanish, and computing these gives
\begin{align} \label{eq:recursivexi}
    \sum \limits_{j = 0}^l \binom{n}{j}\frac{d^j\xi_{l-j}}{dx^j} + \sum \limits_{i = 0}^{n - 2} u_{n - 2 - i}(x)\sum \limits_{j = 0}^{l - 2 - i}\binom{n - 2 - i}{j}\frac{d^j \xi_{l - 2 - i - j}}{dx^j} = \xi_l(x), \quad \forall l \ge 0\,. 
\end{align}
For $l = 0$, we obtain the equality
\[\xi_0 = \xi_0\,.\]
For $l \ge 1$, we obtain the differential equation
\[n \frac{d\xi_{l - 1}}{dx} = - \sum \limits_{j = 2}^{l}\binom{n}{j}\frac{d^j\xi_{l - j}}{dx^j} - \sum \limits_{i = 0}^{n - 2} u_{n - 2 - i}(x)\sum \limits_{j = 0}^{l - 2 - i}\binom{n - 2 - i}{j}\frac{d^j\xi_{l - 2 - i - j}}{dx^j}\,.\]
This system is upper triangular when written in terms of the column vector $(\xi_0, \xi_1, \xi_2, \dots)^t$, so it can be  solved recursively. Notice that every time we just need to solve a first order differential equation whose solution is thus uniquely determined by the conditions
\[\xi_0(x) = 1, \qquad \xi_s(x_0) = 0, \quad s \ge 1\,.\]
So now we have both existence and uniqueness of $\psi(x, k)$, which completes the proof.
\end{proof}

\begin{remark}
We are not going to get into the discussion of whether one can sum such formal series in any way, that's not the point. Note however that we did not say anything about what $k$ is, so this is a true statement for any $k$, and $\psi$ has a very explicit dependence on $k$, which is essential. Here, the eigenvalues~$k$ can deform arbitrarily, and the solutions will deform nicely if we deform the eigenvalue. This is crucial to the whole story.
\end{remark}

\begin{corollary}\label{cor:formalsolutionsall}
Any other formal solution of $L\psi = k^n \psi$ has the form
\begin{equation}
\psi(x, k) = \psi(x, k; x_0)A(k, x_0)
\end{equation}
where $A(k, x_0)$ is a function that only depends on $k$ and $x_0$. 
\end{corollary}
\begin{proof}
As \eqref{eq:normalizedeigenvalueproblem} is a linear differential equation on $x$, for any $A(k, x_0)$ independent of $x$
\[\psi(x,k) = \psi(x, k; x_0)A(k, x_0)\,,\]
clearly solves the same eigenvalue problem as~$\psi(x, k; x_0)$. For the other direction, we need to show that for any eigenfunction of $L$ with eigenvalue $E = k^n$, there exists a unique $A(k, x_0)$ such that the above identity is valid. Given any formal eigenfunction~$\psi(x, k)$ of $L$ with eigenvalue $E = k^n$, we define
\[A(k, x_0) \coloneqq \psi^{-1}(x_0, k; x_0) \cdot \psi(x_0, k)\,.\]
Then $\psi(x, k)A^{-1}(k, x_0)$ is also an eigenfunction of $L$ with eigenvalue $E = k^n$, and satisfies the normalization condition $\psi(x_0, k)A^{-1}(k, x_0) = 1$. By the uniqueness of $\psi(x, k; x_0)$ it follows that
$\psi(x, k; x_0) = \psi(x, k)A^{-1}(k, x_0)$.
\end{proof}

The above discussion completes what we require from the study of formal eigenfunctions of one differential operator in one variable. What we want to consider next is the case when we have more than one differential operator. Assume that we have two differential operators in the standard form
\begin{equation}\label{eq:nm}
\begin{aligned}
    & L_1 = \frac{d^n}{dx^n} + \sum \limits_{i = 0}^{n - 2}u_i(x) \frac{d^i}{dx^i}, \\
    & L_2 = \frac{d^m}{dx^m} + \sum \limits_{j = 0}^{m - 2} v_j(x) \frac{d^j}{dx^j}.
\end{aligned}
\end{equation}
We {\em assume that $m, n$ are coprime}, and the main example in our following discussion would be $n = 2, m = 3$. The non-coprime case is very interesting, but presents further challenges, and will not be needed for our main goals --- but see \cite{Krichever-Novikov1978, Krichever-Novikov1979, Krichever-Novikov1980, Grunbaum1988, Previato-Wilson1992, Mironov-Zheglov2016} for more information.

We know the formal eigenfunctions for $L_1$ and we can do the same analysis for the formal eigenfunctions of $L_2$. Recall that in finite-dimensional linear algebra if we have two linear operators,
we can ask when do they have a common eigenfunction or when are all of their eigenfunctions common? This question becomes even more interesting for our case of infinite dimensional spaces. We would like to ask when are all the eigenfunctions of $L_1$ and $L_2$ common, which is to say when is
$[L_1, L_2] = 0$?

\begin{exercise}\label{ex:L2L3commute}
When $n = 2, m = 3$, write down explicitly the conditions on the coefficients $u_0, v_0, v_1$ equivalent to the commutation relation
\[[L_1, L_2] = 0\,.\]
\end{exercise}
\begin{solution}
We calculate directly the commutator $[L_1, L_2]$ for $L_1 = \frac{d^2}{dx^2} + u_0$ and $L_2 = \frac{d^3}{dx^3} + v_1\frac{d}{dx} + v_0$. Note that the composition of operators satisfies the Leibniz rule, that is the composition of two monomial differential operators is given by
\[\left(f(x)\frac{d^i}{dx^i}\right)\circ\left(g(x)\frac{d^j}{dx^j}\right)  = \sum \limits_{\alpha = 0}^i \binom{i}{\alpha}f(x)g^{(\alpha)}(x)\frac{d^{i + j - \alpha}}{dx^{i + j - \alpha}}\,,\]
and the general case can be handled by combining this with the distributive property.
Note that $[L_1, L_2]$ is a second order differential operator as the coefficients for $\frac{d^i}{dx^i}, i \ge 3$ vanish automatically, and for this operator to be zero the coefficients of all $\frac{d^i}{dx^i}, 0 \le i \le 2$ have to be zero. So we obtain the following system of equations on $u_0, v_0, v_1$:
\begin{align*}
2\frac{dv_1}{dx} & = 3 \frac{du_0}{dx}\,,\\
\frac{d^2v_1}{dx^2} + 2 \frac{dv_0}{dx} & = 3 \frac{d^2u_0}{dx^2}\,,\\
\frac{d^2v_0}{dx^2} & = \frac{d^3u_0}{dx^3} + v_1 \frac{du_0}{dx}\,.
\end{align*}
Plugging $\frac{dv_1}{dx}$ from the first expression into the second, and then $\frac{dv_0}{dx}$ from the second expression into the third, we obtain the following differential equation for~$u_0$ only
\[\frac{1}{4}\frac{d^3u_0}{dx^3} + \left(\frac{3}{2}u_0 + \frac{c_1}{2}\right)\frac{du_0}{dx} = 0\,,\]
where $c_1$ is a constant of integration. If~$u_0$ satisfies this equation, there exist~$v_0$ and~$v_1$ making~$L_1$ and~$L_2$ commute.
\end{solution}

The next theorem gives a criterion for a differential operator to commute with a given one.
\begin{theorem}\label{thm:L12commute}
Let $L_1, L_2$ be as in \eqref{eq:nm}, and let $\psi(x, k; x_0)$ be the formal eigenfunction of $L_1$ normalized at $x = x_0$. Then $[L_1, L_2] = 0$ if and only if $\psi^{-1}(x, k; x_0) L_2 \psi(x, k; x_0)|_{x=x_0} = A(k)$ is independent of $x_0$.
\end{theorem}

\begin{remark}
As the leading term of $\psi(x, k; x_0)$ is $1$, we can make sense of $\psi^{-1}(x, k; x_0)$ near $x=x_0$ by formally inverting the series (and continuing not to worry about convergence). Note that $\psi^{-1}(x, k; x_0)L_2\psi(x, k; x_0)$ in fact does not depend on~$x$ by \Cref{cor:formalsolutionsall}. The point of this theorem is that it does not depend on~$x_0$ either.  
\end{remark}

\begin{proof}
This is again an elementary computation. We first show that the condition is necessary. Assume $[L_1, L_2] = 0$, then $L_2\psi(x, k; x_0)$ is also an eigenfunction of $L_1$ with eigenvalue~$k^n$, since
\[L_1L_2 \psi = L_2L_1\psi = L_2 (k^n\psi)\,.\]
Thus by \Cref{cor:formalsolutionsall}, for any~$x_0$, there exists $A(k, x_0)$ such that
\[L_2 \psi(x, k; x_0) = \psi(x, k; x_0) \cdot A(k, x_0)\,.\]
So $\psi^{-1}(x, k; x_0)L_2 \psi(x, k; x_0)$ is independent of $x$, and our goal is to show that $A(k, x_0)$ is also independent of $x_0$. For this take any other point $x = x_1$, then by the same argument we have
\[L_2 \psi(x, k; x_1) = \psi(x, k; x_1)A(k, x_1)\,,\]
and we need to show that $A(k, x_0) = A(k, x_1)$. Consider
\begin{align*}
\psi(x, k; x_0) e^{k(x_0 - x_1)} & = \left(1 + \xi_1(x)k^{-1} + \dots\right)e^{k(x - x_0)}e^{k(x_0 - x_1)} \\
& = (1 + \xi_1(x)k^{-1} + \dots) e^{k(x - x_1)}.
\end{align*}
This is also an eigenfunction of $L_1$ but in general it is not necessarily equal to $\psi(x, k; x_1)$ since it does not satisfy our normalization condition at $x = x_1$ which requires $\xi_s(x_1)=0$ for all $s > 1$, while by our construction we instead have $\xi_s(x_0)=0$. As it is an eigenfunction of $L_1$ with eigenvalue~$k^n$, by \Cref{cor:formalsolutionsall} there exists $B_{x_1}(k, x_0)$ such that
\[\psi(x, k; x_0)e^{k(x_0 - x_1)} = \psi(x, k; x_1) \cdot B_{x_1}(k, x_0)\,.\]
Then on the one hand we have
\begin{align*}
L_2 B_{x_1}(k, x_0)\psi(x, k; x_1) & = B_{x_1}(k, x_0)\psi(x, k; x_1)A(k, x_1) \\
& = \psi(x, k; x_0) e^{k(x_0 - x_1)}A(k, x_1)\,,
\end{align*}
while on the other hand
\[
L_2 B_{x_1}(k, x_0)\psi(x, k; x_1) = L_2 \psi(x, k; x_0)e^{k(x_0 - x_1)} = e^{k(x_0 - x_1)}\psi(x, k; x_0)A(k, x_0)\,, 
\]
since $L_2\psi(x, k; x_0) = \psi(x, k; x_0)A(k, x_0)$. 

So we obtained $A(k, x_0) = A(x, x_1)$ as desired. For the other direction, we need to show that if $\psi^{-1}(x, k; x_0) = L_2 \psi(x, k; x_0) = A(k)$, then $[L_1, L_2] = 0$.

To prove this, there is a {\em finite-dimensionality trick} that we are going to use in the following a couple of times. Note that
\[L_1L_2 \psi(x, k; x_0) = L_1A(k)\psi(x, k; x_0) = A(k)k^n \psi(x, k; x_0) \qquad \forall x_0\,,\]
and similarly we have
\[L_2L_1 \psi(x, k; x_0) = L_2 k^n L_1\psi(x, k; x_0) = k^n A(k)\psi(x, k; x_0) \qquad \forall x_0\,,\]
so that
\[[L_1, L_2]\psi(x, k; x_0) = 0 \quad \forall k, x_0\,.\]
However,  $[L_1, L_2]$ is a finite order differential operator that is independent of $k$, and thus unless it is identically zero, its kernel is a finite-dimensional vector space. But since $\psi(x, k; x_0)$ lies in its kernel for arbitrary $k$ and $x_0$ --- and the space of such is infinite-dimensional, which can be seen from the explicit dependence of $\psi(x, k; x_0)$ on $k$ --- we must have
\[[L_1, L_2] = 0\,,\]
which completes the proof. Note that this is where we truly benefit from varying $k$ in the family!
\end{proof}

\Cref{thm:L12commute} tells us that for~$L_1$ fixed we can find out if some other differential operator commutes with~$L_1$ by checking how $L_2$ acts on the eigenspaces of $L_1$. Now we can ask what happens if we have three differential operators $L_1, L_2, L_3$ of the type we consider, such that $[L_1, L_2] = [L_1, L_3] = 0$? Then we have the following

\begin{corollary}\label{cor:commutingDOring}
If $[L_1, L_2] = 0 = [L_1, L_3]$, then also $[L_2, L_3] = 0$.
\end{corollary}
\begin{proof}
This is now very easy. Take any formal eigenfunction $\psi(x, k; x_0)$ of $L_1$. Then since $[L_1, L_2] = [L_1, L_3] = 0$, \Cref{thm:L12commute} implies that there exist $A_2(k)$ and $A_3(k)$ such that
\begin{align*}
 L_2 \psi(x, k; x_0) &= A_2(k)\psi(x, k; x_0) \quad\hbox{and}\\
 L_3 \psi(x, k; x_0) &= A_3(k)\psi(x, k; x_0)\,,
\end{align*}
which then implies that $[L_2, L_3] \psi(x, k; x_0) = 0$. This is valid for any $k$ and any $x_0$, which makes the kernel of the differential operator $[L_2, L_3]$ infinite-dimensional, so by the same finite-dimensionality trick as in the proof of \Cref{thm:L12commute}, we have $[L_2, L_3] = 0.$
\end{proof}

\begin{remark}
This corollary is due to Schur \cite{Schur1905}. The problem of classifying three commuting differential operators of orders $3, 4, 5$ was studied in the work of Shabat \cite{Shabat1980}, and some explicit examples of commuting differential operators with any number of generators were constructed in the work of Kodama-Xie \cite{Kodama-Xie2021}. 
\end{remark}
Let's pause to appreciate what happened here. \Cref{cor:commutingDOring} claims that, if we start with a differential operator $L_1$ and consider all differential operators of the same form that commute with $L_1$, then they will also commute with each other. So this means that the differential operators that commute with a given differential operator $L_1$ {\em form a commutative ring}. Of course there are some elements that are trivially contained in this ring, such as the powers of $L_1$ or, if $L_1$ is a power of some other differential operator $L$, all powers of $L$. \footnote{Our standing assumption that $n$ and $m$ are coprime is precisely to avoid the case when both $L_1$ and $L_2$ are powers of some~$L$, in which case the commutativity would be trivial.} We can now ask to classify all commutative rings that can arise as the ring of differential operators commuting with some given~$L_1$. The next 100-year-old theorem provides some information in this direction.

\begin{theorem}[Burchnall-Chaundy~1923~\cite{Burchnall-Chaundy1923}]
If $[L_1, L_2] = 0$, then there exists a polynomial in two variables $Q(\alpha, \beta)$ such that $Q(L_1, L_2) = 0$.
\end{theorem}
\begin{proof}
For any given $E \in \CC$, let 
\[\cL(E) \coloneqq \{\psi : L_1 \psi = E\cdot \psi\}\,,\]
be the space of all formal eigenfunctions of $L_1$ with eigenvalue $E$. Then $\cL(E)$ is a finite-dimensional ($n$-dimensional in the current setup) complex vector space. Since $[L_1, L_2] = 0$, $L_2$ maps $\cL(E)$ into itself. Thus $L_2|_{\cL(E)}$ is a linear map of a finite-dimensional vector space into itself, given by a finite (in fact, $n\times n$) matrix in any chosen basis. We let $Q_E(\alpha)$ be the characteristic polynomial $\det(L_2\vert_{\cL(E)} - \alpha \cdot\text{id})$ of this matrix $L_2|_{\cL(E)}$. We now use the following
\begin{claim}
$Q_E(\alpha)$ depends polynomially on $E$.
\end{claim}
Given the claim, we can think of $Q_E(\alpha) \in \CC[\alpha, E]$ as a polynomial in two variables. Then $Q(L_1, L_2)$ is a differential operator of finite order, and we apply the finite-dimensionality trick again. Indeed, observe that each of the one-parameter family of functions $\psi(x, k; x_0)\in\cL(k^n)$ belongs to the kernel of~$Q(L_1,L_2)$, simply because $\psi(x,k;x_0)\in\cL(k^n)$ is annihilated by the characteristic polynomial $Q_E$ of $L_2|_{\cL(k^n)}$. Thus by varying~$k$ we see that $Q(L_1, L_2)$ has an infinite-dimensional kernel, so it must be identically zero.
\end{proof}
There are different ways to prove the claim --- below we use the symmetry of~$L_1$ and~$L_2$.
\begin{proof}[Proof of the claim]
For each $E$ we have the characteristic polynomial $Q_E(\alpha)$ --- it is a polynomial in~$\alpha$, and we want to prove that its coefficients are themselves polynomials in~$E$. For this we use the symmetry in the usage of~$L_1$ and~$L_2$. Indeed, the roots of $Q_E$ are the eigenvalues of $L_2|_{\cL(E)}$. However, if we start with $L_2$ and denote by $\cL'(\alpha)$ its eigenspace with eigenvalue $\alpha$, then $L_1$ will preserve $\cL'(\alpha)$. Let then $Q'_\alpha(E)$ be the characteristic polynomial of $L_1|_{\cL'(\alpha)}$, which depends polynomially on~$E$. However, $Q$ and $Q'$ are simply the eigenvalues of $L_1$ and $L_2$ on their common eigenfunctions, and thus are the same function, which is thus a polynomial in both variables.
\end{proof}
Another proof of this claim, by an explicit computation, was given by Krichever in \cite{Krichever1977}.

\begin{exercise}\label{ex:2} (see the appendix for a hint).
P. Etingof proposed another way to complete the proof of the above theorem.
\begin{enumerate}
\item Show that if $[L_1, L_2] = 0$, then $L_2$ has constant leading coefficient: $L_2 = \tfrac{d^m}{dx^m} + \dots$
\item Show that $\dim\{L_2: [L_1, L_2] = 0\} \le m + 1$.
\item Compute the growth of the dimension of the graded pieces of the polynomial ring in two variables.
\item Show that the centralizer of $L_1$ is finitely generated.
\end{enumerate}
\end{exercise}

\begin{remark}
While the above construction of a (spectral) curve from a pair of commuting differential operators will mostly suffice for our limited purpose in these lectures, except right at the end, there is a more geometric and general viewpoint on spectral curves, developed by Mulase and others, with multiple further developments solidifying the theory. Indeed, one can consider the space of isospectral deformations of a (Lax) differential operator, i.e.~those deformations that preserve all of its eigenvalues. Then there is a commutative ring of differential operators that determine such deformations, and the spectrum of this ring turns out to be an algebraic curve. This demonstrates the fundamental importance of Sato's infinite-dimensional Grassmannian, and the importance of the KP hierarchy with an infinite collection of time parameters --- while of course in algebraic situations eventually a finite number of differential equations suffices. We further discuss some ideas of this viewpoint in \Cref{rmk:KPgeometry}, \Cref{rmk:Sato}, and \Cref{rmk:UGM} once we introduce the relevant notions, and refer to \cite{mulase} for a detailed survey of the foundations and methods of this theory.
\end{remark}

\subsubsection*{Summary of this section} In principle, everything we have done so far is purely formal, but somehow an algebraic curve defined by the polynomial equation $Q(\alpha, \beta) = 0$ suddenly appears at the end of this discussion. A priori this is an algebraic curve in~$\CC^2$ that can be arbitrarily singular. In algebraic geometry, we would like to further compactify this curve in $\CC\PP^2$ or in $\CC\PP^1 \times \CC\PP^1$. The singularities we will mostly ignore for the lectures, even though they are a major issue that one needs to deal with in many contexts. As an algebraic curve has magically appeared in our formal discussion, we will now switch from differential operators and integrable systems to the other side of the story, geometry of algebraic curves. Eventually these two stories will come together again!

\section{Curves and their Jacobians}\label{sec:Jac}
In this lecture we define the Jacobians of curves analytically and algebraically, and lay out the foundational results on the geometry of the theta divisor, which we will need in what follows.

\subsection{Jacobians of curves --- the analytic definition}\label{sec:curveandJacobian}

Let $C$ be a compact complex algebraic curve of genus~$g$. We will define the Jacobian of the curve $C$, denoted~$\Jac(C)$, in several different ways. First we note that the fundamental group $\pi_1(C)$ of $C$ has (and we can choose) generators $A_1, \dots, A_g$, $B_1, \dots, B_g$ such that 
\[\pi_1(C) = \langle A_1, \dots, B_g \rangle \slash \prod \limits_{i = 1}^g [A_i, B_i] = 1\,.\]
The homology of the curve $C$ with integral coefficients is just the quotient of the fundamental group by its commutator subgroup, and is isomorphic to $\ZZ^{2g}$ :
\[H_1(C, \ZZ) = \pi_1(C) \slash [\pi_1(C), \pi_1(C)] \cong \ZZ^{2g}\,.\]
We have a symplectic pairing on $H_1$ by computing the intersection number of two real cycles on a Riemann surface; explicitly it is defined on the generators by
\[\langle A_i, A_j\rangle = 0 = \langle B_i, B_j\rangle, \qquad \langle A_i, B_j \rangle = \delta_{ij}\,.\] 

Our first definition of the Jacobian is analytic in nature. 
\begin{definition}[Analytic Jacobian] The Jacobian of a curve is
\[\Jac(C) \coloneqq H^{1, 0}(C, \CC) \slash H_1^*(C, \ZZ)\,.\]
\end{definition}
So to construct the Jacobian of a curve $C$, we take the vector space of holomorphic one-forms on $C$ and then take the quotient by the dual of the integral homology lattice. What does this mean? In more explicit terms this means the following: given $A_1, \dots, A_g, B_1, \dots, B_g$, there exists a unique basis $\omega_1, \dots, \omega_g$ of the space of holomorphic one-forms on $C$, such that the integral of $\omega_i$ over $A_j$ is equal to the Kr\"onecker delta-symbol $\delta_{ij}$:
\begin{equation}\label{eq:normalizeddifferential}
    \int_{A_j}\omega_i = \delta_{ij}\,.
\end{equation}
Then the {\em period matrix} $\tau$ of $C$ is the matrix defined by
\[\tau_{ij} \coloneqq \int_{B_j}w_i \in \CC\,,\] 
and the Jacobian is $\Jac(C)=\CC^g/(\ZZ^g+\tau\ZZ^g)$.

The~$g\times g$ complex matrix~$\tau\in\Mat_{g\times g}(\CC)$ satisfies the following \footnote{We call `facts' statements that are essential and fundamental for our exposition, but whose proofs we do not even indicate.} 
\begin{fact}[Riemann's bilinear relations]
The period matrix is symmetric: $\tau=\tau^t$, and its imaginary part (which is a symmetric real $g \times g$ matrix) is positive definite, which we write as $\Im\tau>0$.
\end{fact}

\subsection{Complex tori and abelian varieties}
The Jacobian of an algebraic curve is a principally polarized abelian variety, and we now introduce these in general.

\begin{definition} \label{def:complextorus}
 A complex torus is a quotient $\CC^g \slash \Lambda$, where $\Lambda\subset\CC^g$ is a lattice, that is an abelian subgroup isomorphic to $\ZZ^{2g}$, which has full rank, that is $\Lambda \otimes_{\ZZ} \RR = \CC^g$.
\end{definition}

We can always choose coordinates on~$\CC$ so that the first $g$ generators of $\Lambda$ are the coordinate vectors in $\CC^g$. Then we can write $\Lambda = \ZZ^g + \tau\ZZ^g$ where $\ZZ^g$ is the sublattice generated by the first $g$ generators, and $\tau\ZZ^g$ is what the other $g$ generators of~$\Lambda$ generate. A holomorphic map from one $g$-dimensional complex torus to another can be lifted to their universal covers, to yield a holomorphic map from $\CC^g$ to $\CC^g$. This lifting must have linear growth, and must thus be a linear map. The situation in dimension 1 is completely explicit, and is described by the following

\begin{exercise}\label{ex:g=1}(see the appendix for a hint).
For $g = 1$, which is the classical case when complex tori are elliptic curves $\CC \slash \ZZ +\tau \ZZ$:
\begin{enumerate}
\item When are two elliptic curves biholomorphic?
\item Which complex tori of dimension $1$ have nontrivial automorphisms?
\item Describe the group of automorphisms or endomorphisms of each complex torus. 
\end{enumerate}
\end{exercise}

\begin{definition}
  A complex torus $A$ is called an {\em abelian variety} if it is a projective variety, that is if there exists an embedding~$A\hookrightarrow\CC\PP^N$ of it into the complex projective space.
\end{definition}

What is so special about such an embedding? From the point of view of differential geometry, $\CC\PP^N$ has a Fubini-Study metric, which has positive curvature. From the point of view of algebraic geometry, $\CC\PP^N$ has an ample bundle $O_{\CC\PP^N}(1)$. For any subvariety $A \subseteq \CC\PP^N$ we can take this Fubini-Study metric or ample bundle $O_{\CC\PP^N}(1)$  and restrict it to $A$, obtaining a positive/ample line bundle on~$A$, which we will denote by~$L$.

When speaking about morphisms of abelian varieties, we simply mean holomorphic maps, which are automatically algebraic by the GAGA principle. 
\begin{definition}\label{def:ppav}
A complex principally polarized abelian variety $(A, \Theta)$ is an abelian variety $A$ together with the first Chern class $\Theta\coloneqq c_1(L)$ of an ample line bundle $L$ on~$A$ such that dim $H^0(A, L) = 1$.  
\end{definition}
In general, a polarization on an algebraic variety~$X$ is the first Chern class of an ample line bundle~$L$ on $X$. This class is positive from the point of view of differential geometry, and integral, as the class of a line bundle:
\[c_1(L) \in H^{1, 1}_{> 0}(X, \CC) \cap H^2(X, \ZZ)\,.\]
A polarization is called principal if the dimension of the space of sections is precisely $1$: dim $H^0(X, L) = 1$. 

\begin{exercise}
Check that the intersection pairing on the homology of a curve $H_1(C,\ZZ)=\langle A_1,\dots,A_g,B_1,\dots,B_g\rangle$ defines a principal polarization on its Jacobian. 
\end{exercise}

We denote by $\cA_g$ the moduli space of principally polarized abelian varieties, that is the set of their isomorphism classes. Note that some care is needed here, as every abelian variety has an infinite number of automorphisms. In fact every complex torus acts on itself by translations: for any fixed point $a\in A=\CC^g/\Lambda$, we have the map $t_a:A\to A$ given by $t_a(z)\coloneqq a+z$. To define good moduli spaces, we should exclude such maps, and thus should only consider maps of abelian varieties fixing the origin. By an isomorphism of principally polarized abelian varieties, we further mean that the polarization pulls back to the polarization.

There is a theorem that characterizes abelian varieties among all complex tori.
\begin{theorem}[Riemann's condition]
$\CC^g \slash \ZZ^g + \tau\ZZ^g$ is an abelian variety if and only if $\tau$ is symmetric and the imaginary part of $\tau$ is positive definite.
\end{theorem}
Denoting 
\[
\HH_g\coloneqq\{\tau\in\Mat_{g\times g}(\CC)\,:\,\tau=\tau^t,\ \Im\tau>0\}
\]
the {\em Siegel space} consisting of all symmetric complex matrices with positive-definite imaginary part, it follows that there is a surjective map $\HH_g\to\cA_g$. In fact, similarly to \Cref{ex:g=1}, one can prove that $\cA_g=\HH_g/\Sp(2g,\ZZ)$ is the quotient by the integral symplectic group. Given $\tau\in\HH_g$, what is then the polarization on $\CC^g/\ZZ^g+\tau\ZZ^g$? It can in fact be explicitly given by the Riemann theta function, which is the main player in our lectures, and which we now define.
\begin{definition}\label{def:theta}
For any $\tau \in \HH_g$, and for any $z \in \CC^g$, we let
\[\theta(\tau, z) \coloneqq \sum \limits_{n \in \ZZ^g} \exp(\pi i n^t (\tau n + 2z))\,.\]
\end{definition}
The fact that $\Im\tau$ is positive definite makes this sum convergent, as morally we can think of $\tau$ as being something like $i$, so that we are summing something like $\exp(-\pi n^2)$. Of course the actual proof is more complicated. 

Here is an easy exercise to familiarize ourselves with this formula and to see if we made any of the customary typos in the order of matrix multiplication.
\begin{exercise}\label{ex:thetamodularity}
The theta function is quasi-periodic: for any $m_1, m_2 \in \ZZ^g$, 
\[\theta(\tau, z + m_1 + \tau m_2) = \exp \left(\pi i (-2m_2^t z -m_2^t \tau m_2 )\right)\theta(\tau, z)\,.\]
\end{exercise}
As the exponential factor is nowhere zero, this means that the zero locus of the theta function, for fixed~$\tau$, as a function of~$z$, is invariant under translation by the lattice, and we obtain

\begin{corollary}
The locus
\[\Theta_{\tau} \coloneqq \{z \in A_{\tau}: \theta(\tau, z) = 0\}\]
is well-defined as a subvariety of the abelian variety $A_{\tau} = \CC^g \slash (\ZZ^g + \tau \ZZ^g)$. It is called the {\em theta divisor}.
\end{corollary}

Here is a fact which we will need in the following discussion. 
\begin{fact}
$\Theta_{\tau}$ defines a principal polarization on $A_{\tau}$.
\end{fact}
The locus $\Theta_{\tau}$ is a complex codimension $1$ subset of $A_{\tau}$, and thus its cohomology class lies in $H^{1, 1}(A_\tau,\CC)\cap H^2(A_\tau,\ZZ)$. The above fact is the claim that it's an ample class and that the dimension of the space of sections is one. By abuse of notation, in the following when we talk about the principal polarization, we will mean one of the following three things: an ample line bundle defined up to translation, its Chern class, and its section (unique up to scaling) --- and we will write $\Theta$ for all of these.

\subsection{The Schottky problem}
We now want to formalize the setup for the Torelli map and the Schottky problem. We denote by $\cM_g$ the moduli space of genus $g$ curves, up to isomorphisms. One miracle in the theory is that this `parameter space of algebraic varieties' itself has a complex structure. More precisely, $\cM_g$ is really a Deligne-Mumford stack, or an orbifold, which is to say that locally it is modelled on a quotient of a complex manifold by a finite group action, where the group may act with stabilizers. For our purposes, the stackiness will not play an important role, as we will be mostly thinking about characterizing the locus of Jacobians as a set, i.e.~as a coarse space underlying the stack. The study of~$\cM_g$ goes back at least to Riemann \cite{Riemann1857b}, who by considering an arbitrary genus~$g$ curve as a branched cover of $\CC\PP^1$ and applying essentially the Riemann-Roch theorem, computed that for $g>1$
\[\dim_{\CC}\cM_g = 3g - 3\,.\]
Repeating Riemann's dimension count is a nice exercise, and we encourage you to do this, and especially to think about what kind of general position assumptions are used to prove this. However, unless one already somehow knows that~$\cM_g$ is an orbifold, a priori it does not even make sense to ask about the dimension.

The map that sends a curve to its Jacobian is called the Torelli map
\[ J:\cM_g\to\cA_g\]
from the moduli of curves to the moduli of (principally polarized) abelian varieties, which we recall is $\cA_g =  \HH_g/\Sp(2g, \ZZ).$ The moduli space of abelian varieties is in fact also a (Deligne-Mumford) stack --- though to see this, it is not enough to think about its universal cover $\HH_g$, as we need a local representation as a quotient of a manifold by a finite group. As the group $\Sp(2g,\ZZ)$ is discrete, it follows (once we know that $\cA_g$ is an orbifold, and thus dimension makes sense) that $\dim_\CC\cA_g=\dim_\CC\HH_g=\tfrac{g(g+1)}{2}$, simply because this is the dimension of the space of complex symmetric $g\times g$ matrices, and the positive-definitiveness of the imaginary part is an open condition.

\begin{theorem}[Torelli's theorem]
The map that sends a curve~$C$ to its Jacobian $\Jac(C)$ is an embedding
\[ J:\cM_g \hookrightarrow\cA_g\,.\]
\end{theorem}
\begin{remark}
This is a nice theorem except it is false. To properly work with this map, we should consider $\cM_g$ and $\cA_g$ as stacks. Any abelian variety~$A$ is an abelian group, and thus has an automorphism $a\mapsto -a$ preserving the zero. One can check that the principal polarization (the Chern class of the line bundle) is always preserved by this involution, and thus every point of $\cA_g$ has at least $\ZZ/2\ZZ$ worth of automorphisms. In fact the automorphisms of a general abelian variety are precisely $\ZZ/2\ZZ$, and thus from the point of view of orbifolds a generic point $(A,\Theta)\in\cA_g$ is really half of a point, as $(A,\Theta)$ has one nontrivial involution. On the other hand, a general curve of genus $g\ge 3$ has no automorphisms, and is just an actual point of~$\cM_g$, and not half of one. One can check (exercise once the Abel-Jacobi map is discussed below) that the only Jacobians of curves where the sign involution is inherited from an involution of the curve are those of hyperelliptic curves (that is, double covers of $\CC\PP^1$). Thus the correct statement at the level of stacks is that for $g\ge 3$ the Torelli map is 2-to-1 onto its image, branched along the hyperelliptic locus.
\end{remark}

There are many many proofs of Torelli's theorem (see for example \cite{Weil1957, Andreotti1958, Matsusaka1958, Martens1963, Andreotti-Mayer1967}). Torelli's theorem is one of the cornerstones of the moduli theory and many classification problems. It is also the first instance of Hodge theory because Jacobian is the classifying space for weight one Hodge structures, and this captures the geometry of the curve.

While for $g\ge 4$ one checks that $\dim\cM_g=3g-3<\dim\cA_g=\tfrac{g(g+1)}{2}$, for $g=2,3$ the dimensions match.
\begin{fact}
The Torelli maps $J:\cM_2\to\cA_2$ and $J:\cM_3\to\cA_3$ are dominant.
\end{fact}
Dominant does not mean that these maps are isomorphisms, but means their images are dense. More precisely, the complements of the images consist of those abelian varieties that are products of lower-dimensional ones. It is an interesting exercise to check that the Jacobian of a smooth curve cannot be a product of lower-dimensional abelian varieties (crucially, as principally polarized abelian varieties, taking the product of polarizations --- there is otherwise a whole field of study of Jacobians that decompose up to isogenies, see eg \cite{Wirtinger1895, Chevalley-Weil1934, Ekedahl-Serre1993, Lange-Recillas2004, Earle2006, Moonen-Oort2013, Paulhus-Rojas2017, Lombardo-Lorenzo-Ritzenthaler-Sijsling2023}).

Thus the Torelli maps are well-understood for $g\le 3$, but for $g\ge 4$ the image $J(\cM_g)$ cannot be dense in~$\cA_g$. We denote by $\cJ_g\subseteq\cA_g$ the closure of $J(\cM_g)$, and call it {\em the Jacobian locus}. The Schottky problem asks to determine $\cJ_g$ as a subvariety of $\cA_g$.

The classical approach to this problem, going back to Schottky himself, is to solve this problem in terms of theta constants, which we now define. For $\ve, \delta \in \frac{1}{2}(\ZZ \slash 2\ZZ)^g$, and $m \coloneqq \tau \ve + \delta$, we define
\begin{align*}
    \theta\left[\begin{smallmatrix} \ve \\ \delta \end{smallmatrix}\right](\tau, z) & \coloneqq \sum \limits_{n \in \ZZ^g} \exp(\pi i(n + \ve)^t (\tau (n + \ve) + 2 (z + \delta)))\\
    & = \exp (\pi i \ve^t \tau \ve + 2 \pi i \ve^t (z + \delta))\theta(\tau,t_m (z))\,,
\end{align*}
which, up to a simple exponential factor, is thus the theta function translated by a two-torsion point $m$ (a point on $A_\tau$ such that $2m\in\ZZ^g+\tau\ZZ^g$), where we recall that $t_m:A_\tau\to A_\tau$ is the translation by~$m$.

The idea of the following theorem goes back to Schottky in the 1880s, while the actual rigorous proof is due to Igusa in 1980s. 
\begin{theorem}[Igusa~1981~\cite{Igusa1981}]\label{thm:Schottky}
    $A_{\tau}$ (for $\tau \in \HH_4$) is a Jacobian ($A_{\tau} \in \cJ_4$) if and only if
    \begin{equation}\label{eq:Schottkyg4}
        2^4 \sum \limits_{\ve, \delta \in (\ZZ \slash 2\ZZ)^4} \theta^{16}\left[\begin{smallmatrix} \ve \\ \delta \end{smallmatrix}\right](\tau, 0) = \left(\sum \limits_{\ve, \delta \in (\ZZ \slash 2 \ZZ)^4} \theta^8\left[\begin{smallmatrix} \ve \\ \delta \end{smallmatrix}\right] (\tau, 0)\right)^2\,.
    \end{equation}
\end{theorem}
The above expressions $ \theta\left[\begin{smallmatrix} \ve \\ \delta \end{smallmatrix}\right](\tau, 0)$, i.e.~the values of theta functions with characteristics at $z=0$ are called theta constants; up to some exponential factors, they are simply the values of $\theta(\tau,z)$ at $z=m$ being the corresponding two-torsion points. The simple expression above is due to Igusa; Schottky in 1888 \cite{Schottky1888} wrote a different polynomial in theta constants that also characterizes the locus of genus 4 Jacobians.

This result is amazing in that it relates two very different structures. It describes the image of the moduli of curves of genus $4$ within the moduli space of abelian varieties of dimension $4$ very explicitly, by an equation that one can easily remember, and explain in words. The solution is given as a polynomial in theta constants, i.e.~as a Siegel modular form (though we do not discuss the theory here, not even defining what a Siegel modular form is). Despite lots of progress since Schottky's original work, no such explicit solution/characterization of the Jacobian locus within the moduli space of abelian varieties, by polynomials in theta constants, is known for any $g \ge 5$. This does not at all mean that there has not been any progress in the last 140 years. In fact, multiple more geometric approaches and solutions to the Schottky problem have been developed, and the goal of these lectures is precisely to present one solution to the Schottky problem, due to Igor Krichever, via flex lines of the Kummer variety. Before we proceed to explain what this is, let us comment on the developments in this most classical approach to the Schottky problem.

First, note that it is natural to generalize \eqref{eq:Schottkyg4} to arbitrary genus: we take the same sums of 8'th and 16'th powers of theta constants, and multiply by $2^g$ on the left, instead of $2^4$. There was a conjecture of Belavin, Knizhnik, and Morozov \cite{Belavin-Knizhnik1986, BKMP1986, Morozov1987}, and later D'Hoker and Phong \cite{D'Hoker-Phong2005}, formulated in physics terms by saying that the cosmological constants for the $E_8 \times E_8$ and $SO(32)$ bosonic string theories are equal, which is equivalent to say that this generalized equation holds on $\cJ_5$ --- but this turned out not to be true.
\begin{theorem}[Grushevsky-Salvati~Manni~2011~\cite{Grushevsky-Salvati2011, Grushevsky-Salvati2012}]
    The generalization of equation \eqref{eq:Schottkyg4} does {\em not} hold on $\cJ_g$ for any $g \ge 5$.
\end{theorem}
While a generalization of \eqref{eq:Schottkyg4} is thus not one of the defining equations of $\cJ_g\subseteq\cA_g$, there is something we can do, which is motivated by the actual equation Schottky derived in genus 4 (which is not \eqref{eq:Schottkyg4} above). Indeed, the way Schottky arrived at his equation was by investigating what are now known as the Prym varieties, and by obtaining (in a later work with Jung \cite{Schottky-Jung1909}) what are now known as Schottky-Jung proportionalities, relating theta constants of genus~$g$ Jacobians to theta constants of their Pryms, which are $(g-1)$-dimensional abelian varieties. As a result, in genus 4 Schottky obtained a degree 16 polynomial in theta constants which vanishes for Jacobians. The reason the vanishing of that polynomial is equivalent to \eqref{eq:Schottkyg4} is that theta constants of any abelian varieties themselves satisfy many relations: indeed $\dim\cA_g=\tfrac{g(g+1)}{2}$, while there are $2^{2g}-1$ nontrivial 2-torsion points and corresponding theta constants. Though roughly half of the theta constants, corresponding to so-called odd theta characteristics, vanish identically, there are still numerous relations among the even theta constants, and thus different polynomial expressions in them may be identically equal on $\cA_g$. In fact the problem of determining all relations among theta constants on $\cA_4$ still remains open.

From Schottky's original approach one can obtain numerous explicit equations satisfied by theta constants of Jacobians, and one can ask whether the full set of equations thus obtained solves the Schottky problem, i.e.~characterizes $\cJ_g\subseteq\cA_g$. While this question remains completely open for any $g\ge 5$, a weaker result was recently obtained.

\begin{theorem}[Farkas-Grushevsky-Salvati~Manni~2021~\cite{Farkas-Grushevsky-Salvati2021}]
    Let $N \coloneqq \tfrac{g(g + 1)}{2} - (3g-3)=\dim\cA_g-\dim\cM_g$. Then for any $g\ge 5$ there exists an explicit set of polynomials in theta constants $F_1, \dots, F_N$ such that $\cJ_g$ is an irreducible component of the common zero locus $\{F_1 = \dots = F_N = 0\} \subset \cA_g$.
\end{theorem}

This is a so-called weak solution to the Schottky problem, where `weak'  means that we do not characterize the Jacobian locus itself, but only characterize it up to some extra irreducible components. There is a separate long story of weak solutions to the Schottky problem, the first famous one due to Andreotti and Mayer \cite{Andreotti-Mayer1967}, who showed that $\cJ_g$ is an irreducible component of the locus of abelian varieties  such that the dimension of the singular locus of the theta divisor is at least $g-4$, that is $\dim\operatorname{Sing}\Theta_\tau\ge g-4$. 

The equations~$F_i$ in the above weak solution are completely explicit, but they are not simple enough that one can just write them down on a piece of paper from memory. Moreover, the way $F_i$ are obtained involves combinatorial choices, and by making a different combinatorial choice one can obtain still more explicit equations satisfied by theta constants of Jacobians. It is unclear whether eventually one eliminates all extraneous irreducible components of the common zero set, to obtain a full solution to the Schottky problem in this way.

\smallskip
The interlude on the classical approach to the Schottky problem is over. So remember the goal of these lectures is to present Krichever's solution to the (strong, not weak!) Schottky problem, in terms of flexes of Kummer varieties. For that we first need to understand what the theorem is geometrically. So again we have the moduli space of abelian varieties, which is an arithmetically defined object, and we have the moduli space of algebraic curves, which is geometric. The geometric data on the abelian variety is really its polarization divisor (unique, up to translation, for our case of principally polarized abelian varieties) --- the abelian variety itself locally just looks like~$\CC^g$, with no intrinsic geometry. We will now discuss the algebraic definition of the Jacobians of curves and the geometry of their theta divisors. We will then proceed to Riemann's theta singularity theorem describing the structure of the theta divisor for Jacobians, which will then lead to the Kummer varieties and the existence of their trisecants for Jacobians.

\subsection{Jacobians of curves --- the algebraic definition}
We first discuss divisors on a compact complex genus $g$ algebraic curve~$C$. A {\em divisor} on $C$ is just a formal finite linear combination $\sum \limits_{i = 1}^N m_ip_i$ where $m_i \in \ZZ, p_i \in C$. We denote  $\Div(C)$ the set of all divisors on $C$; it naturally forms an abelian group just by formally adding such expressions. There is a degree map from divisors to integers that just takes the sum of all coefficients:
\[\deg: \left(\sum \limits m_ip_i\right) \in\Div(C)\longmapsto \left(\sum m_i\right) \in \ZZ\,.\]
For a nonconstant meromorphic function~$f$ on $C$, that is for $f: C \to \CC\PP^1$, the divisor of $f$ is defined as
\begin{equation}\label{eq:defdiv}
\div(f) \coloneqq \sum_{p\in C}\mult_p f\,.
\end{equation}
Here we note that the multiplicity is zero unless~$p$ is a zero or pole of~$f$, so that the sum is in fact finite,
\[\div (f)=\sum_{i=1}^N m_ix_i\,,\] 
where the~$x_1,\dots,x_N$ are all the zeroes and poles of~$f$.

The divisors of functions are called {\em principal} divisors, and they form a subgroup $P\Div(C)\subset\Div(C)$ because
\[\div(f \cdot g) = \div(f) + \div(g)\,.\]
\begin{definition}\label{def:Pic}
The Picard group of $C$ is the quotient of the group of divisors by the subgroup of principal divisors:
\[\Pic(C) \coloneqq \Div(C) /P\Div(C)\,.\]
\end{definition}
\begin{remark}\label{rmk:Dsf1}
$\Div(C) /P\Div(C)$ is usually called the divisor class group and the Picard group is defined as the isomorphism classes of line bundles on $C$. There is a 1-1 correspondence between them as follows: the divisor $D$ on $C$ can be described by giving an open cover $\{U_i\}$ of $C$, and for each $i$ a meromorphic function $f_i$, such that for each $i, j$, $g_{ij}\coloneqq f_i \slash f_j$ is a holomorphic nowhere zero function on $U_i \cap U_j$. The functions $g_{ij}$ then provide a set of well-defined transition functions for a line bundle, as $g_{ij} \cdot g_{jk} \cdot g_{ki} = 1$. This line bundle is denoted by $[D]$ and it is trivial if and only if $D$ is the divisor of a global meromorphic function on $C$. The local data $\{U_i, f_i\}$ also give a meromorphic section $s_0$ of~$[D]$ with $\div(s_0) = D$.
\end{remark}
\begin{remark}\label{rmk:Dsf2}
Let $\mathscr{L}(D)$ be the space of meromorphic functions $f$ on $C$ such that
\[D + \div(f) \ge 0,\]
i.e.~$D + \div(f)$ is effective, i.e.~the coefficients of all points~$p\in C$ in $D+\div (f)$ are non-negative. Denote by $|D|$ the set of all effective divisors linearly equivalent to $D$, i.e.~those that differ from $D$ by a principal divisor. Then there is a 1-1 correspondence between meromorphic functions in $\mathscr{L}(D)$ and holomorphic sections of $[D]$: let $s_0$ be a global holomorphic section of $[D]$ with $\div (s_0) = D$; then multiplication by $s_0$ gives an identification
\[\mathscr{L}(D) \xrightarrow{\otimes s_0} H^0(C, [D]).\]
As $C$ is compact, for any $D' \in |D|$, there exists $f \in \mathscr{L}(D)$ such that
\[D' = D + \div(f)\,,\]
and such $f$ is unique up to multiplying by a nonzero constant. Then we have
\[|D| \cong \PP(\mathscr{L}(D)) \cong \PP(H^0(C, [D])).\]
In particular, we have
\[\text{dim}\,|D| = h^0(C, [D]) - 1.\]
In the following we will use these concepts interchangeably, i.e.~an effective divisor $D' \in |D|$ is the same as a holomorphic section $s \in \PP(H^0(C, \cO([D])))$, and is the same as a meromorphic function $f \in \PP(\mathscr{L}(D))$.
\end{remark}

Recall that the number of zeros of a meromorphic function is always the same as the number of its poles, counted with multiplicity, so the degree of a principal divisor is always $0$, and thus the degree map descends to a well-defined map $\deg:\Pic(C)\to\ZZ$. We denote by $\Pic^d(C) \subset \Pic(C)$ the set of degree $d$ divisors. Crucially, $\Pic^d(C)$ can then be interpreted as the set of degree~$d$ line bundles on~$C$ up to linear equivalence. 

We can now tell you what the Jacobian of $C$ is geometrically. 
\begin{definition}[Algebraic Jacobian]
The Jacobian of~$C$ is the set of linear equivalence classes of degree~$g-1$ line bundles on~$C$, that is
\[\Jac(C) \coloneqq \Pic^{g - 1}(C). \]
\end{definition}
We note that $\Jac(C)$ is isomorphic to $\Pic^d(C)$ for any~$d\in\ZZ$, but such an isomorphism is not canonical. Indeed, any $D \in \Pic^{g - 1 - d}(C)$ induces an isomorphism of $\Pic^d(C)$ to $\Jac(C)$ simply by adding~$D$ to a given divisor; in terms of line bundles, this means that we take a degree~$d$ line bundle on~$C$ and take its tensor product with $[D]$, getting a line bundle of degree~$g-1$.

Why do we then define $\Jac(C)$ as $\Pic^{g-1}(C)$ and not as for example~$\Pic^0(C)$, which would have the advantage of obviously being an abelian group? The point is the principal polarization, which we must define on $\Jac(C)$ to make it into a principally polarized abelian variety, and there is a very natural such polarization in $\Pic^{g-1}(C)$.
\begin{definition}
The principal polarization $\Theta$ on $\Pic^{g - 1}(C)$ is defined to be the locus of effective divisors, i.e.
\[\Theta \coloneqq \{p_1 + \dots + p_{g - 1} :p_i\in C\} \subset \Pic^{g - 1}(C)\,.\]
\end{definition}
Points in $\Theta$ are parameterized by $(g-1)$-tuples of points of $C$, symmetrically, so by definition we have the surjective map $C^{\times (g-1)}/S_{g-1}\to \Theta$. It can be shown that this map is generically finite (we will in fact learn much more than that below in \Cref{thm:Riemannsingularity}), and thus $\dim_{\CC}\Theta = g - 1.$ It is true that the Jacobian we just defined is the same as the Jacobian defined analytically, and in particular it is true that $\dim_\CC\Pic^d(C)=g$. It is a priori unclear why $\Theta\subset\Jac(C)$ defines an ample divisor that has a unique section, and we will accept this fact without proof.

\begin{remark}
Another confusing thing about the Jacobian is the group structure. In the analytic viewpoint, the Jacobian of a curve is a quotient of~$\CC^g$ by an abelian subgroup, and thus automatically is an abelian group. Here we have the obvious addition maps $\Pic^d(C)\times\Pic^{d'}(C)\to\Pic^{d+d'}(C)$. Moreover, as the canonical bundle (the bundle of holomorphic one-forms) is a canonically chosen linear equivalence class $K_C\in\Pic^{2g-2}(C)$, we can canonically identify $\Pic^{2g-2}(C)\simeq\Pic^0(C)$, and thus have the addition map
\[\Jac(C) \times \Jac(C) \to\Pic^{2g-2}(C)\simeq\Pic^0(C)\,\]
which, however, does not make $\Pic^{g-1}(C)$ into a group. To make it into a group one needs to pick an element of $\Pic^{g-1}(C)$; a natural choice is to pick $\eta\in\Pic^{g-1}(C)$ such that $\eta+\eta=K_C$. There are in fact $2^{2g}$ such choices (as many as there are two-torsion points on a~$g$-dimensional abelian variety), and to complete the identification of the analytic and algebraic Jacobian one such~$\eta$ needs to be chosen. This is called the Riemann's constant, and we will suppress its choice from now on, just to confuse you (but don't worry). Note, however, that $D\mapsto K_C-D$ defines a canonical involution on~$\Jac(C)$ that does not depend on any choices.
\end{remark}

Let's recap where we are. We defined the moduli of (principally polarized) abelian varieties, and viewed this in the arithmetic/analytic way as the quotient of the Siegel upper half-space by the symplectic group. In the analytic viewpoint, given a curve, we chose a basis of holomorphic one-forms on it, a basis of homology, and use this to define the period matrix $\tau$ of a curve. This period matrix is in the Siegel upper half-space by Riemann's theorem, and thus we know that the theta function --- for which we have an analytic expression --- gives the principal polarization on the Jacobian of a curve. The following fundamentally important theorem describes geometrically what the theta divisor of the Jacobian of an algebraic curve looks like, in the algebraic viewpoint. This is of essential importance in understanding Jacobians and solving the Schottky problem, as there is simply no similar geometry for general principally polarized abelian varieties. There are many many questions about theta divisors of general abelian varieties that are not approachable because there is no analogue of the following statement. 
\begin{theorem}[Riemann theta singularity theorem]\label{thm:Riemannsingularity}
For any divisor $D \in \Pic^{g - 1}(C) \cong \Jac(C)$, the multiplicity of the theta divisor at the point $D\in\Jac(C)$ is equal to the dimension of the space of sections of~$[D]$ on~$C$:
\begin{equation}\label{eq:Riemannsingularity}
\mult_D\Theta =\dim H^0(C, [D]).
\end{equation}
\end{theorem}
Here we are really using the algebraic viewpoint on the Jacobian, identifying divisors as points of~$\Jac(C)$ with linear equivalence classes of line bundles on~$C$. The theta divisor $\Theta$ is a divisor in $\Pic^{g - 1}(C)$, which is the zero locus of the theta function, and thus we can ask what the multiplicity of this divisor (or this function) is at the point $D\in\Pic^{g-1}(C)$. The theorem says that this multiplicity is equal to the dimension of the space of sections of~$[D]$ on the curve~$C$, thus relating the geometry of the curve to the geometry of its Jacobian.

Note also that $\mult_D\Theta>0$ is equivalent to saying $D\in\Theta$, which by \Cref{thm:Riemannsingularity} is equivalent to saying that~$[D]$ has a nonzero section, let's denote it by~$s$. But then $\div(s)$ is by definition the set of points of~$C$ where $s$ vanishes, and since $[D]$ has degree~$g-1$, this means that~$s$ has~$g-1$ zeroes on~$C$, counted with multiplicity. But this says that $\div(s)=p_1+\dots+p_{g-1}$ for some $p_i\in C$ (which may coincide), which is precisely our definition of the theta divisor in~$\Pic^{g-1}(C)$ to start with.

\begin{remark}
We defined the theta function in two different ways. Given any $\tau\in\HH_g$ in the Siegel upper half-space, we defined $\theta$ by the exponential sum in \Cref{def:theta}. Now for Jacobians of curves we have defined the theta divisor geometrically. The claim that for Jacobians of curves the analytic formula defines the same divisor as the geometric construction is far from obvious, and is miraculous in that it relates two completely different worlds. 
\end{remark}
\begin{remark}
Given any $\tau\in\HH_g$, we have defined explicitly by a formula the principal polarization $\Theta_\tau$ on $A_\tau=\CC^g/\ZZ^g+\tau\ZZ^g$. In fact one can also explicitly define other suitable automorphic forms that will define {\em non}-principal polarizations on~$A_\tau$, and these formulas will also depend holomorphically on $\tau\in\HH_g$. However, to understand the isomorphisms of non-principally polarized abelian varieties constructed this way one would then need to quotient~$\HH_g$ not by $\Sp(2g, \ZZ)$, but by another group. As our main interest is Jacobians of curves, non-principally polarized abelian varieties will never again be mentioned in these lectures, and we will keep saying `abelian variety', while we always mean abelian variety with a principal polarization.
\end{remark}

Let us familiarize ourselves with the Riemann theta singularity theorem by the following exercise.

\begin{exercise}\label{ex:3}(see the appendix for a hint)
\begin{enumerate}
\item For $g = 1$, observe that $C \cong \Jac(C)$. What does Riemann theta singularity theorem say in this case? Is it true that $\cM_1 = \cA_1$? This question admits both answers yes and no depending on how one thinks about it. Both are correct. 

\item For $g = 2$, how is $C$ related to Jac$(C)$? And what does the Riemann theta singularity theorem say then? There is a nice uniform answer to this question, as indeed all curves of genus $2$ look the same from the point of view of the geometry of their theta divisors.

\item For $g = 3$, what does the Riemann theta singularity theorem say? What are the possible singularities of theta divisors of genus 3 Jacobians? What do the possible singular points look like? Notice that not all curves of genus~$3$ look the same anymore, from the point of view of the geometry of theta divisors of their Jacobians.
\end{enumerate}
\end{exercise}
\begin{remark}
As discussed before, for $g\ge 4$ not every abelian variety is a Jacobian for dimension reasons, as $\dim\cM_g<\dim\cA_g$. Still, one could describe rather explicitly theta divisors of Jacobians, and in particular (this is a nice exercise) one could show that $\mult_D\Theta\le\tfrac{g+1}{2}$ for any $D\in\Pic^{g-1}(C)$. As we will see, theta divisors of Jacobians are in general very special and singular, and one can conjecture \cite{Smith-Varley2004, Casalaina-Martin2008, Casalaina-Martin2009,Grushevsky-Hulek2013} that this bound for the multiplicity of theta divisors holds for all indecomposable \footnote{A principally polarized abelian variety is called indecomposable if it is not a product of lower-dimensional principally polarized abelian varieties. For such products (called decomposable) the multiplicity of the theta function is the sum of the multiplicities on the factors, and decomposable abelian varieties need to be excluded from many geometric discussions.} $A_\tau\in\cA_g$. This conjecture is proven for $g\le 5$ in \cite{Casalaina-Martin-Friedman2005}, where the Prym construction provides a geometric description of all abelian varieties, but remains completely open in general, showing how little we understand the geometry of theta divisors of arbitrary abelian varieties.
\end{remark}

\subsection{Theta functions on Jacobians}

Recall that we have defined $\Theta \subset \Pic^{g - 1}(C)$ as a divisor in $\Pic^{g - 1}(C)$ that is the set of all elements in the form
$\Theta = \{p_1 + \dots + p_{g - 1}: p_i\in C\}$. Now we want to try to think of the theta function defining this divisor as a function on the curve~$C$ itself. This requires thinking about the curve as sitting inside its Jacobian, and this is how we do it. Observe that there is a tautological map (called the Abel-Jacobi map) $C\to\Pic^1(C)$, which simply sends the point $p \in C$ to the point $p \in \Pic^1(C)$ (here and throughout, we are increasingly freely using the identification of divisors and line bundles on a curve). The Jacobian is of course~$\Pic^{g-1}(C)$, and to map the curve~$C$ in it we will need to non-canonically identify it with $\Pic^1(C)$ --- or rather, with~$\Pic^{-1}(C)$ for convenience. For this we will need a degree~$g$ divisor. By the Riemann-Roch theorem, every degree~$g$ line bundle on a genus~$g$ curve has a holomorphic section, and we can thus use the divisor of zeroes of such a section.

\begin{definition}
Given a $g$-tuple of not necessarily distinct points $p_1,\dots,p_g\in C$, the associated Abel-Jacobi map is defined as
\begin{equation}\label{eq:Abel-Jacobimap}
    \begin{array}{rcl}
    \AJ_{p_1,\dots,p_g}: C & \hookrightarrow & \Pic^{g-1}(C) \\
     p &  \mapsto & p_1+\dots+p_g-p,
\end{array}
\end{equation}
\end{definition}
Often the choice of the points~$p_i$ will be understood, and we will then drop it from the notation. 

\begin{exercise}
Show that analytically the Abel-Jacobi map is defined as
\[\AJ: z \in C \mapsto \left(\int_{p_0}^z \omega_1, \dots, \int_{p_0}^z \omega_g\right)^t\,,\]
where $(\omega_1, \dots, \omega_g)$ is a normalized basis of holomorphic differentials \eqref{eq:normalizeddifferential} on $C$, and relate the starting point of integration~$p_0$ to the divisor $p_1+\dots+p_g$.
\end{exercise}
Importantly, from the analytic definition it immediately follows that
\begin{equation}\label{eq:AJdifferential}
    \frac{\partial \AJ(z)}{\partial z} = \left(\omega_1, \dots, \omega_g\right)^t.
\end{equation}
More generally, for $D = \sum_i p_i - \sum_i q_i \in \Pic^0(C)$, we have
\[\AJ(\sum p_i - \sum q_i) = \left(\sum_i \int_{q_i}^{p_i}\omega_1, \dots, \sum_i \int_{q_i}^{p_i} \omega_g\right)^t.\]

We can now easily think of the theta function~$\theta$, which is a section of the principal polarization on the Jacobian, as a section of a line bundle --- still called the theta function --- on the curve, by pulling it back under the Abel-Jacobi map, so that we consider $\theta\circ \AJ_{p_1,\dots,p_g}(p)$ as a function of~$p\in C$.

Note that by the algebraic definition of the theta divisor we have $\theta\circ \AJ_{p_1,\dots,p_g}(p_i)=0$ for any $1\le i\le g$; indeed
$$
  \AJ_{p_1,\dots,p_g}(p_i)=p_1+\dots+\hat p_i+\dots+p_g\in\Pic^{g-1}
$$
is manifestly the sum of~$g-1$ points of~$C$, and thus lies on the theta divisor. It turns out that these are the only zeroes.
\begin{fact}\label{fact:zerotheta}
$\theta\circ \AJ_{p_1,\dots,p_g}(p)=\theta(\tau,p_1 + p_2 + \cdots + p_g - p)$ is zero if and only if $p = p_i$ for some $1\le i\le g$.  
\end{fact}
To prove this fact, we need to show that there are no further zeroes. This can be done by checking that the degree of the function $\theta\circ \AJ$ on the curve~$C$ is~$g$, i.e.~that it is a section of some degree $g$ bundle on~$C$. We are not going to prove this in general here, but the following exercise for $g = 1$ is easy. 
\begin{exercise}
    Show that for any $\tau \in \HH_g, z \in \CC^g$, we have $\theta(\tau,z) = \theta(\tau, -z)$. What is $\Theta$ in genus $1$? (Can you see modularity of $\theta(\tau, z)$ with respect to $\tau$ for $g = 1$.)
\end{exercise}
Why are the theta functions on the curve interesting and useful? For example, it turns out that since we understand their zero loci so explicitly, we can write down {\em any} meromorphic function on~$C$ as a rational expression in theta functions. To do this, suppose $f: C \to \CC\PP^1$ is a meromorphic function, with its divisor defined in \eqref{eq:defdiv}. As discussed before, the divisors of functions have degree zero, as the numbers of zeroes and poles, counted with multiplicity, are the same (being the degree of the map~$f$ over $0$ and $\infty\in\CC\PP^1$ respectively); this is to say $\sum m_i=0$. 

\begin{proposition}\label{thm:functionsintheta}
    For any choice of general points $p_1, \dots, p_{g-1} \in C$, for a function~$f:C\to\CC\PP^1$ with $\div(f)=\sum m_i x_i$, we have
    \begin{equation}\label{eq:ftheta}
        f(x) = \const\cdot \prod \limits_{i = 1}^N \theta^{m_i}(p_1 + \dots + p_{g - 1} + x_i - x)\,.
    \end{equation}
\end{proposition}
Here by general we mean more precisely that $\dim H^0(C, p_ 1 + \dots + p_{g - 1}) = 1$, i.e.~that $p_1+\dots+p_{g-1}$ is a smooth point of the theta divisor. 

\begin{proof}
    Note that
    \[\div(\theta^{m_i}(p_1 + \dots + p_{g-1} + x_i - x)) = m_i \cdot (p_1 + \dots + p_{g-1} + x_i)\,,\]
    by \Cref{fact:zerotheta}, and thus the divisor of the right-hand-side of \eqref{eq:ftheta} is equal to
    \[\sum \limits_{i = 1}^N m_i(p_1 + \dots + p_{g-1}) + \sum \limits_{i = 1}^N m_ix_i = \left(\sum \limits_{i = 1}^N m_i\right) (p_1 + \dots + p_{g-1}) + \sum m_i x_i = \sum m_ix_i\,.\]
    Thus the product $\prod \limits_{i = 1}^N \theta^{m_i}(p_1 + \dots + p_{g - 1} + x_i - x)$ has the same divisor of zeroes and poles as the function $f$, and thus their ratio is a meromorphic function on~$C$ that has no zeros or poles, and is thus constant.
\end{proof}

This provides us with a way to construct all meromorphic functions on Riemann surfaces explicitly by using theta functions. Note, however, that the explicit formula is very strange as the expression visibly depends on the choice of a general point $p_1 + \dots + p_{g - 1}$ on the theta divisor. There are of course many such choices, but somehow~the resulting function~$f$ is independent of any choice made. This is to say that the ambiguities in the choices of~$p_i$ cancel. For your own benefit, try to see how this relates to Abel's theorem. \footnote{Abel's theorem is that a degree zero divisor $\sum m_ix_i$ on a curve~$C$ is principal if and only if its image $\sum m_i x_i\in \Pic^0(C)$ is equal to zero, which is obvious from the algebraic definition but not from the analytic one; we do not pursue this further here.}

\subsubsection*{Summary of this section} We have reviewed the analytic and algebro-geometric constructions of the Jacobian of a complex curve, and related them to each other. The Schottky problem is the problem of characterizing Jacobians among all principally polarized abelian varieties, and the Riemann theta singularity theorem provides rich information on the geometry of the theta divisor of a Jacobian, which we will use below to approach the Schottky problem.

\section{Geometry of Theta Divisors}\label{sec:theta}
We now further investigate the geometry of theta divisors of Jacobians by applying Riemann's theta singularity theorem. Recall that the Riemann theta singularity theorem equates the multiplicity of the theta function at a point $D$ in the Jacobian $\Pic^{g-1}(C)$ to the number of sections of the corresponding line bundle $[D]$. So what this does is it relates the geometry of the curve to the geometry of its Jacobian. There is no analog of this for a general abelian variety, and we will now use this to derive some interesting properties of Jacobians that we will then see are not shared by general abelian varieties.

\subsection{Weil reducibility}
Our first  goal is to deduce from Riemann theta singularity theorem the following Weil reducibility theorem. 
\begin{theorem}[Weil's reducibility~1957~\cite{Weil1957}]\label{thm:Weilreducibility}
For any curve $C$, and for any points $p, q, r, s \in C$ with $p \ne q$
\[\left(\Theta \cap \Theta_{p-q}\right) \subset \left(\Theta_{p-r} \cup \Theta_{s-q}\right)\,.\]
\end{theorem}
In the statement of the theorem, and in what follows, we will write $\Theta_a$ for the translate of the theta divisor by~$a$. More precisely, we recall that the natural addition map of divisors descends to a map $\Pic^0(C)\times\Pic^{g-1}(C)\to\Pic^{g-1}(C)$, i.e.~to an action on the Jacobian. So thinking of $p - q$ as a divisor in $\Pic^0(C)$, $\Theta_{p - q}$ is then this translate of the theta function; analytically, it is the zero locus, in the variable $z\in\Jac(C)$, of $\theta(z+p-q)$.

\begin{remark}
Note that there is something very strange about \Cref{thm:Weilreducibility}. Indeed, both $\Theta$ and $\Theta_{p - q}$ are codimension $1$ subvarieties of the Jacobian. As the Jacobian is indecomposable, they are in fact both irreducible as by definition theta divisor is the image of $C^{\times(g-1)}/S_{g-1}$. Thus their intersection is some union of irreducible codimension~$2$ subvarieties of the Jacobian, which depends on $p$ and $q$. The theorem claims that this intersection is contained in the union of two other codimension~$1$ subvarieties of the Jacobian. However, the points $r$ and $s$ only appear on the right-hand-side of the theorem, in $\Theta_{p-r}$ and $\Theta_{s-q}$! This means that the codimension~$2$ subvariety $\Theta \cap \Theta_{p-q}$, which does not depend on~$r$ and~$s$, is contained in a two-parameter family of unions of two such irreducible codimension~$1$ varieties, obtained by varying~$s$ and~$r$. If we remove either translate, then the conclusion is verifiably false, and indeed what happens generically is that $\Theta\cap\Theta_{p-q}$ has precisely two irreducible components, one contained in~$\Theta_{p-r}$, and the other in~$\Theta_{s-q}$. We encourage you to think about this as you are reading the proof below.
\end{remark}
\begin{proof}
The only technique used is the Riemann theta singularity theorem. Suppose $L\in\Jac(C)$ lies in the intersection (from now on by abuse of language and notation we will interchangeably speak of divisors and line bundles on curves, and when necessary also of their equivalence classes in the Jacobian, see \Cref{rmk:Dsf1}):
\[L \in \Theta \cap \Theta_{p - q}\,.\]
We want to check that $L$ is then contained in one of the two indicated translates of the theta divisor. Now what does it mean that $L$ is contained in the intersection of the theta divisor and its translate?  By the Riemann theta singularity theorem, $L\in\Theta$ means that the corresponding line bundle has at least one section, i.e.~$h^0(C, L) \ge 1$. Similarly, $L\in\Theta_{p-q}$ means that $h^0(C,L+p-q)\ge 1$. Thus what we need to prove is
\[ h^0(C,L)\ge 1\hbox{ and }h^0(C,L+p-q)\ge 1\Rightarrow h^0(C,L+p-r)\ge 1\hbox{ or }h^0(C,L+s-q)\ge 1.\]
It is now useful to think of sections of a line bundle~$L$ as meromorphic functions on the curve with poles of orders bounded below by the corresponding divisor~$D$ (see \Cref{rmk:Dsf2}). In particular, $h^0(C, L +  p - q) \ge 1$ means that there exists a meromorphic function~$f$ on $C$ that may have poles at $D$ (the divisor corresponding to~$L$) and $p$ (counted with multiplicity), and vanishing at~$q$. Since $H^0(C, L + p)$ is similarly the space of meromorphic functions that may have poles at $D$ and $p$, and thinking similarly of $H^0(C,L)$, we have the inclusions
\begin{equation}\label{eq:contain}
H^0(C, D + p - q)\subseteq H^0(C, D + p)\supseteq H^0(C,D).    
\end{equation}
We now investigate separately the two cases $h^0(C, D + p) = 1$ and $h^0(C, D + p) \ge 2$.

{\em Case 1:} $h^0(C, D + p) = 1$. In this case by \eqref{eq:contain} the space $H^0(C,D+p)$ must be simply equal to~$H^0(C,D+p-q)=\CC\cdot f$. However, by the right containment in that equality we must also have $H^0(C,D)=\CC\cdot f$, which is to say that the function~$f$ does not have a pole at~$p$ (or, if the support of the divisor~$D$ contains~$p$, its pole order at~$p$ is still bounded by the multiplicity of~$p$ in~$D)$. But then by definition $f\in H^0(C,D-q)$, as it does not have a pole at~$p$ and we already know it vanishes at~$q$. Similarly to \eqref{eq:contain}, we have $H^0(C,D-q)\subseteq H^0(C,D+s-q)$, and thus $h^0(C,D+s-q)\ge 1$, which by using the Riemann theta singularity in reverse implies $L\in \Theta_{s-q}$.

{\em Case 2:} $h^0(C, D + p)\ge 2$. In this case there exist two linearly independent sections $f_1, f_2 \in H^0(C, D + p)$. But then there exists some linearly combination $f\coloneqq\alpha_1f_1+\alpha_2f_2$, for suitable $\alpha_1,\alpha_2\in\CC$ that vanishes at $r$ (just consider the ratio of the values $-f_1(r)/f_2(r)$). But then since the poles of~$f$ are at worst the same as the poles of~$f_1,f_2$, we see that $0\ne f\in H^0(C,D+p-r)$, and thus again using the Riemann theta singularity shows that $L\in \Theta_{p-r}$.
\end{proof}

\subsection{Weil reducibility from the analytic viewpoint}\label{sec:analyticWeil}

Let's try to understand what Weil reducibility means analytically in terms of functions. We argue it informally to see how it works but all the statements can be made rigorous.

First, what does it mean that $z$ lies on the theta divisor? It means $\theta(z) = 0$. Similarly, $z \in \Theta_{p - q}$ if and only if $\theta(z + p - q) = 0$. So $z \in \Theta \cap \Theta_{p - q}$ means that $\theta(z) = 0=\theta(z + p - q)$.
Now $z \in \Theta_{p - r} \cup \Theta_{s - q}$ means that either $\theta(z + p - r) = 0$ or $\theta(z + s - q) = 0$ which is of course the same as to say the product $\theta(z + p - r)\theta(z + s - q)$ is equal to zero. So Weil reducibility is equivalent to the following claim:
\begin{equation}\label{eq:analyticWeil}
\theta(z) = 0=\theta(z + p - q) \Longrightarrow \theta(z + p - r)\theta(z + s - q) = 0.
\end{equation}
To further use this, it would be great to derive some stronger or more precise statement. More precisely, we want to write an equation that will imply \eqref{eq:analyticWeil} --- and then we will hope that in fact this equation would be equivalent to \eqref{eq:analyticWeil}. How could this happen? So if two functions $\theta(z)$ and $\theta(z + p-q)$ both vanish, then the function~$\theta(z + p - r)\theta(z + s - q)$ also vanishes. Said differently, this means that the function $\theta(z + p - r)\theta(z + s - q)$ should lie in the ideal generated by $\theta(z)$ and $\theta(z + p - q)$. For this to happen there must exist holomorphic functions $F(z)$ and $G(z)$ such that 
\begin{equation}\label{eq:analyticWeil2}
F(z)\theta(z) + G(z)\theta(z + p - q) = \theta(z + p - r)\theta(z + s - q)
\end{equation}
for all~$z\in\CC^g$ (and this would certainly imply \eqref{eq:analyticWeil}). We are writing these as functions on $\CC^g$, but theta functions are really  sections of some line bundles on the Jacobian. So what is $\theta(z + p - r)\theta(z + s - q)$ a section of? By definition, $\theta(z)$ is a section of the polarization bundle~$\Theta$, so $\theta(z + p - r)$ is a section of $\Theta_{p-r}$, and thus altogether $\theta(z + p - r)\theta(z + s - q)$ is a section of some translate of $2\Theta$ \footnote{From here onwards we use the standard algebraic geometry convention of writing $2\Theta$ for $\Theta^{\otimes 2}$, etc.}. In fact the translates of the two factors add up, and thus altogether $\theta(z + p - r)\theta(z + s - q)$ lies in $H^0(\Jac(C), 2\Theta_{p+s-r-q})$, which means we take the line bundle $\Theta^{\otimes 2}$, and tensor it with $\cO_C(p + s - r - q)$, which is a line bundle in $\Pic^0(C)$ that we also denoted $[p+s-r-q]$ before. 

So the right-hand-side of \eqref{eq:analyticWeil2} is a section of a line bundle, and it is natural to expect each of the two summands on the left to also be sections of the same line bundle. For $G(z)\theta(z + p - q)$ to be a section of $2\Theta_{p + s - r - q}$, $G(z)$ must be a section of $\Theta^{\otimes 2}\otimes \cO_C(p + s - r - q) \otimes \Theta^{-1}\otimes\cO_C(q -p)=\Theta\otimes\cO_C(s-r)$. But the bundle $\Theta_{s-r}$ is a translate of the principal polarization, and thus only has one section $\theta(z + s - r)$, up to a constant factor. By the same reasoning $F(z)$ must be a constant multiple of $\theta(z + p + s - r - q)$, and we thus expect to obtain for some $A,B\in\CC$ the equation
\begin{equation}\label{eq:analyticWeil3}
A\theta(z + p + s - r - q)\theta(z) + B\theta(z + s - r)\theta(z + p - q) = \theta(z + p - r)\theta(z + s - q)\,.
\end{equation}
This obviously implies \eqref{eq:analyticWeil2}. As all the factors now live in the same bundle, \eqref{eq:analyticWeil3} is actually equivalent to \eqref{eq:analyticWeil2}. 

\begin{fact}
Equations \eqref{eq:analyticWeil}, \eqref{eq:analyticWeil2} and \eqref{eq:analyticWeil3} are all equivalent, and equivalent to the Weil reducibility.
\end{fact}
The proof of this fact uses Koszul cohomology and does not fall within the scope of these lectures, see eg.~\cite{Beauville-Debarre1986}, \cite[p.~34]{Fay1973}. Moreover, technically for the result we will prove this equivalence will not be necessary.
\begin{remark}
Equation \eqref{eq:analyticWeil3} is a functional equation for the theta function, and it must hold for some constants $A$ and $B$. The constants $A$ and $B$ will depend on $p, q, r, s$, but the point is that they do not depend on~$z$. It is an equation that must hold for any $z \in \CC^g$, or for any point in the abelian variety, depending on how one thinks about it. 
\end{remark}

Geometrically equation \eqref{eq:analyticWeil3} means that if we have an arbitrary quadruple of points $p,q,r,s\in C$, there is a linear dependence between three sections of $2\Theta_{p+s-r-q}$ on the Jacobian. Let us try to appreciate this result: by construction the bundle $\Theta$ itself has only one section, but how many sections does the bundle $2\Theta$ have? Here is another fact which we will not fully prove.

\begin{fact}
For any principally polarized abelian variety $h^0(A, 2\Theta) = 2^g$. 
\end{fact}

Here is a quick sketch of a proof. As $\Theta$ is an ample bundle, $h^i(A, \Theta)=0$ for any $i>0$. The fact that $h^0(A, \Theta) = 1$ then implies that the top self-intersection number $\Theta^g$ on~$A$ is equal to $g!$. But the top self-intersection number of $2\Theta$ differs from the top self-intersection number of $\Theta$ by a factor of $2^g$. As $h^i(C, 2\Theta)$ still vanishes for all $i>0$, it follows that $h^0(A, 2\Theta)=2^g\cdot h^0(A,\Theta)=2^g$. 

Moreover, still for any principally polarized abelian variety, and not just for Jacobians, we can conveniently write down an explicit basis for $H^0(A, 2\Theta)$, given by the so-called {\em theta functions of the second order}
\begin{equation}\label{eq:Thetavarepsilon}
\Big\lbrace\Theta[\ve](\tau, z) \coloneqq \theta\left[\begin{smallmatrix} \ve \\ 0 \end{smallmatrix}\right](2\tau, 2z)\Big\rbrace \quad \text{for all } \ve \in \tfrac{1}{2}(\ZZ \slash 2\ZZ)^g.
\end{equation}
Here $\Theta[\ve]$ is a notation for a theta function of the second order, which should not be confused with the theta divisor $\Theta$. The fact that these are sections of $2\Theta$ is elementary and easy to check, but the fact that they form a basis for sections of $2\Theta$ requires a proof, which again we omit here.

The next fact is a special case of a much more general statement called the isogeny theorem.
\begin{fact}[Riemann's bilinear relations] For any $\tau\in\HH_g$, and any $x,y\in\CC^g$
\begin{equation}\label{eq:isogeny}
\theta(\tau, x + y)\theta(\tau, x - y) = \sum \limits_{\ve \in \frac{1}{2}(\ZZ \slash 2\ZZ)^g}\Theta[\ve](\tau,x) \cdot \Theta[\ve](\tau,y).
\end{equation}
\end{fact}
One can in fact prove this by a direct computation using the definition of theta functions of second order, but we do not dwell on this here. So why is it called an isogeny theorem? It's called that because if we take an abelian variety $A$ times itself, then there is a very interesting map 
\[\begin{array}{rcl}
\pi: A \times A & \longrightarrow & A \times A\\
(x, y) & \longmapsto & (x + y, x - y).
\end{array}
\]
This is an isogeny of abelian varieties which means it is a map of abelian varieties of finite degree. What is the kernel of this map? The kernel are the points that map to zero in $A \times A$, and these are precisely the points of the form $(x,x)\in A\times A$ such that $2x=0\in A$; such $x$'s are called two-torsion points of~$A$, and are sometimes denoted~$A[2]$. So the kernel of this map are the two-torsion points of $A\times A$ that lie on the diagonal. Note that $A \times A$ has a natural polarization as a box product of principal polarizations: $\Theta \boxtimes \Theta$. The box product means the following: $A \times A$ has two projections $p_1$ and $p_2$ to the two factors of $A$, and the polarization is
\[\
\Theta \boxtimes \Theta : =  p_1^*\Theta \otimes p_2^*\Theta\,.
\]
What if we take the pullback of this polarization on $A\times A$ under $\pi$? The pullback $\pi^*(\Theta \boxtimes \Theta)$ can be easily computed to be  simply $2\Theta \boxtimes 2\Theta$. So how does this relate to  Riemann's bilinear relations? Well, $\theta(\tau,u)\theta(\tau,v)$ is a section of $\Theta\boxtimes \Theta$ on the right (where we use $u,v$ as variables on that copy of $A\times A$), and thus $\theta(\tau,x + y)\theta(\tau,x - y)$ is precisely a section of $\pi^*(\Theta\boxtimes\Theta)$, which is to say it is a section of $2\Theta\boxtimes2\Theta$. But as $\lbrace \Theta[\ve](\tau,x)\rbrace$ is a basis of sections of $2\Theta$ on the first factor on the left, and similarly $\lbrace \Theta[\ve](\tau,y)\rbrace$ is a basis of sections of $2\Theta$ on the second factor, it follows that $\theta(\tau,x+y)\theta(\tau,x-y)$ can be written as some linear combination $\sum \alpha_{\ve_1,\ve_2}\Theta[\ve_1](\tau,x)\Theta[\ve_2](\tau,y)$. Riemann's bilinear relation is then the statement that the coefficients $\alpha_{\ve_1,\ve_2}$ are simply the Kr\"onecker delta symbol.

Now we need another fact to proceed --- the statement that the line bundle $2\Theta$ is very ample on~$A/\pm 1$. This is to say, consider the map $A\to\CC\PP^{2^g - 1}$ given by the linear system $|2\Theta|$, i.e.~we send a point $z \in A$ to the set of values of a basis of sections of $2\Theta$. Let us call this map $\Kum$, and the claim is then

\begin{fact}[Lefschetz Theorem] 
The Kummer map
\[\begin{array}{rcl}
\Kum: A & \xrightarrow{\vert 2\Theta \vert} & \CC\PP^{2^g - 1}\\
z & \longmapsto & \{\Theta[\ve](\tau, z)\}_{\ve \in (\ZZ \slash 2\ZZ)^g},
\end{array}\]
is an embedding of the quotient $A \slash \pm 1$ into $\CC\PP^{2^g - 1}$.
\end{fact}

Note that $\Theta[\ve](\tau, z)$ are all even functions of $z$ by definition, so the values of $\Kum$ at $z$ and $-z$ are certainly the same, and thus the Kummer map is well-defined on the quotient $A/\pm 1=A/z\sim -z$. The claim is that these are the only points which will have the same image under the Kummer map. The full statement of Lefschetz theorem includes the claim that any higher multiple of the principal polarization, for example~$3\Theta$, is already very ample on~$A$, i.e.~defines an embedding of~$A$ itself into a suitable projective space. The fact we state here is that for~$2\Theta$ it is almost an embedding, up to this plus-minus one business. The proof of Lefschetz theorem requires some technology we did not develop, and thus we omit it.
 
{\em Warning}: the quotient $A \slash \pm 1$ is not smooth; indeed the fixed points of the involution sending $x\mapsto -x$ are precisely the points such that $x=-x$, which is to say $2x=0$, so these are the two-torsion points $A[2]$ again, and they will play a special role in some of what follows.

\begin{example}
For $g = 1$, what we get for $A \slash \pm 1$ is an elliptic curve quotient by $\pm 1$, and $\Kum$ is a map from this to $\CC\PP^1$. Well, we know that an elliptic curve is a double cover of $\CC\PP^1$ branched at four points, and the four points are exactly the two-torsion points.

For $g = 2$, the image $\Kum(A/\pm 1)\subset\CC\PP^3$ is the so-called Kummer surface. It is a degree 16 hypersurface (i.e.~given by an equation of degree 16 in theta functions of the second order), which is singular at the 16 points that are images of the two-torsion points $A[2]$. An explicit equation for the Kummer surface can be written, and the moduli $\cA_2$ of abelian surfaces can be studied via this map. This is a very rich theory, first comprehensively explored in the 1905 classical book \cite{Hudson1905}.

For $g = 3$, the geometry becomes very complicated because we have a threefold embedded into~$\CC\PP^7$, and there are many equations for the image. There are numerous recent research papers related to this story.
\end{example}

Now we come back to \eqref{eq:analyticWeil3}, which we rewrite in a more symmetric form
\begin{equation}\label{eq:analyticWeil4}
A\cdot\theta(z)\theta(z + p + r - s - q) + B\cdot\theta(z + p - q)\theta(z + r - s) + C\cdot \theta(z + p - s)\theta(z + r - q) = 0\,,
\end{equation}
with some constants $A,B,C\in\CC$, and to which we now apply Riemann's bilinear relations. To simplify formulas, by a small abuse of notation we will sometimes think of $\Kum(z)$ as a point in $\CC^{2^g}$ rather than $\CC\PP^{2^g-1}$, and we will then use the scalar product notation $\Kum(x) \cdot \Kum(y)$ to mean $\sum \limits_{\ve \in \frac{1}{2}(\ZZ \slash 2\ZZ)^g} \Theta[\ve](x) \cdot \Theta[\ve](y)$. Applying Riemann's bilinear relation \eqref{eq:isogeny} to equation \eqref{eq:analyticWeil4} then yields

\begin{align*}
& \quad A \cdot \Kum\left(z + \tfrac{p + r - s - q}{2}\right) \cdot \Kum\left(\tfrac{p + r - s - q}{2}\right)\\ &+ B\cdot \Kum\left(z + \tfrac{p + r - s - q}{2} \right) \cdot \Kum\left(\tfrac{p + s - r - q}{2}\right) \\
& + C \cdot \Kum\left(z + \tfrac{p + r - s - q}{2}\right) \cdot \Kum\left(\tfrac{p + q - r - s}{2}\right) = 0\,,
\end{align*}
which can be rearranged as
\begin{align*}
&\left(A \cdot \Kum\left(\tfrac{p + r - s - q}{2}\right) + B \cdot \Kum\left(\tfrac{p + s - r - q}{2}\right) + C \cdot \Kum\left(\tfrac{p + q - r - s}{2}\right) \right)\\ &\cdot\Kum\left(z + \tfrac{p + r - s - q}{2}\right) = 0, \quad \forall z\,.
\end{align*}
To make sense of dividing by $2$, we can think of this as an equation with variables in~$\CC^g$, but actually noticing the multiplication by 2 in the definition of the theta functions of the second order, no choice needs to be made.
 
Recall that $\Kum\left(z + \frac{p + r - s - q}{2}\right)$ is a vector made up by the values of a basis of sections of $2\Theta_{p+r-s-q}$. The basis of sections is linearly independent, as functions of $z$, and thus cannot satisfy any linear relation with constant (in~$z$) coefficients. This forces the sum in the parenthesis to be zero, i.e.
\begin{equation}\label{eq:trisecantlinear}
A \cdot \Kum\left(\tfrac{p + r - s - q}{2}\right) + B \cdot \Kum\left(\tfrac{p + s - r - q}{2}\right) + C \cdot \Kum\left(\tfrac{p + q - r - s}{2}\right) = 0.
\end{equation}
This is an equation for three vectors to be linearly dependent in $\CC^{2^g}$, that is for them to lie in a plane. Equivalently, it can be written as
\begin{equation}\label{eq:trisecantwedge}
\Kum\left(\frac{p + q - r - s}{2}\right) \bigwedge \Kum\left(\frac{p + r - s - q}{2}\right) \bigwedge \Kum\left(\frac{p + s - r - q}{2}\right) = 0 
\end{equation}
for any $p, q, r, s \in C$. Projectivizing this equation, it means that the three corresponding points in~$\CC\PP^{2^g - 1}$ must be collinear, that is the line passing through two of them has to go through the third one, forming a so-called {\em trisecant line} of the image $\Kum(\Jac(C))\subset\CC\PP^{2^g-1}$.

What we have just proved is called Fay-Gunning's trisecant identity: 

\begin{theorem}[Trisecant identity, see~\cite{Fay1973, Mumford1975, Gunning1982b, Welters1983}]
For any $C$,  and any $p, q, r, s \in C$, the Kummer images of the points
\[\frac{p + q - r - s}{2}, \quad \frac{p + r - s - q}{2}, \quad \frac{p + q - s - r}{2}\]
are collinear in~$\CC\PP^{2^g-1}$.
\end{theorem}

Let's take a minute to see why Fay's trisecant identity is really a very striking and surprising condition. Indeed, $\Kum(\Jac(C))\subset\CC\PP^{2^g-1}$ is a $g$-dimensional variety. Given any two points on the Kummer variety, consider the line through these two points. The set of points on all such lines for all pairs of points of the Kummer variety is called the secant variety of $\Kum(\Jac(C))$ in~$\CC\PP^{2^g-1}$, and it has expected dimension $2g+1$ ($g$ for each of the two points, and an extra 1 for moving along the line). Then we are asking whether this secant variety intersects the Kummer variety again, i.e.~whether there exists a line secant to the Kummer variety that intersects it in a third point (of course this third point has to be not one of the two original points). But in general for dimension reasons one expects a $k$-dimensional and an $\ell$-dimensional subvariety of $\CC\PP^N$ to intersect when $k+l\ge N$. However, for us, for $g$ large enough $2g+1+g\ll 2^g-1$, and thus having a trisecant is a very rare condition! Even if we varied over all Kummer images of all abelian varieties, this would only add $\tfrac{g(g+1)}{2}$ to the dimension count, and thus for dimension reasons alone one would expect that for~$g$ very large, no Kummer variety ever has a trisecant line. However, the above theorem says that for {\em all} Jacobians of curves and for {\em all} quadruples of points on the curve, the Kummer variety has a trisecant line.

The following wonderful theorem of Gunning claims that the existence of infinitely many of such trisecants characterizes the Kummer images of Jacobians.

\begin{theorem}[Gunning~1982~\cite{Gunning1982b}]\label{thm:Gunning'sthm1}
If a principally polarized abelian variety $(A, \Theta) \in \cA_g$ is such that its Kummer image $\Kum(A)$ has a family of trisecant lines, of the form $\Kum(p+a_1)\wedge \Kum(p+a_2)\wedge \Kum(p+a_3)=0$, for $a_1,a_2,a_3\in A$ fixed distinct points, and for~$p$ varying in a one-parameter family, then~$A$ is the Jacobian of a curve.
\end{theorem}

This theorem provides a solution of the Schottky problem, telling us when a principally polarized abelian variety $(A,\Theta)$ is the Jacobian of a curve~$C$. However, there is a caveat: we need to have a one-dimensional family of trisecants to start with! And indeed~$C$ is not going to be some abstract curve, but will rather turn out to be the locus of~$p\in A$ giving a trisecant as above. So to apply it we need an abelian variety and a curve inside it to start with, so this theorem is really a solution of the restricted Schottky problem. 

The proof of \Cref{thm:Gunning'sthm1} uses the Matsusaka-Hoyt-Ran Criterion \cite{Matsusaka1959, Hoyt1963, Ran1981} characterizing Jacobians of curves by the fact that the so-called minimal homology class $\tfrac{[\Theta]^{g-1}}{(g-1)!}\in H^2(A,\ZZ)$ can be represented by an actual complex curve in~$A$. Matsusaka-Hoyt-Ran Criterion is a beautiful result with a beautiful proof, giving another solution to the restricted Schottky problem. The proof is not so hard in retrospect but we will not state it here as it is in a different direction from where we are going. We thus return to investigating trisecants more closely.

\subsection{Trisecants of Kummer varieties}\label{sec:geometryofKummer}
The trisecant is very geometric, but maybe not quite analytic. If we look at equation \eqref{eq:trisecantwedge}, we are allowed to vary $p, q, r, s$ along the curve, in particular we can take the limit as some of the points approach each other (this was already advocated by Mumford \cite{Mumford1984}; see also \cite{Welters1983}). In the limit this will be a differential equation satisfied by the theta functions, and we now follow through this viewpoint methodically to see where it leads.

First we take the limit as $s\to p$. Just setting $s=p$ in \eqref{eq:trisecantwedge} gives the Kummer image of $\frac{q - r}{2}$ in the first factor, the Kummer image of $\frac{r - q}{2}$ in the second factor, and the Kummer image of $\frac{2p - r - q}{2}$ in the last factor, so we obtain
\begin{equation*}\label{eq:trisecantlimitfalse1}
\Kum\left(\tfrac{q - r}{2}\right) \bigwedge \Kum\left(\tfrac{r - q}{2}\right) \bigwedge \Kum\left(\tfrac{2p - r - q}{2}\right) = 0\,.
\end{equation*}
But the Kummer map is an even function of the argument $z$, that is $\Kum(z)=\Kum(-z)$, and thus the first two vectors are simply equal, so this equation is trivially satisfied. Thus we need to take into account the next order term in the Taylor expansion of the trisecant as $s$ approaches $p$, by taking derivatives: in the first factor we should get $q - r$ plus an infinitesimal $p - s$ term, in the second factor we get $r - q$ plus the same infinitesimal $p - s$ term. Formally, we write $s = p - U \ve$, so that \eqref{eq:trisecantwedge} yields
\begin{equation}\label{eq:trisecantlimitU}
\Kum\left(\frac{q - r + U\ve}{2}\right) \bigwedge \Kum\left(\frac{r - q + U\ve}{2}\right) \bigwedge \Kum\left(\frac{2p - r - q}{2}\right) = 0.
\end{equation}
As we are taking the limit $\ve\to 0$, in the last factor we simply dropped the $-U\ve$ in the argument, to obtain the lowest (constant in $\ve$) order term in the expansion. Expanding this equation in powers of $\ve$, the constant term is \eqref{eq:trisecantlimitfalse1}, which is automatically zero, while the vanishing of the coefficient of $\ve$ implies 
\begin{equation}\label{eq:trisecantlimit1}
\Kum\left(\frac{q - r}{2}\right) \bigwedge \partial_U \Kum\left(\frac{q - r}{2}\right) \bigwedge \Kum\left(\frac{2p - r - q}{2}\right) = 0.
\end{equation}
Note that by \eqref{eq:AJdifferential} $U$ is actually the first derivative of the Abel-Jacobi map AJ$: C \to \Jac(C)$ at the point~$p$:
\[U = (\omega_1(p), \dots, \omega_g(p))^t\,.\]

We now take further limits, denoting $V$ the second  derivative of the Abel-Jacobi map $\AJ:C \to \Jac(C)$ at $p$, and letting $r$ approach $p$ in \eqref{eq:trisecantlimit1}. Then letting $r=p+U\ve+V\ve^2$ yields
\begin{equation}\label{eq:trisecantlimitV}
\Kum\left(\tfrac{q - p - U\ve - V\ve^2}{2}\right) \bigwedge \Kum\left(\tfrac{p - q + U\ve + V \ve^2}{2}\right) \bigwedge \Kum\left(\tfrac{p - q - U\ve - V\ve^2}{2}\right) = 0.
\end{equation}
Expanding this at $\ve=0$, using \eqref{eq:trisecantlimit1}, the lowest order nonzero term is of order $\ve^2$, which yields, when we now set $p=0$ (as we are free to choose the basepoint in the definition of the Abel-Jacobi map) the following equation
\begin{equation}\label{eq:Kummerflex1}
\Kum(\tfrac{q}{2}) \bigwedge \partial_U \Kum(\tfrac{q}{2}) \bigwedge (\partial^2_U + \partial_V)\Kum(\tfrac{q}{2}) = 0,
\end{equation}
for any $q \in C$. This computation we leave as an exercise.

Geometrically, this is an equation that guarantees that we have a {\em flex line} to the Kummer variety at the point $\Kum(\tfrac{q}{2})$, where by a flex line we mean a line that is tangent to $\Kum(\Jac(C))$ at that point with multiplicity three. In this sense we view the flex line as the most degenerate trisecant line.

If we have a family of flex lines for the Kummer image, then following the same strategy we can actually take one more limit by letting $q$ approach $p(= 0)$ and expand the expression \eqref{eq:Kummerflex1} in $q$ to all orders. The lowest nontrivial term is then given by
\begin{exercise}
Let $q$ approach $p(= 0)$; then the lowest nontrivial term is of order~$\ve^4$, which gives
\begin{equation}\label{eq:Kummerflex2}
(\partial_U^4 - \partial_U\partial_W + \tfrac{3}{4}\partial_V^2 + c)\Kum(0) = 0\,,
\end{equation}
where $c$ is some constant and $W$ is the third order derivative of the Abel-Jacobi map at the chosen point.
\end{exercise}

\begin{remark}
Note that as $0$ is a two-torsion point of the Jacobian, the Kummer image is actually singular there, and since the Kummer map is even, the first and third derivatives of $\Kum(z)$ vanish at $z=0$, so that only the second and fourth derivatives appear in \eqref{eq:Kummerflex2}.
\end{remark}
\begin{remark}
If~$C$ is a hyperelliptic curve, and one starts with a Weierstrass point $p\in C$, then $\partial_V\Kum(0)=0$, and the equation above simplifies to
\[(\partial_U^4 - \partial_U\partial_W + c)\Kum(0) = 0\,.\]
\end{remark}

Since now we have an explicit differential equation \eqref{eq:Kummerflex2} for the Kummer image, we can use Riemann's bilinear relation to convert it to a differential equation for the usual Riemann theta function. 

\begin{exercise}\label{ex:getKP}
Use Riemann's bilinear relation to convert \eqref{eq:Kummerflex2} to a differential equation for the theta function at any $z\in\CC^g$.
\end{exercise}
This will turn out to be exactly the KP equation that will be introduced in \Cref{sec:KP} (where we will also solve this exercise).

Welters studied the infinitesimal version of Gunning's \Cref{thm:Gunning'sthm1} resulting from the above analysis.

\begin{theorem}[Welters~1983~\cite{Welters1983}] \label{thm:Weltersthm1}
    For $A \in \cA_g$, if there exists a fixed point $z\in A$ and a curve $C \subset A$ such that for any $p\in C$, the Kummer variety admits a flex line at $\Kum\left(p+\tfrac{z}{2}\right)$, then $A = \Jac(C)$.
\end{theorem}
The next stronger result of Welters \cite{Welters1984} further relaxes the assumption of the existence of a geometric one-dimensional family of flexes parameterized by a curve~$C$ to the existence of an infinitesimal formal germ of such a family. Morally, there is a clear approach to prove this strengthening: starting with a germ, one extends in to an actual family, and then applies the theorem above. That's how the proof proceeds, but we are not going to spell out the details. 

\begin{remark}
Note that an infinitesimal formal germ contains the information of an infinite jet at a point. Arguing as before by considering higher and higher order derivatives of the Abel-Jacobi map, Welters' result \cite{Welters1984} analytically is the statement that the full KP hierarchy characterizes Jacobians of curves (see also \cite{Mulase1984}). Note that from this point of view the KP hierarchy is just obtained as higher and higher order expansions of the family of flex lines, as we'll also discuss below.
\end{remark}

\Cref{thm:Weltersthm1} is a beautiful result, but again it  starts with a curve $C\subset A$ (or at least its infinite germ), and thus it gives a solution only to the restricted Schottky problem. The next theorem, due to Arbarello and De Concini, relaxed the condition of the existence of an infinite germ of a family of trisecants to just a finite order jet.

\begin{theorem}[Arbarello-De~Concini~1984~\cite{Arbarello-DeConcini1984}]\label{thm:ADfinitegerm}
For any $g$, there exists an $N = N(g)$ such that the existence of an $N$'th order formal jet of a family of flex lines of the Kummer variety characterizes Jacobians.
\end{theorem}

The bound $N(g)$ is not given explicitly in \cite{Arbarello-DeConcini1984}, and if one follows the argument, it would certainly grow horribly as $g$ increases. The point, however, is that having a finite (but high) order jet of flexes suffices to extend this to an infinitesimal formal germ, and then it suffices to extend this to an actual curve. That's how the logic of the proof goes.

The next theorem was a conjecture of Novikov, proven by Shiota.
\begin{conjecture}[Novikov's conjecture; Theorem of~Shiota~1986~\cite{Shiota1986}]\label{conj:Novikov}
The existence of a $4$'th order jet of a family of flex lines of the Kummer variety characterizes Jacobians. Equivalently, a principally polarized (indecomposable) abelian variety is a Jacobian if and only if the associated theta functions satisfies the KP equation \eqref{eq:Kummerflex2}.
\end{conjecture}
The infinite order jet is the KP hierarchy, and \Cref{conj:Novikov} is about the KP equation, the first nontrivial equation in the hierarchy. So why should one believe this conjecture? Somehow one wants to start from a fourth order jet of a family of flexes, and then inductively to argue that the higher order jets of a family of flexes exist, by proving at each step that the obstruction to the extension should vanish; since $(n+1)$-jet bundle is an extension of the $n$-jet bundle by $1$-jets, one hopes that the obstruction is the same at every step. So one might believe that the existence of a $4$-jet guarantees that the obstruction vanishes, and then it will also vanish for any higher order jets. This is morally what one would like to believe, but the proof is hard. There is also an algebro-geometric proof of Novikov's conjecture due to Arbarello, De Concini and Marini \cite{Arbarello-DeConcini1987, Marini1998}.

\smallskip
Notice that in all of the above progressively stronger results we start with the existence of a jet of a family of flex lines, and then the results are proven by extending this to a higher and higher order jet, then an infinite order jet, and then arguing the existence of an actual geometric curve of flex lines. In \cite{Welters1984} Welters made a much stronger conjecture that gets rid of jets altogether:
\begin{conjecture}[Welters~1984~\cite{Welters1984}]\label{conj:Welters1}
If the Kummer variety $\Kum(A)$ of some indecomposable principally polarized abelian variety has {\em a} trisecant line, $A$ is the Jacobian of some curve.
\end{conjecture}
The precise conjecture actually has three versions, by considering a usual non-degenerate trisecant through three distinct points, by considering a curve tangent to the Kummer variety at one point and intersecting it at another, and by considering the fully degenerate case of the existence of one flex line. Though this conjecture follows in the spirit of the above theorems, it is a very different statement. There is no curve, not even a jet of a curve to start with, and thus \Cref{conj:Welters1} provides a solution to the full Schottky problem, not just to the restricted Schottky problem.

All three cases of the Welters' conjecture were proven by Krichever in the breakthrough works \cite{Krichever2006,Krichever2010}. The ultimate goal of these notes is to present Krichever's proof of the most degenerate case of Welters' conjecture:

\begin{theorem}[Krichever~2010~\cite{Krichever2010}]\label{thm:FLEX}
The existence of one flex line of the Kummer variety characterizes Jacobians of curves among all indecomposable principally polarized abelian varieties.
\end{theorem}
Notice that this theorem, and the way we have to think about, is disconnected from what we had above, because if somehow we have a jet we can try to extend it, but here we simply do not have a jet to start with. In the previous discussion in this section, we were conveniently taking limits, but here we could not take any limits if we have just one flex line. This is the goal for the rest of the lectures!

\begin{remark}
Recall that the Kummer variety is singular at the two-torsion points, so the tangent space to it at a two-torsion point is a tricky business. We  will thus prohibit all of our secants to go through any two-torsion points. There is an open conjecture, called the~$\Gamma_{00}$ conjecture \cite{vanGeeman-vanderGeer1986}, that characterizes Jacobians in terms of the geometry of the base locus of a linear subsystem of $|2\Theta|$ associated to a two-torsion point. The $\Gamma_{00}$ conjecture remains completely open, except for the easiest case \cite{GrushevskyGamma}.
\end{remark}

\begin{remark}
The analogous characterization of the Prym varieties among all principally polarized abelian varieties, via a symmetric pair of quadrisecants of the Kummer variety, was obtained by Grushevsky and Krichever in \cite{Grushevsky-Krichever2010}, and the integrable system used there is a discrete analogue of Novikov-Veselov hierarchy. Such integrable discretization was investigated in \cite{DGNS2007} earlier by Doliwa, Grinevich, Nieszporski and Santini. Taimanov \cite{Taimanov1990} might be the first one who tried to use the Novikov-Veselov equation to resolve the analogue of Novikov's conjecture for Prym varieties. 
\end{remark}

In the rest of these lecture notes, we try to give a relatively complete proof of Krichever's \Cref{thm:FLEX} characterizing Jacobians of curves by the existence of one flex line of the Kummer variety. For that we will allow ourselves to freely use Gunning's \Cref{thm:Gunning'sthm1} whose proof requires a whole set of different tools that we do not develop here. 

\subsubsection*{Summary of this section} Starting just from Riemann's theta singularity theorem, we have deduced Weil's reducibility for the intersection of a theta divisor of a Jacobian with its suitable translate. This turns out to be equivalent to the existence of trisecant lines of the Kummer image of the Jacobian in $\CC\PP^{2^g-1}$. Furthermore, the existence of a family of trisecants characterizes Jacobians among all principally polarized abelian varieties. Krichever proved Welters' conjecture that the existence of one trisecant characterizes Jacobians --- and giving an outline of the proof of this for the case of one flex line is the goal of the remaining lectures.

\section{Baker-Akhiezer Functions and the KP Hierarchy}\label{sec:BAfunction}

Now we switch gears again to differential equations and develop some key techniques for the proof of Krichever's theorem in the flex line case. We describe in this section a technique for obtaining solutions of certain differential equations using algebro-geometric data, as discussed for motivation right at the beginning of these notes. The setup is remarkably versatile, and in full generality uses the so-called Baker-Akhiezer functions introduced by Krichever in 1970s.

Let $C$ be a curve of genus $g$, and let~$p_{\infty} \in C$ be a point. Let $k^{-1}$ be a local coordinate on $C$ around $p_{\infty}$, so that $k(p_{\infty}) = \infty$. If we think about this, this is a lot of data! The moduli space of curves of genus $g$ is a finite-dimensional space, and we choose a point on it, which then adds one complex dimension, but choosing a local coordinate $k^{-1}$, at all orders in the Taylor expansion, means choosing infinitely many terms, and is thus infinite-dimensional datum. The space of such data admits a forgetful map to $\cM_{g,1}$, with fibers that are spaces of all local coordinates at~$p_\infty$, which are infinite-dimensional linear spaces.

\subsection{The Baker-Akhiezer functions}

\begin{propdef}\label{def:Baker-Akhiezer}
For any $(C, p_{\infty}, k^{-1})$ as above and for any general divisor $D = p_1 + \dots + p_g\in \Div^g(C)$, there exists a unique function $\psi(x, p)$ (of the variables $x \in \CC, p\in C$), called the {\em Baker-Akhiezer function}, such that 
\begin{enumerate}
\item for any fixed~$x$, $\psi$ is a meromorphic function on $C \setminus \{p_{\infty}\}$, with simple poles at $D$, and holomorphic elsewhere on $C \setminus \{p_{\infty} \cup D\}$;
\item $\psi$ has an essential singularity at $p_{\infty}$, such that $\psi(x, p)e^{-kx}$ is holomorphic around $p_{\infty}$, and equal to $1$ at $p_{\infty}$.
\end{enumerate}
\end{propdef}
What the above means is that $\psi(x, p)$ is a function of $p \in C$, with $x \in \CC$ as a parameter. The claim is that there exists a unique function with these properties. Akhiezer considered the problem of the uniqueness of Baker-Akhiezer function at least for hyperelliptic curves \cite{Akhiezer1961}, and Baker derived many differential equations satisfied by abelian functions \cite{Baker1897, Baker1907} including essentially the KP equation, which we are going to discuss in the following. The hyperelliptic case in its full generality was completed by Its and Matveev \cite{Its-Matveev1975} and people from the Novikov school \cite{Dubrovin-Matveev-Novikov1976}. The Baker-Akhiezer function in full generality, as defined above, is due to Krichever \cite{Krichever1977}.

\begin{proof}[Proof of \Cref{def:Baker-Akhiezer}]

\noindent Uniqueness: we first prove the uniqueness of the Baker-Akhiezer function $\psi(x, p)$ which is the easy part of the proposition. If $\psi(x, p)$, $\tilde\psi(x, p)$ are two functions satisfying the conditions above, then their quotient $\tfrac{\tilde\psi(x, p)}{\psi(x, p)}$ has no essential singularity at $p$ (the essential singularity cancels) and is a meromorphic function with poles at the divisor of zeros $D'$ of $\psi(x, p)$. By the following explicit construction of $\psi(x, p)$ we would see that $D'$ is also a general divisor, so the ratio is necessarily a constant. As $\frac{\tilde\psi(x, p_{\infty})}{\psi(x, p_{\infty})}=1$, this constant is necessarily $1$.

\noindent Existence: $\psi(x, p)$ satisfying conditions (1) and (2) can be explicitly constructed by the following formula
\begin{equation}\label{eq:Baker-Akhiezer1}
 \psi(x, p) \coloneqq \frac{\theta(\int_{p_{\infty}}^p \omega + Ux + D)\theta(D)}{\theta(\int_{p_{\infty}}^{p}\omega + D)\theta(Ux + D)} \cdot\exp\left(\int_{p_0}^{p}\Omega(p)x - bx\right),
\end{equation}
where $\omega = (\omega_1, \dots, \omega_g)^t$ is the basis of abelian differentials normalized as in \eqref{eq:normalizeddifferential}, 
$\Omega(p)$ is the meromorphic differential on~$C$ with a unique second order pole at $p_\infty$, with singular part given by
\[\Omega(p) \sim dk, \quad \text{as } p \to p_{\infty}\]
(note that since $k(p_\infty)=\infty$, this prescribes a second order pole!), and satisfying the normalization condition
\[\int_{A_i} \Omega(p) = 0, \qquad \forall 1 \le i \le g\,.\]
The constant $b$ is chosen so that condition (2) is satisfied, that is 
\[\int_{p_0}^p \Omega(p) = k + b + o(1).\]
Finally, the constant vector $U = (U_1, \dots, U_g)^t$ is 
\[U_j \coloneqq \int_{B_j}\Omega(p)\,,\]
and $D = (D_1, \dots, D_g)^t$ is the Abel-Jacobi image of~$D$, given analytically by
\[D_j \coloneqq - \sum \limits_{i = 1}^g \int_{p_0}^{p_i} \omega_j - K_j, \qquad 1 \le j \le g\,, \]
where $K = (K_1, \dots, K_g)^t$ is the Riemann constant, explicitly given by
\begin{equation}\label{eq:Riemannconstant}
K_j \coloneqq\frac{\tau_{jj}}{2} - \sum \limits_{l \ne j}\left(\int_{A_l} \omega_l(p)\int_{p_0}^p \omega_j\right), \quad 1 \le j \le g\,.
\end{equation}
In the above expressions, $p_0 \ne p_{\infty} \in C$ is any chosen fixed point. 

Note that the function $\psi$ given by \eqref{eq:Baker-Akhiezer1} has poles exactly at the points of the divisor~$D$, and by construction it has an essential singularity at $p_{\infty}$ with the right asymptotics prescribed by condition (2) with the right normalization. The only thing that needs to be checked is that $\psi(x, p)$ is indeed a well-defined function on $C$, i.e.~ that the expression above is single-valued on~$C$. This is equivalent to $\psi(x, p)$ being invariant when the point $p$ goes around an arbitrary cycle $\gamma \in H_1(C, \ZZ)$. In particular we can verify this on the basis $A_1, \dots, A_g, B_1, \dots, B_g$, and this invariance follows from the quasi-periodicity of the $\theta$-function in \Cref{ex:thetamodularity} and the normalizations above.

From \eqref{eq:Baker-Akhiezer1}, we also see that the divisor of zeroes of $\psi(x, p)$ is $Ux + D$, more precisely, is the inverse image of the vector $-(Ux + D)$ under the Abel-Jacobi map. As the inverse image of $-D$ is non-special by the genericity assumption, its small perturbation $-(Ux + D)$ is also a non-special divisor on $C$ for $x$ sufficiently small. 
\end{proof}

\begin{remark}
Why should we believe that things such as the Baker-Akhiezer function exist? Maybe we shouldn't believe it easily --- there are actually a lot of choices involved in \eqref{eq:Baker-Akhiezer1}: the basis of $H_1(C, \ZZ)$ is not canonical, the point $p_0 \in C$ in the definition of $D$ is arbitrary, and so forth. At the end, all these ambiguities went away, and we obtained a well-defined function on~$C$. There are some reasons, however, that make this construction believable. Recall that in \Cref{thm:functionsintheta} we constructed arbitrary meromorphic functions on a curve from theta functions, and there were also many choices involved. In some sense all we did here was to add an explicit exponential factor to create an essential singularity of a specified form.
\end{remark}
\begin{remark}
Why is the Baker-Akhiezer function  useful? After all, it is just a meromorphic function on $C\setminus\{p_\infty\}$ with an essential singularity at~$p_\infty$. Much of its use comes from the free extra parameter~$x$, which we will now exploit. Indeed, consider the derivative of $\psi(x,p)$ with respect to~$x$. This will again be a meromorphic function on~$C\setminus\{p_\infty\}$ with a different essential singularity at~$p_\infty$. By equating the asymptotics at essential singularities, we will be able to obtain some differential equations satisfied by $\psi$, which is what we are heading for now.
\end{remark}

The definition of the Baker-Akhiezer function can be easily generalized to the multi-variable case. Let $\ttt\coloneqq (t_1, t_2, t_3, \dots)$ be a sequence of parameters, then in the same setting as \Cref{def:Baker-Akhiezer} there exists a unique Baker-Akhiezer function $\psi(\ttt, p)$ such that
\begin{enumerate}
    \item $\psi$ is a meromorphic function on $C \setminus\{p_{\infty}\}$, with simple poles at $D$, and holomorphic elsewhere on $C \setminus \{p_{\infty} \cup D\}$;
    \item $\psi$ has essential singularity at $p_{\infty}$, such that $\psi(\ttt, p)\exp\left(-\sum \limits_{i = 1}^{\infty}k^it_i\right)$ is holomorphic around $p_{\infty}$, and equal to $1$ at $p_{\infty}$.
\end{enumerate}
Such a function $\psi(\ttt, p)$ is given by a similar explicit formula:
\begin{equation}\label{eq:Baker-Akhiezer2}
    \psi(\ttt, p) = \frac{\theta\left(\int_{p_{\infty}}^p \omega + \sum \limits_{i = 1}^{\infty}U^{(i)} t_i + D\right)\theta(D)}{\theta\left(\int_{p_{\infty}}^{p}\omega + D\right)\theta\left(\sum \limits_{i = 1}^{\infty}U^{(i)} t_i + D\right)} \exp\left(\sum \limits_{i = 1}^{\infty}\left(\int_{p_0}^{p}\Omega_i(p) - b_i\right)t_i\right),
\end{equation}
where each $\Omega_i(p)$ is the meromorphic differential of the second kind (recall that this classical terminology means all of its residues are zero) with the unique singularity at $p_{\infty}$ of the form
\[\Omega_i(p) \sim dk^i, \quad \text{as } p \to p_{\infty} \]
and satisfying
\[\int_{A_j} \Omega_i(p) = 0, \qquad 1 \le j \le g\,.\]
The constants $b_i$'s are chosen so that (2) is satisfied, that is
\[\int_{p_0}^p\Omega_i(p) = k^i + b_i + o(1)\,,\]
and the vector $U^{(i)} = \left(U^{(i)}_1, \dots, U^{(i)}_g\right)^t$ is then given by
\[U^{(i)}_j \coloneqq \int_{B_j}\Omega_i(p)\,.\]
Here we are not going to worry about in what sense the infinite sums involving $t_i$'s will converge, as we will only use the first three variables $t_1 = x, t_2 = y, t_3 = t$ in the following, setting all higher $t_i$ to zero.

\subsection{The KP hierarchy}\label{sec:KP}

Recall that in \Cref{sec:CommDO} we were solving the eigenvalue problem for a differential operator $L$ in one variable, and we required the normalized eigenfunction to have the form \eqref{eq:normalizedeigenfunctionpsi}  consisting of an exponential function multiplied by a formal series in $k^{-1}$. The Baker-Akhiezer function \eqref{eq:Baker-Akhiezer1} is something of the same sort. What is the difference between $\psi$ defined in \eqref{eq:normalizedeigenfunctionpsi} and the Baker-Akhiezer function \eqref{eq:Baker-Akhiezer1}? The main difference is that the discussion in \Cref{sec:CommDO} is purely formal and local, but the Baker-Akhiezer function $\psi(x, p)$ is a meromorphic function on the Riemann surface $C\setminus \{p_\infty\}$ --- that is, it is a global object. Now the goal is quite clear: we want to unite these two approaches! The advantage of the global solutions is that we have a whole set of tools from geometry of Riemann surfaces. Somehow if we can get functions of this sort, they are not just formal, but they are actual functions on the curve with an essential singularity and we can pass from formal solutions to actual geometric solutions. So we should see that the Baker-Akhiezer functions indeed solve some eigenvalue problem. For that we will use their uniqueness: noting that if we have more than one way to construct functions with the same prescribed essential singularity at some point $p_{\infty} \in C$, for example by taking differential operators acting on the Baker-Akhiezer function $\psi$ or by multiplying it with a function on the curve, then they should be the same function. Indeed, we have the following

\begin{theorem}\label{thm:BAaseigenfunction}
For any meromorphic function 
\[E: C \to \CC\PP^1\]
on the algebraic curve $C$ with an $n$'th order pole at $p_{\infty}$ and holomorphic elsewhere, there exists a unique $n$'th order differential operator in one variable,
\[L = \sum \limits_{i = 0}^n u_i(x)\frac{d^i}{dx^i}\]
such that
\[L\psi(x, p) = E(p)\psi(x, p)\,,\]
where $\psi(x, p)$ is the Baker-Akhiezer function \eqref{eq:Baker-Akhiezer1}.
\end{theorem}
\begin{proof}
We can first expand the Baker-Akhiezer function \eqref{eq:Baker-Akhiezer1} in the form of \eqref{eq:normalizedeigenfunctionpsi}, i.e.~we expand the fractional part involving various theta functions
\[\psi(x, p)\exp\left(-k(p)(x - x_0)\right)\]
in terms of powers of $k^{-1}$. We then have
\begin{claim}
For any series of the form \eqref{eq:normalizedeigenfunctionpsi} there exists a unique operator $L$ such that
\begin{equation}\label{eq:congruenceL1}
L\psi(x, k) \equiv k^n \psi(x, k)u_n \mod O(k^{-1}e^{k(x - x_0)}).
\end{equation}
\end{claim}
\begin{proof}[Proof of the claim]
The coefficients of $L$ can be found successively from \eqref{eq:recursivexi} for $l = 1, \dots, n$ (note that in \eqref{eq:recursivexi} we normalized $u_n$ to be $1$ and this is compensated by the $u_n$ on the right-hand side of~\eqref{eq:congruenceL1}). 
\end{proof}
Now we consider the function $L\psi(x, p) - E(p)\psi(x, p)$ defined on $C$. It satisfies all the requirements defining the Baker-Akhiezer function $\psi(x, p)$ except one, that is its value at $p_{\infty}$ is equal to zero. From the uniqueness of the Baker-Akhiezer function, it follows that $L\psi(x, p) = E(p)\psi(x, p)$. 
\end{proof}

Let $\cA(C, p_{\infty})$ be the ring of meromorphic functions on $C$ that are holomorphic on $C\setminus\{p_\infty\}$ and may have an arbitrary pole order at~$p_{\infty}$. Then by \Cref{thm:BAaseigenfunction} the Baker-Akhiezer function \eqref{eq:Baker-Akhiezer1} gives a homomorphism~$\Lambda$ from $\cA(C, p_{\infty})$ into the ring of linear differential operators: for any  $E(p) \in \cA(C, p_{\infty})$, we associate with it the differential operator constructed in \Cref{thm:BAaseigenfunction}. Since $\cA(C, p_{\infty})$ is a commutative ring, its image under $\Lambda$ is also commutative. Combining the results from \Cref{sec:CommDO} and the discussion in this section, we have

\begin{theorem}[Krichever~1977~\cite{Krichever1977}]\label{thm:spectralcurve}
For arbitrary differential operators in one variable $L_1, L_2$ such that $[L_1, L_2] =0$, and $\gcd(\ord(L_1), \ord (L_2)) = 1$, there exists a curve $C$ defined as $\{(\alpha, \beta) \in \CC^2: Q(\alpha, \beta) = 0\}$, a point $p_{\infty} \in C$, and a local coordinate $k^{-1}$ around~$p_\infty$, such that $\psi(x, p)$ is a common eigenfunction of $L_1$ and $L_2$. Moreover, $Q(L_1, L_2) = 0$.    
\end{theorem}
\begin{remark}
When the orders of $L_1$ and $L_2$ are not coprime, the situation is more elaborate and much less well-understood, though with a lot of progress for example in \cite{Krichever-Novikov1978, Krichever-Novikov1979, Krichever-Novikov1980, Grunbaum1988, Previato-Wilson1992}.
\end{remark}

This theorem says that for any commuting differential operators $L_1,L_2$ of coprime orders in one variable, there exists a curve $C$ given by the polynomial equations satisfied by $L_1$ and $L_2$, and the corresponding Baker-Akhiezer functions constructed starting from~$C$ give common eigenfunctions of~$L_1$ and~$L_2$. Recall that in \Cref{sec:CommDO} we learned that if we have two commuting differential operators in one variable, then they satisfy a polynomial equation. This polynomial equation defines an algebraic curve in $\CC^2$, which should be compactified either in $\CC\PP^2$ or in $\CC\PP^1 \times \CC\PP^1$. The point at infinity, however, is a singular point which in general has $n$ branches coming together, and if we want to separate the branches we should set $E= k^n$ at $p_{\infty}$. The curve $C$ here does not have to be smooth and the point at infinity is not smooth. What we could do is take the closure of the curve $C$ in $\CC\PP^2$ or $\CC\PP^1 \times \CC\PP^1$, and then take its normalization at infinity. Alternatively, one could consider Baker-Akhiezer functions directly on singular curves, as long as the point~$p_\infty$ is smooth.

In the next theorem, we are going to use a slightly more general Baker-Akhiezer function, i.e.~by setting $t_1 = x$, $t_2 = y$, $t_3 = t$, and setting all other time parameters to be zero in \eqref{eq:Baker-Akhiezer2}, to obtain the Baker-Akhiezer function $\psi(x, y, t, p)$. We have then

\begin{theorem}\label{thm:BAassolution}
Let $t_1 = x$, $t_2 = y$, $t_3 = t$, $t_4 = t_5 = \dots = 0$ in \eqref{eq:Baker-Akhiezer2}. Then for $\psi(x, y, t, p)$ there exist unique differential operators $L_1 = \sum \limits_{i = 0}^2 u_i(x, y, t)\frac{\partial^i}{\partial x^i}$ and $L_2 = \sum \limits_{i = 0}^3 v_i(x, y, t)\frac{\partial^i}{\partial x^i}$ such that 
\[L_1\psi = \frac{\partial \psi}{\partial y}, \qquad L_2\psi = \frac{\partial \psi}{\partial t}\,.\]
\end{theorem}

\begin{proof}
As in the proof of \Cref{thm:BAaseigenfunction}, we can first expand $\psi(x, y, t, p)$ in the form
\[\psi(x, y, t, k) = \left(\sum \limits_{s = 0}^{\infty} \xi_s(x, y, t)k^{-s}\right) e^{kx + k^2 y + k^3 t}\,.\]
Then there exist unique operators of the form $L_1$ and $L_2$ in \Cref{thm:BAassolution} such that
\begin{align*}
    & \left(L_1 - \frac{\partial}{\partial y}\right)\psi(x, y, t, k) \equiv 0 \mod O(k^{-1})e^{kx + k^2y + k^3t}\\
    & \left(L_2 - \frac{\partial}{\partial t}\right)\psi(x, y, t, k) \equiv 0 \mod O(k^{-1})e^{kx + k^2y + k^3t}.     
\end{align*}
The coefficients of $u_i, v_i$ can be found from similar recursive relations as in the proof of \Cref{thm:BAaseigenfunction}. And again we finish the proof of the theorem by the uniqueness of the Baker-Akhiezer function on the curve $C$ with prescribed singularities.
\end{proof}

\begin{remark}
In \Cref{thm:BAassolution} we prescribed the essential singularity at $p_{\infty}$ to be $\exp(kx + k^2 y + k^3t)$. More generally, one could prescribe the exponential part to be $\exp\left(kx + Q(k)y + R(k)t\right)$, where $Q$ and $R$ are arbitrary polynomials of $k$. Then there would still exist unique differential operators $L_1, L_2$ of orders $\deg Q$ and $\deg R$, respectively, such that
\[L_1\psi = \frac{\partial \psi}{\partial y}, \qquad L_2\psi = \frac{\partial \psi}{\partial t}\,.\]
The proof is the same.
\end{remark}

For $L_1$ and $L_2$ constructed in \Cref{thm:BAassolution} we have
\begin{theorem}
\[\left[L_1 - \frac{\partial}{\partial y}, L_2 - \frac{\partial}{\partial t}\right] = 0\,.\]
\end{theorem}

\begin{proof}
We essentially know how to prove this. Recall that in \Cref{sec:CommDO} on formal differential operators we showed that if for two differential operators $L_1$ and $L_2$ a formal function $\psi$ was a common eigenfunction for each value of $x$, then they must commute. The proof used the finite-dimensionality trick: indeed, $\psi$ gives an infinite-dimensional kernel for the commutator of $L_1$ and $L_2$ which is a finite order differential operator. Here the commutator $\left[L_1 - \frac{\partial}{\partial y}, L_2 - \frac{\partial}{\partial t}\right]$ annihilates $\psi(x, y, t, k)$ for any $x, y, t, k$, while we note that the commutator itself contains differentiation only with respect to $x$, and is thus a finite order differential operator:
\[\left[L_1 - \frac{\partial}{\partial y}, L_2 - \frac{\partial}{\partial t}\right] = \frac{\partial}{\partial t}L_1 - \frac{\partial}{\partial y}L_2 + [L_1, L_2]\,.\]
Thus the kernel of this commutator must still be finite-dimensional, but it contains a one-parameter family of functions $\psi(x, y, t, k)$. This implies that the commutator is identically zero. 
\end{proof}

Next we consider the central example of constructing solutions of the Kadomtsev-Petviashvili (KP) equation using the above theorem.
\begin{example}
Let
\[\psi(x, y, t, p) = \left(1 + \xi_1(x, y, t)k^{-1} + \xi_2(x, y, t)k^{-2} + \dots\right) \exp \left(kx + k^2 y + k^3 t\right).\]
be the asymptotics of the essential singularity in the coordinate~$k^{-1}$ near~$p_{\infty}$. What are then the differential operators constructed by the procedure above, explicitly? It turns out that
\[L_1 = \frac{\partial^2}{\partial x^2} + u\,, \qquad L_2 = \frac{\partial^3}{\partial x^3} + \frac{3u}{2}\frac{\partial}{\partial x} + w\,,\]
where 
\[u = 2 \frac{\partial\xi_1}{\partial x}\,, \quad \text{and} \quad w = 3 \frac{\partial \xi_2}{\partial x} + 3 \frac{\partial^2 \xi_1}{\partial x^2} + \xi_1 \frac{\partial \xi_1}{\partial x}\,.\]
As we have an explicit expression of the Baker-Akhiezer function $\psi(x, y, t, p)$ in terms of data from algebraic curves, we will be able to compute all the coefficients $\xi_s$, in particular we have
\begin{equation}\label{eq:potentialuintheta}
u(x, y, t) = 2 \left.\frac{\partial^2}{\partial x^2}\ln \theta(U^{(1)}x + U^{(2)} y + U^{(3)} t + Z)\right\vert_{p_{\infty}},
\end{equation}
where $Z \in \CC^g$ is a vector which can be explicitly computed but whose expression will not be so relevant for our discussion in the following. \footnote{There is usually also a constant vector we need to add to $u$, but it can always be set to $0$ by a simple shift of the independent variables (see \cite{Dubrovin1981}).} 

On the other hand, the commutation relation
\[\left[L_1 - \frac{\partial}{\partial y}, L_2 - \frac{\partial}{\partial t}\right] = 0\,,\]
gives the following differential equation for $u$:
\begin{equation}\label{eq:KP}
\frac{3}{4}u_{yy} = \frac{\partial}{\partial x}\left(u_t - \frac{3}{2}u u_x - \frac{1}{4}u_{xxx}\right).
\end{equation}
This is finally the so-called {\em Kadomtsev-Petviashvili (KP)} equation, and as we reached this crucial point in our presentation, many remarks and pointers to related and more general constructions are merited.
\end{example}

\begin{remark}\label{rmk:KPhierarchy}
We can consider the more general Baker-Akhiezer function \eqref{eq:Baker-Akhiezer2}. By the same argument, for each~$t_k$ we have a differential operator $L_k$ such that $\left(L_k - \frac{\partial}{\partial t_k}\right)$ annihilates $\psi(\ttt, p)$. By considering the commutators among these operators we get the so-called Zakharov-Shabat (aka zero-curvature) equations
\[\left[L_n - \frac{\partial}{\partial t_n}, L_m - \frac{\partial}{\partial t_m}\right] = 0\,.\]
This gives an infinite sequence of differential equations for the coefficients of the operators~$L_k$, and altogether they form an integrable hierarchy, called the KP hierarchy.
\end{remark}
\begin{remark}
Note that the Korteweg-de Vries (KdV) equation is a special case of \eqref{eq:KP}. Indeed, if~$u$ does not depend on $y$, then $u(x, t)$ satisfies the following simpler nonlinear partial differential equation
\begin{equation}\label{eq:KdV}
u_t - \frac{3}{2}u u_x - \frac{1}{4}u_{xxx} = \,0.
\end{equation}
Recalling how $u$ is constructed, we see that this happens when there exists a function $E:C\to\CC\PP^1$ with a unique double pole at $p_\infty$, in which case the curve $C$ is hyperelliptic. KdV also arises in many other places: in the study of intersection theory on the moduli space of curves, shallow water waves, and in lots of other beautiful mathematics.
\end{remark}
\begin{remark}
In most computations it is convenient to set $p_0=0$ in \eqref{eq:Baker-Akhiezer1}, as there is then still enough freedom of choice left in this whole story. The vector $Z \in \CC^g$ is the Abel-Jacobi image of the general divisor $D$. Note that if we have some equation which is satisfied for a general point of an abelian variety, then it is satisfied at every point of an abelian variety. Thus we can choose $D$ freely, and thus any point~$Z$ in the abelian variety would work. Geometrically there is a surjective map from the $g$'th symmetric power of the curve $C$ onto the Jacobian
\[\text{Sym}^g C \twoheadrightarrow \Jac(C)\,, \]
and so every point of the Jacobian can be written as a sum of Abel-Jacobi images of~$g$ points of~$C$, and $Z\in\Jac(C)$ can thus be arbitrary. So we do not really need~$p_0$, but we do want to emphasize that \eqref{eq:KP} is an equation valid for all $Z$.
\end{remark}

\begin{remark}\label{rmk:KPgeometry}
There is an algebro-geometric viewpoint on the KP hierarchy, by using the infinite-dimensional Grassmannian (see \cite{Sato1981, Segal-Wilson1985} and the in-depth survey \cite{mulase}), which we now briefly discuss and make connections with the theory we are developing here. The theory of Krichever on algebro-geometric solutions of KP hierarchy begins with an algebraic curve and the Baker-Akhizer function on it. Crucially, the normalized Baker-Akhiezer function with a prescribed essential singularity is unique, which allowed us to deduce that it satisfied certain linear differential equations. The compatibility conditions of these equations turn out to be the KP hierarchy. 
It is possible to go the other way around, starting from a system of compatible linear differential equations, and then trying to construct their common eigenfunctions.
This theory, initiated by Sato, begins with a pseudo-differential operator
\begin{equation}\label{eq:pseudoL}
    \cL = \partial_x + u_2 \partial_x^{-1} + u_3\partial_x^{-2} + \dots,
\end{equation}
where $\partial_x = \frac{\partial}{\partial x}$ and $\partial_x^{-1}$ is the formal inverse of $\partial_x$, i.e.~$\partial_x^{-1}\partial_x = \partial_x \partial_x^{-1} = 1$, and $u_i$'s are functions that may depend on infinitely many time variables $\ttt = (t_1, t_2, t_3, \dots)$, which will serve as deformation parameters. We require the composition of the pseudo-differential operators to satisfy the generalized Leibniz rule, i.e.~that multiplying by~$f(x)$ and then applying a pseudo-differential operator $\partial_x^{m}$ is given by
\begin{equation}\label{eq:partialmf}
   \partial_x^{m} \circ f  = \sum \limits_{j \ge 0}\binom{m}{j}\partial_x^j(f)\partial_x^{m - j}, 
\end{equation}
where the binomial coefficient $\binom{m}{j}$ is defined as
\[\binom{m}{j} = \frac{m(m - 1) \cdots (m - j + 1)}{j(j-1) \cdots 1}.\]
By definition, the KP hierarchy is the isospectral deformation of the pseudo-differential operator $\cL$, that is we consider the formal eigenvalue problem\footnote{The meaning of equation \eqref{eq:eigenvalueproblemcL} is as follows: we apply each term of $\cL$ to each term of $\psi$ viewed as a function of $x$ according to the rule \eqref{eq:partialmf} and reorganize the results according to the powers of $\partial_x$, and we require it to equal the multiplication by $k$ operator, that is the coefficient of $\partial_x^0$ in $\cL \psi$ is $k$, and the coefficients of all other $\partial_x^i, i \ne 0$ vanish.}
\begin{equation}\label{eq:eigenvalueproblemcL}
    \cL \psi = k\psi,
\end{equation}
and  deform the coefficients $u_i$ of $\cL$ in such a way that $k \in\CC$ is constant, independent of $\ttt$. The pseudo-differential operator $\cL$ defined in \eqref{eq:pseudoL} and the associated eigenvalue problem \eqref{eq:eigenvalueproblemcL} are universal in the sense that for any $n$'th order differential operator such as the one defined in \eqref{eq:Lcanonical}, and its associated eigenvalue problem \eqref{eq:normalizedeigenvalueproblem}, there is an associated pseudo-differential operator $\cL = L^{1 \slash n}$ and the corresponding eigenvalue problem of the form \eqref{eq:eigenvalueproblemcL} equivalent to the original one for $L$, and the pseudo-differential operator $\cL$ viewed as a formal symbol has more freedom to be manipulated algebraically. By the {\em KP flows} we refer to the time evolutions of the pseudo-differential operator $\cL$ with respect to $\ttt$. If $\psi$ is deformed so that
\begin{equation}\label{eq:psideformBn}
    \partial_{t_n}\psi = \cB_n\psi,
\end{equation}
where $\cB_n$ is an $n$'th order pseudo-differential operator, then the compatibility condition between \eqref{eq:eigenvalueproblemcL} and \eqref{eq:psideformBn} is
$$
    \partial_{t_n}(\cL \psi)  = (\partial_{t_n}\cL)\psi + \cL \partial_{t_n}\psi
     = (\partial_{t_n}\cL)\psi + \cL \cB_n\psi
     = \partial_{t_n}(k\psi) = \cB_n\cL\psi
$$
which gives the following Lax form of the KP hierarchy
\begin{equation}\label{eq:LaxpaircL}
    \partial_{t_n}(\cL) = [\cB_n, \cL]\,.
\end{equation}
Here the left-hand side is a pseudo-differential operator of degree at most $-1$, so that the pseudo-differential operator $\cB_n$ is required to commute with~$\cL$ up to~$\partial_x^{-1}$, and a canonical choice of such a~$\cB_n$ is then~$\cB_n = (\cL^n)_{\ge 0}$, that is the differential part of the pseudo-differential operator $\cL^n$, as $\cL^n$ obviously formally commutes with $\cL$ to all orders. For example, we have $\cB_2 = (\cL^2)_{\ge 0} = \partial_x^2 + 2u_2$ and $\cB_3 = (\cL^3)_{\ge 0} = \partial_x^3 + 3u_2 + 3(u_{2, x} + u_3)$, and then the Lax representation \eqref{eq:LaxpaircL}  of the KP hierarchy implies
\begin{equation*}
    \left\{\begin{array}{l}
         u_{2, t_2} = u_{2, xx} + 3u_{3, x}\\
         2u_{2, t_3} = 3(u_{2, x} + u_3)_{t_2} - (u_{2, xx} - 3u_{3, x} + 3u_2^2)_x.
    \end{array}
    \right.
\end{equation*}
After denoting $y=t_2,t=t_3$ and eliminating~$u_3$ from the system, the function $u \coloneqq 2u_3$ satisfies the KP equation in the form \eqref{eq:KP}. When  $\psi$ is constructed using the Baker-Akhiezer function on an algebraic curve, these two theories are in fact equivalent to each other (see \cite{Segal-Wilson1985}). 
\end{remark}

\begin{remark}\label{rmk:Sato}
Let $\cE$ be the ring of pseudo-differential operators in $x$ with multiplication given by the generalized Leibniz rule \eqref{eq:partialmf}. Let $\cE x$ be the maximal left ideal of $\cE $ generated by $x \in \cE $. Then $\cE = \cE x \oplus \CC(\!(\partial^{-1})\!)$ where  $\CC(\!(\partial^{-1})\!)$ is the ring of pseudo-differential operators with constant coefficients, i.e., each element $P \in \cE$ can be uniquely written as $P = Qx + R$, where $Q \in \cE$, and $R \in \CC(\!(\partial^{-1})\!)$. The $\CC(\!(\partial^{-1})\!)$ part can be further identified with the ring $V \coloneqq \CC(\!(k^{-1})\!)$ by replacing $\partial$ by $k$. Then we have the natural direct sum decompositions of complex vector spaces and a morphism~$f$ between them:
\begin{equation}\label{eq:directsumE}
\begin{array}{rcl}
    f:\cE  = \cD  \oplus \cE^{(-1)} & \longrightarrow & \cE \slash \cE x \cong V = V_+ \oplus V_- \coloneqq \CC[k] \oplus k^{-1}\CC[\![k^{-1}]\!]\\
    P = Qx + R & \mapsto & R
\end{array}
\end{equation}
where $\cD \subset \cE$ is the subalgebra of differential operators and $\cE^{(-1)} \subset \cE$ is the subalgebra of pseudo-differential operators of order at most $-1$. As $\mathcal{D} = \mathcal{D}x \oplus \CC[\partial]$ and $\cE^{(-1)} = \cE^{(-1)}x \oplus \CC[\![\partial^{-1}]\!]$, the projection map $f$ is compatible with the two decompositions, that is we have $V_+=f(\cD)$ and $V_-=f(\cE^{(-1)})$. Note that as $V \cong \cE \slash \cE x$, there is a natural left $\cE$-module structure on $V$.

There are now several equivalent ways to define Sato's Universal Grassmannian Manifold (UGM). One way is to define it as the set of right $\cD $-submodules $\cF \subset \cE $ of $\cE $ such that
\begin{equation}\label{eq:Edirectsummanddecomp}
    \cE  = \cF \oplus \cE^{(-1)}.
\end{equation}
Keeping only the complex vector space structure of $\cE \slash \cE x$, Sato's UGM can be given in the following more familiar form: under the projection map $f$, the module $\cF$ can be identified  with the set of linear subspaces $U\subset V$ such that the map
\begin{equation}\label{eq:gammaU}
      \gamma_{U}: U \rightarrow {V} \slash {V}_- \cong {V}_+
\end{equation}
is Fredholm of index zero, i.e.~such that $\gamma_U$ has finite-dimensional kernel and cokernel, of the same dimension (in fact general Fredholm operators with nonzero index can be allowed, but they can all be identified with the index zero case by a proper shift of indices; for simplicity of exposition, we will only work with the index zero case). Roughly speaking, this means that Sato's UGM parametrizes all subspaces of $V$ that are comparable with $V_+$.

There are two more spaces that can be identified with Sato's UGM, up to an automorphism, which are relevant to KP theory.
\begin{enumerate}
    \item The group of monic $0$'th order pseudo-differential operators $\cW$ can be bijectively identified with Sato's UGM by sending $\cW \ni W \longmapsto W^{-1}\cD \in \textrm{UGM}$, that is, $\cF \in \cE$ as we described above is a cyclic right $\cD$ module generated by some $0$'th order pseudo-differential operator $W^{-1}$  (we will see shortly why the choice of $W^{-1}$ is made).
    \item For each pseudo-differential operator $\cL$ of the form \eqref{eq:pseudoL}, there exists a monic $0$'th order pseudo-differential operator $W$ such that
    \[
      \cL = W \circ \partial \circ W^{-1}.
    \]
    Such $W$ is unique up to multiplying it on the right by a monic $0$'th order pseudo-differential operator that commutes with $\partial$, that is by a monic $0$'th order pseudo-differential operator with constant coefficients. Thus the space of first order pseudo-differential operators~$\cL$ of the form \eqref{eq:pseudoL} can be identified with the quotient $\cW \slash \cW_c$, where $\cW_c$ is the space of monic $0$'th order pseudo-differential operators with constant coefficients.
\end{enumerate}
Thus each pseudo-differential operator $\cL$ in the form \eqref{eq:pseudoL} also defines a point in Sato's UGM. To make such an identification explicit we need to choose a lift of the Lax operator \eqref{eq:pseudoL} to a point in UGM, however the choice of such a lift would not affect the solutions of the KP hierarchy \eqref{eq:LaxpaircL}.

To provide some intuition for the construction presented here, one may draw an analogy to the finite-dimensional case. In this analogy, $\cL$~corresponds to a (diagonalizable) matrix, $\partial$~to the diagonal matrix of its eigenvalues, and~$W$ to the  matrix of eigenvectors of~$\cL$, which is uniquely determined only up to normalization.

Note that the operator~$W^{-1}$ conjugates~$\cL$ into $\partial$, and the corresponding linear systems \eqref{eq:eigenvalueproblemcL} and \eqref{eq:psideformBn} for $\partial$ can be easily solved by $\psi_0 = e^{\sum k^it_i}$ with $t_1 = x$. The wave function $\psi$ for $\cL$ can then be formally represented as $\psi = W\psi_0$, which is again unique up to multiplication by a series with constant coefficients of the form $1 + c_1k^{-1} + c_2k^{-2} + \cdots$ coming from the action of $\cW$ on $\psi_0$. The system \eqref{eq:psideformBn} then defines an infinite sequence of commuting vector fields on Sato's UGM under the constraints \eqref{eq:LaxpaircL} (see \cite{Sato1981}), which can be viewed as the action of elements $\ttt = (t_1, t_2, t_3, \cdots) \in \CC^{\infty} =: T$ of an additive torus on Sato's UGM.
    
The point of Sato's theory is that the nonlinear dynamical system \eqref{eq:LaxpaircL} is in fact linearized on Sato's UGM! To see this let's pretend that the Lie algebras $\cE , \cD , \cE^{(-1)}$ have corresponding Lie groups $G, G_+, G_-$, and pretend also that we have the group factorization $G = G_- \cdot G_+$ associated with the direct sum decomposition \eqref{eq:directsumE}, so that Sato's UGM can be identified with the homogeneous space $G/G_-$. Then we introduce the action of a multiplicative torus by integrating the action of the additive torus described above.

Formally, this means that we set $\exp\left(\sum \cL_0^it_i\right) \coloneqq g_-^{-1}(\ttt)g_+(\ttt) \in G_- \cdot G_+$, where $\cL_0$ is the initial value of the pseudo-differential operator $\cL$, i.e.~$\cL_0$ is obtained by setting all the time variables $\ttt = (t_1, t_2, t_3, \cdots)$ to zero, and $g_-(\mathbf{0}) = g_+(\mathbf{0}) = 1$. Then it can be checked directly that $\cL(\ttt) = g_-(\ttt)\, \cL_0\, g_-^{-1}(\ttt)$. As $\cL(\ttt) = g_-(\ttt)\,W(\mathbf{0})\, \partial W^{-1}(\mathbf{0}) g_-^{-1}(\ttt)$, a fixed choice of~$W(\mathbf{0})$ defines the map $\cL(\ttt) \longmapsto W(\ttt)^{-1}G_+ = W(\mathbf{0})^{-1}\,g_-^{-1}(\ttt)\,G_+ \in G/G_-$  from the space of first order pseudo-differential operators to Sato's UGM; despite the ambiguity of the choice of of $W(\mathbf{0})^{-1}G_+$ related to the normalization of $\psi$, note that $\cL$ is independent of these choices. The resulting  point in Sato's UGM can then be identified with
\begin{align*}
    W(\mathbf{0})^{-1}g_-^{-1}(\ttt)G_+ & = W(\mathbf{0})^{-1}\,g_-^{-1}(\ttt)\,g_+(\ttt)G_+ \\
    & = W(\mathbf{0})^{-1}\, \exp \left(\sum L_{\mathbf{0}}^it_i\right)\,G_+\\
    & = \exp \left(\sum \partial^it_i \right)\,W(\mathbf{0})^{-1}\,G_+ \in \text{UGM}.
\end{align*}
Under this identification, as~$\cL$ evolves according to \eqref{eq:LaxpaircL}, the initial point $\cL_{\mathbf{0}} \longmapsto W(\mathbf{0})^{-1}G_+$ on Sato's UGM evolves under the action of
$$\exp \left(\sum \partial^it_i \right)\,W(\mathbf{0})^{-1}\,G_+\,,$$ which is the simplest linear dynamics for the action of~$\partial_{t_n} = \partial^n$. To see that this is indeed linear, consider the simplest linear ODE $\dot{y} = ay$; its solution is given by $y = e^{at}y_0$ where $y_0 = y(0)$ is the initial value. The dynamics (evolution in $\ttt$) we are considering $W(\mathbf{0})^{-1}\,G_+ \longmapsto \exp \left(\sum \partial^it_i \right)\,W(\mathbf{0})^{-1}\,G_+  \in \text{UGM}$ is exactly of this form in multiple time variables. In this sense, the dynamics are linearized on Sato's UGM.
\end{remark}

\begin{remark}\label{rem:flows2}
There is also an  algebro-geometric interpretation of the KP flows on the Jacobian of the associated spectral curve. From \eqref{eq:potentialuintheta} we see that the KP flows are linearized on the Jacobian of the spectral curve: $U^{(i)}$'s in the argument $U^{(1)} x + U^{(2)} y + U^{(3)}t$ of the theta function are constant vectors which are independent of the time variables $\bf{t}$'s. Since the theta function is a section of a line bundle on the Jacobian, KP flows are linearized on the Jacobian. Now let's have a closer look what at the directions $U^{(n)}$. Let $\omega_j(z) = f_j(z)dz = \sum_{\ell = 0}^{\infty} \alpha_{\ell}^{(j)}z^{\ell}dz, 1 \le j \le g$ be the normalized basis of holomorphic differentials (see \eqref{eq:normalizeddifferential}), where $z = \frac{1}{k}$ is a local coordinate around~$p_{\infty}$. Note that in this local coordinate, the normalized abelian differential of the second kind $\Omega_j(p)$ has a singularity of the form
\[\Omega_j(p) \sim dk^j \sim -j\frac{dz}{z^{j + 1}}, \qquad \text{as } p \to p_{\infty}.\]
Then it is a standard fact that~$\Omega_j(p)$ satisfies the following relation with the holomorphic differentials (this can be proved using Riemann's bilinear relations, see \cite[p.~69, Eq.~(3.8.2)]{Farkas-Kra1992})
\begin{equation}\label{eq:12abeliandifferential}
    U^{(n)}_j = \int_{B_j}\Omega_n = -2\pi i \alpha_{n-1}^{(j)}, \qquad 1 \le j \le g,\ n \ge 1.    
\end{equation}
By definition we have $\alpha_{n}^{(j)} = \frac{f_j^{(n)}(0)}{n !}$.
The Abel-Jacobi map~$AJ$ is given by integrating holomorphic differentials, and thus its derivative is the vector of values of the holomorphic differentials, and thus the vector $\vec{f}(z) = (f_1(z), f_2(z), \dots, f_g(z))$ is a tangent vector to the Abel-Jacobi image of the spectral curve in its Jacobian (see~\eqref{eq:Abel-Jacobimap}). From \eqref{eq:12abeliandifferential} we thus compute
\begin{equation}\label{eq:12abeliandifferential2}
    U^{(n)} = -\frac{2\pi i}{(n-1)!}\vec{f}^{(n-1)}(0) = -\left.\frac{2\pi i}{(n-1)!}\frac{d^n\text{AJ}}{dz^n}\right\vert_{z = 0}\,.
\end{equation}
to be the higher derivatives of the Abel-Jacobi map.
\end{remark}

\begin{remark}\label{rmk:UGM}
Combining the last two remarks we see that Jacobian varieties of algebraic curves are nothing but some finite-dimensional orbits of KP hierarchy in Sato's UGM. So one strategy of proving that Jacobian varieties are characterized by the property of their theta function satisfying the KP hierarchy (the stronger Novikov's conjecture was that just one KP equation sufficed) is to show that no abelian varieties other than Jacobians can appear as finite-dimensional orbits of Sato's UGM, and this is the strategy adopted in \cite{Mulase1984}. The spectral curve arises naturally as the spectrum of a commutative subalgebra of the ring of differential operators $\cD $. This is what Mulase explains in \cite{mulase}, and also explained to us in personal communication.
\end{remark}

Formula \eqref{eq:potentialuintheta} is valid for any curve $C$, for any parameters $x, y, t \in \CC$, and for any point $p_0 \in C$. For any such data the formula is the same, and the theta function intrinsically knows which curve we are taking. The discussion in \Cref{sec:CommDO} was completely formal, and we can now realize that it was modeled on a neighborhood of a point on a Riemann surface where the Baker-Akhiezer function has the prescribed essential singularity. Since we are on a curve, the previous formal computation and the uniqueness of Baker-Akhiezer function tell us that we actually have constructed something geometric. That is what happened so far.

\medskip
Finally, to see the connection between the KP equation and equations obtained in \Cref{sec:geometryofKummer} by degenerating the trisecants, we have \footnote{This is the promised solution to \Cref{ex:getKP}}

\begin{proposition}
KP equation \eqref{eq:KP} with $u$ as in \eqref{eq:potentialuintheta} is equivalent to 
\begin{equation}\label{eq:equationsontheta}
\partial_U^4\Theta[\ve] - 4\partial_U \partial_W \Theta[\ve] + 3 \partial_V^2\Theta[\ve] + \const \Theta[\ve] = 0\qquad\forall \ve \in \frac{1}{2}(\ZZ\slash 2\ZZ)^g \,.
\end{equation}
\end{proposition}
Here $\Theta[\ve] = \Theta[\ve](0)$, where $\Theta[\ve](z)$ is defined in \eqref{eq:Thetavarepsilon}.

\begin{proof}
This is again a direct computation. After substituting $u$ of the form \eqref{eq:potentialuintheta} into \eqref{eq:KP}, we obtain 
\[-2\partial_x^2\left[\frac{3\theta_{yy}\theta - 3 \theta_y^2 + 4\theta_x \theta_t - 4 \theta_{xt} \theta +  \theta_{xxxx}\theta - 4 \theta_{xxx}\theta_x + 3\theta_{xx}^2}{\theta^2}\right] = 0\,,\]
(where from now on we use the subscript to denote the partial derivative in the corresponding variable). Assuming that the abelian variety is indecomposable, so that $\Theta$ is irreducible, this implies the equation
\begin{equation} \label{eq:thetaequations}
\theta_{xxxx}\theta - 4 \theta_{xxx}\theta_x + 3\theta_{xx}^2 + 4\theta_x \theta_t - 4 \theta_{xt} \theta + 3\theta_{yy}\theta - 3 \theta_y^2 + 8c \theta^2 = 0.
\end{equation}
for any $z \in A$. Note that \eqref{eq:thetaequations} is a quadratic equation for $\theta(z)$, while the equation \eqref{eq:equationsontheta} we want to show is a linear equation on $\Theta[\varepsilon](z)$, so to obtain it we would better be able to pull out a common factor. To do that we first write equation \eqref{eq:thetaequations} in the bilinear form. We will use $z_1, z_2$ to denote the independent variables in the two factors of the bilinear form, and we will see that the correct variables to be used to decouple the bilinear form would be $w_1 = \frac{1}{2}(z_1 + z_2), w_2 = \frac{1}{2}(z_1 - z_2)$.

We introduce the directional derivative operators
\[
\left\{\begin{array}{ll}
D_{1, z_1} \coloneqq \sum \limits_{j = 1}^g U_j \frac{\partial}{\partial z_{1j}}, \qquad & D_{1, z_2} \coloneqq \sum \limits_{j = 1}^g U_j \frac{\partial}{\partial z_{2j}} \\
D_{2, z_1} \coloneqq \sum \limits_{j = 1}^g V_j \frac{\partial}{\partial z_{1j}}, \qquad & D_{2, z_2} \coloneqq \sum \limits_{j = 1}^g V_j \frac{\partial}{\partial z_{2j}} \\
D_{3, z_1} \coloneqq \sum \limits_{j = 1}^g W_j \frac{\partial}{\partial z_{1j}}, \qquad & D_{3, z_2} \coloneqq \sum \limits_{j = 1}^g W_j \frac{\partial}{\partial z_{2j}} 
\end{array}\right.
\]
and, by analogy, operators $D_{1, w_1}, D_{1, w_2}, D_{2, w_1}, D_{2, w_2}, D_{3, w_1}, D_{3, w_2}$, where all $z$ are substituted for $w$, and they are connected by the relations
$$
D_{1, z_1} = \tfrac{1}{2}(D_{1, w_1} + D_{1, w_2}),\  D_{2, z_1} = \tfrac{1}{2}(D_{2, w_1} + D_{2, w_2}), \ D_{3, z_1} = \tfrac{1}{2}(D_{3, w_1} + D_{3, w_2})\,,
$$
$$
D_{1, z_2} = \tfrac{1}{2}(D_{1, w_1} - D_{1, w_2}),\  D_{2, z_2} = \tfrac{1}{2}(D_{2, w_1} - D_{2, w_2}),\   D_{3, z_2} = \tfrac{1}{2}(D_{3, w_1} - D_{3, w_2})\,. 
$$
Then \eqref{eq:thetaequations} can be written in the form
\begin{equation}\label{eq:thetaequations2}
\begin{aligned}
\Big[D_{1, z_1}^4 - 4D_{1, z_1}^3D_{1, z_2} &+ 3D_{1, z_1}^2 D_{1, z_2}^2 + 4D_{1, z_1}D_{3, z_2} - 4D_{1, z_1}D_{3, z_1}\\ &+ 3D_{2, z_1}^2 - 3 D_{2, z_1}D_{2, z_2} + 8c\Big]\cdot\theta(z_1)\theta(z_2) \vert_{z_1 = z_2} = 0.
\end{aligned}
\end{equation}
We now express the operator in the square brackets in terms of the $D_{1, w_i}, D_{2, w_i}$ and $D_{3, w_i}$ and apply it to the right hand side of Riemann's addition theorem \eqref{eq:isogeny} in the following form
\[\theta(z_1)\theta(z_2) = \sum \limits_{\ve \in \frac{1}{2}(\ZZ_2)^g} \Theta[\ve]\left(\tfrac{z_1 + z_2}{2}\right) \Theta[\ve] \left(\tfrac{z_1 - z_2}{2}\right)\,,\]
for $w_1 = z, w_2 = 0$. Since $\Theta[\ve, 0](w)$ is even, it suffices to leave the even powers of $D_{1, w_2}, D_{2, w_2}$ and $D_{3, w_2}$ in the resulting expression. We obtain
\[\left[D_{1, w_2}^4 - 4D_{1, w_2}D_{3, w_2} + 3 D_{2, w_2}^2 + 16c\right] \left.\sum \limits_{\varepsilon \in \frac{1}{2}(\ZZ /  2\ZZ)^g} \Theta[\ve](w_1)\Theta[\ve](w_2) \right\vert_{\substack{w_1 = z \\ w_2 = 0}} = 0\,.\]
Note that the $2^g$ functions $\Theta[\ve](z), \ve \in \frac{1}{2}(\ZZ /  2\ZZ)^g$, are linearly independent (they form a basis in the space of sections of $2\Theta$). Equating the coefficients of these functions to zero we obtain the system \eqref{eq:equationsontheta}. This finishes the proof.
\end{proof}

\subsubsection*{Summary of this section} 
Starting from the data of a complex curve with a marked point and a local coordinate at that point, we have reviewed Krichever's construction of Baker-Akhiezer functions. These functions are functions on a curve with one controlled essential singularity, and they naturally provide solutions of partial differential equations. It turns out that these functions provide solutions to the KP hierarchy, which relates to flex lines of the Kummer variety studied in the previous section, and also to the questions on pairs of commuting differential operators from the first section.

\section{Krichever's Proof of Welters' Trisecant Conjecture in the Flex Line Case}\label{sec:flex}

We are heading towards Krichever's proof of the flex line case of Welters' conjecture, that is for the most degenerate trisecant of the Kummer variety, which is a line in $\CC\PP^{2^g-1}$ tangent to $\Kum(A)$ at one point, with multiplicity three.

Our presentation follows most closely the discussion in \cite{Krichever2006}, though we provide much more detail.

\subsection{The analytic meaning of a flex line}
First recall that by \eqref{eq:Kummerflex1} the existence of a flex line of a Kummer variety at a point $q \in A\setminus A[2]$ is equivalent to the existence of vectors $V, U \in \CC^g$, and $p, E \in \CC$ such that the following differential equation holds
\begin{equation}\label{eq:K}
(\partial_V + \partial^2_U - 2p \partial_U + (p^2 - E)) \cdot \Kum(q \slash 2) = 0.
\end{equation}
So  the existence of a flex line of the Kummer variety is a second order differential equation on the Kummer image. Indeed, given a tangent vector to the Kummer image at some point, for it to be tangent to the Kummer variety with multiplicity three, there is a condition on  second derivatives, and \eqref{eq:K} is this condition! This is one side of the story. 

On the other hand, the Baker-Akhiezer function
\begin{equation}
\begin{aligned}
    \psi(x, y, t, p) &= \frac{\theta(\int_{p_{\infty}}^p \omega + U^{(1)}x + U^{(2)}y + U^{(3)}t + D)\theta(D)}{\theta(\int_{p_{\infty}}^{p}\omega + D)\theta(U^{(1)}x + U^{(2)}y + U^{(3)}t + D)}\\ &\cdot \exp\left( \left(\int_{p_0}^{p}\Omega_1(p) - b_1\right)x + \left(\int_{p_0}^p\Omega_2(p) - b_2\right)y + \left(\int_{p_0}^p\Omega_3(p) - b_3\right)t\right)\,,
\end{aligned}
\end{equation}
satisfies partial differential equations
\begin{equation}\label{eq:KPZSpair}
    L_1\psi = \frac{\partial \psi}{\partial y}, \qquad L_2\psi = \frac{\partial \psi}{\partial t}\,
\end{equation}
for
\[L_1 = \partial_x^2 + u, \qquad L_2 = \partial_x^3 + \frac{3}{2}u\partial_x + w\,,\]
and furthermore the commutativity
\[[L_1 - \tfrac{\partial}{\partial y}, L_2 - \tfrac{\partial}{\partial t}] = 0\,,\]
is equivalent to the  KP equation. Novikov's \Cref{conj:Novikov}, proven by Shiota, states that if the theta function satisfies the KP equation, then the abelian variety is the Jacobian of a curve. Gunning's \Cref{thm:Gunning'sthm1} and Welters' \Cref{thm:Weltersthm1} show that if a Kummer variety has a family of trisecant lines or a family of flex lines, then it is a Jacobian. We want to characterize Jacobians by the existence of one flex, which we would like to write as a differential equation
\begin{equation}\label{eq:L}
(\partial_x^2 + u(x, y))\psi = \partial_y \psi
\end{equation}
for $u$ given by equation \eqref{eq:potentialuintheta}, expressed in terms of the theta function. To deduce the characterization by one flex line, one could try to reduce to Novikov's conjecture by showing that if the first equation in~\eqref{eq:KPZSpair} is satisfied, 
then the second equation in~\eqref{eq:KPZSpair} is also satisfied. This is not how Krichever's proof goes, though, as it does not use Shiota's result directly. Krichever's proof, which we follow closely, involves only equation \eqref{eq:L} with
\begin{equation} \label{eq:psiintheta}
\psi = \frac{\theta(Ux + Vy + D + Z)}{\theta(Ux + Vy + Z)} e^{px + Ey}.
\end{equation}

\begin{theorem}[Krichever~2006~\cite{Krichever2006}]\label{thm:Kricheverflex}
An indecomposable principally polarized abelian variety $(A, \Theta)$ is the Jacobian of an algebraic curve if and only if there exist $g$-dimension vectors $U \ne 0, V, D$ such that equation \eqref{eq:L} (equivalently \eqref{eq:K}) is satisfied.
\end{theorem}

\begin{remark}
Note one important caveat here, in that we explicitly exclude the case of decomposable abelian varieties, that is products of lower dimensional abelian varieties. Indeed, if $(A,\Theta)=(A',\Theta')\times (A'',\Theta'')$, then the theta divisor $\Theta=(\Theta'\times A'')\cup (A'\times\Theta'')$ is reducible, the theta function is the product $\theta(\tau,z)=\theta'(\tau',z')\cdot\theta(\tau'',z'')$, and the singular locus of the theta divisor $\Sing\Theta$ contains a $(g-2)$-dimensional subvariety $\Theta'\times\Theta''$ (in fact, this is a characteristic property of the locus of decomposable abelian varieties, see \cite{Ein-Lazarsfeld1997}). This allows us to easily construct decomposable counterexamples to the theorem above and to many statements that we will have, by taking a product of a Jacobian of a curve and an abelian variety that is not a Jacobian, and making sure that all the differentiation happens only in the direction of the Jacobian. The way the indecomposability assumption shows up in the proof is in using the classical fact that the theta divisor~$\Theta$ of an indecomposable principally polarized abelian variety~$A$ is irreducible (see eg.~\cite[p.~86, Lem.~10]{Igusa1972}, so that in particular its intersection with any other irreducible $(g-1)$-dimensional subvariety of~$A$ has dimension~$g-2$.
\end{remark}

Here is an observation due to Krichever which he always thought was of crucial importance, while underappreciated. This is a computational lemma, and it says \eqref{eq:L} is equivalent to \eqref{eq:K} and they imply the following

\begin{lemma}
\eqref{eq:L} $\Longleftrightarrow$ \eqref{eq:K} and they imply that the equation
\begin{equation}\label{eq:T}
\begin{aligned}
\left((D_2\theta)^2 - 2(D_1^2\theta)^2\right)D_1^2\theta - 2\left(D_1^2\theta D_1^3\theta + D_2\theta D_1D_2\theta\right)D_1\theta\\ + \left(D_2^2\theta - 2D_1^4\theta\right)(D_1\theta)^2 = 0 
\end{aligned}
\end{equation}
holds on the theta divisor, where $D_1\theta \coloneqq D_U\theta$ and $D_2\theta \coloneqq D_V\theta$ are the partial derivatives of the theta function with respect to some vectors $U, V \in \CC^g$.
\end{lemma}

\begin{remark}
Equation \eqref{eq:T} is valid on the theta divisor, that is for~$z$ such that~$\theta(z)=0$, and there is no easy way to go from \eqref{eq:T} to \eqref{eq:K} in principle, that is there is no easy way to get from an equation on theta divisor to an equation for the theta function globally. Similarly, there is no easy way at the level of equations to go from \eqref{eq:T} to flexes. Krichever's proof of the flex line case of the Welters' conjecture, however, uses equation \eqref{eq:T}, and Igor always believed that \eqref{eq:T} was deeper and more conceptual than the existence of a geometric flex of the Kummer variety.
\end{remark}

\begin{proof}
The equivalence of \eqref{eq:K} and \eqref{eq:L} follows from Riemann's bilinear relations (this is a review exercise). 

Now let $u$ and $\psi$ be as in \eqref{eq:potentialuintheta} and \eqref{eq:psiintheta}. Consider a general point $\tilde{x}U + yV + Z$ in the theta divisor $\Theta$, and expand in $x-\tilde x$. We thus obtain
\begin{align*}
\psi & = \frac{\alpha}{x - \tilde{x}} + \beta + \gamma\cdot (x - \tilde{x}) + \delta\cdot (x - \tilde{x})^2 + \dots,\\
u & = -\frac{2}{(x - \tilde{x})^2} + v + w\cdot(x - \tilde{x}) + \dots.
\end{align*}
We view $\tilde{x}, \alpha, \beta, \gamma, \delta, v, w$, as functions of $y$, and we may assume that $\alpha \ne 0$. Equating the coefficients of $(x - \tilde{x})^i$ for $i = -2, -1, 0$ in equation \eqref{eq:L} we get
\begin{align*}
 \alpha \dot{\tilde{x}} + 2 \beta & = 0\,,\\
-\dot{\alpha} + \alpha v - 2 \gamma & = 0\,,\\
- \dot \beta + \gamma \dot{\tilde{x}} + \beta v + \alpha w & = 0\,.
\end{align*}
Here and in the following we will use the dot to denote the derivatives with respect to $y$, and will use the prime to denote the derivatives with respect to $x$. Taking the $y$-derivative of the first equation and using all three equations to simplify yields
\begin{equation}\label{eq:ddottildex}
\alpha\cdot (\ddot{\tilde{x}} + 2w) = 0 \Rightarrow \ddot{\tilde{x}} = -2w \quad (\hbox{since }\alpha \ne 0).
\end{equation}
We compute $w$ by recalling the expression of $u$ from \eqref{eq:potentialuintheta}, and compute $\ddot{\tilde{x}}$ by using that $\theta\left(\tilde{x}(y)U + yV + Z\right) \equiv 0$ is identically zero. This yields the equation
\begin{equation} \label{eq:equationsonTheta}
\begin{aligned}
& -D_1^2\theta (D_2\theta)^2 + 2 D_1\theta D_2\theta D_1D_2\theta - (D_1\theta)^2 D_2^2\theta\\ = &  -2(D_1\theta)^2 D_1^4\theta -2 D_1\theta D_1^2\theta D_1^3\theta -2 (D_1^2\theta)^3
\end{aligned}
\end{equation}
on the theta divisor, which is of course just another form of \eqref{eq:T}.

The detailed computation goes as follows. We first expand
\[u = 2 \partial^2_x\ln \theta(xU + yV + Z) = -\frac{2}{(x - \tilde{x})^2} + v + w\cdot(x - \tilde{x}) + \dots\,,\]
so $w$ is the constant term of $2D^3_1\ln\theta \vert_{x = \tilde{x}}$, but then we also compute
\begin{align*}
& \theta = D_1\theta \cdot (x - \tilde{x}) + \frac{D_1^2\theta}{2} \cdot (x - \tilde{x})^2 + \frac{D_1^3\theta}{3!} \cdot (x - \tilde{x})^3 + \dots,\\
& \theta' = D_1\theta + D_1^2\theta \cdot (x - \tilde{x}) + \frac{D_1^3\theta}{2} \cdot (x - \tilde{x})^2 + \frac{D_1^4\theta}{3!} \cdot (x - \tilde{x})^3 + \dots,\\
& \theta'' = D_1^2\theta + D_1^3\theta \cdot (x - \tilde{x}) + \frac{D_1^4\theta}{2} \cdot (x - \tilde{x})^2 + \frac{D_1^5\theta}{3!} \cdot (x - \tilde{x})^3 + \dots,\\
& \theta''' = D_1^3\theta + D_1^4\theta \cdot (x - \tilde{x}) + \frac{D_1^5\theta}{2} \cdot (x - \tilde{x})^2 + \frac{D_1^6\theta}{3!} \cdot (x - \tilde{x})^3 + \dots,
\end{align*}
and
\[2D_1^3\ln \theta = 2 \frac{\theta^2\theta''' - 3 \theta\theta'\theta'' + 2(\theta')^3}{\theta^3}\,.\]
As $\theta^3 = (D_1\theta \cdot (x - \tilde{x}))^3 + o((x - \tilde{x})^3)$, we have by definition
\begin{align*}
w = & \text{constant term of $2 \tfrac{\theta^2\theta''' - 3 \theta\theta'\theta'' + 2(\theta')^3}{\theta^3}$ at $x = \tilde{x}$}\\
 = & \tfrac{2[(D_1\theta)^2 D_1^4\theta + 2(D_1\theta \cdot \frac{1}{2}D_1^2\theta) D_1^3\theta]}{(D_1\theta)^3} \\
& -\tfrac{6[D_1\theta(D_1\theta \cdot \frac{1}{2}D_1^4\theta + D_1^2\theta D_1^3\theta + \frac{1}{2}D_1^2\theta D_1^3 \theta) + \frac{1}{2}D_1^2\theta(D_1\theta D_1^3 \theta + (D_1^2\theta)^2) + \frac{1}{6}D_1^3\theta D_1\theta D_1^2\theta]}{(D_1\theta)^3} \\
& + \tfrac{4[3 \cdot \frac{1}{6}D_1^4\theta  (D_1\theta)^2 + 6 \cdot D_1\theta  D_1^2 \theta \cdot \frac{1}{2} D_1^3 \theta + (D_1^2\theta)^3]}{(D_1\theta)^3}\\
= & \tfrac{(D_1\theta)^2 D_1^4\theta + D_1\theta D_1^2\theta D_1^3\theta + (D_1^2\theta)^3}{(D_1\theta)^3}\,.
\end{align*}
To calculate $\ddot{\tilde{x}}$ we use $\theta\left(\tilde{x}(y)U + yV + Z\right) \equiv 0$, from which by differentiating we obtain
\begin{align*}
& D_1\theta \cdot \dot{\tilde{x}} + D_2\theta = 0,\\
& (D_1^2 \theta \cdot \dot{\tilde{x}} + D_2D_1\theta) \cdot \dot{\tilde{x}} + D_1\theta \cdot \ddot{\tilde{x}} + D_1D_2\theta \cdot \dot{\tilde{x}} + D_2^2\theta = 0.
\end{align*}
Combining the above two equalities yields
\[\ddot{\tilde{x}} = \frac{-D_1^2\theta (D_2\theta)^2 + 2 D_1\theta D_2\theta D_1D_2\theta - (D_1\theta)^2 D_2^2\theta}{(D_1\theta)^3}\,.\]
Thus using \eqref{eq:ddottildex} finally gives
\begin{align*}
& -D_1^2\theta (D_2\theta)^2 + 2 D_1\theta D_2\theta D_1D_2\theta - (D_1\theta)^2 D_2^2\theta\\ = & -2(D_1\theta)^2 D_1^4\theta -2 D_1\theta D_1^2\theta D_1^3\theta -2 (D_1^2\theta)^3\,.
\end{align*}
\end{proof}

\subsection{Formal solutions of \eqref{eq:L} with simple poles}
Krichever's strategy of the proof is as follows: think about the equation in the form \eqref{eq:K}, which is a partial differential equation for theta functions. If we somehow had another differential operator $L_2$ commuting with $L_1$, then we could try to construct a spectral curve. The moment we have a spectral curve, then hopefully the abelian variety is going to be the Jacobian of that spectral curve. There are certain things to check, for example we should check that the spectral curve is smooth etc., but that is mostly a matter of technical computations for experts. So morally the moment we have constructed another differential operator commuting with~$L_1$, we should be able to declare victory.

The main idea of Shiota's proof of the Novikov conjecture is to show that if $u$ is as in \eqref{eq:potentialuintheta} and satisfies the KP equation, then it can be extended to a $\tau$-function of the KP hierarchy, as a global holomorphic function of the infinite number of variables $\ttt = \{t_i\}$. Local existence directly follows from the KP equation, and then crucially one needs global existence of the $\tau$-function. The core of the problem is that a priori there is a homological obstruction to the global existence of~$\tau$. Let the subvariety $\Theta_1\subset A$ be defined by the equations $\Theta_1 \coloneqq \{z\in A : \theta(z) = D_1\theta(z) = 0\}$. As we will see, the reason for singling out this locus is that there are nonzero solutions of the KP equation that vanish on $\Theta_1$ and cannot be extended uniquely to solutions of the KP hierarchy (see \cite[p.~359]{Shiota1986}, where Shiota called such solutions anomalous). The $D_1$-invariant subset $\Sigma$ of $\Theta_1$ will be called the singular locus. Then the obstruction is controlled by the cohomology group $H^1(\CC^g \setminus \Sigma, \cV)$, where $\cV$ is the sheaf of $D_1$-invariant meromorphic functions on $\CC^g \setminus\Sigma$ with poles along $\Theta$. The hardest part of Shiota's work is the proof that the locus $\Sigma$ is empty, which ensures the vanishing of $H^1(\CC^g, \cV)$. 

In implementing the above, it will be crucial to recall that we are working with an indecomposable abelian variety. Indeed, if the abelian variety were decomposable as a product $(A',\Theta')\times (A'',\Theta'')$, and $D_1$ were a partial derivative only in the direction of the first factor, then $A'\times\Theta''\subset \Theta_1$ would be a $(g-1)$-dimensional component of the locus~$\Theta_1$, invalidating many of the constructions below.

For indecomposable principally polarized abelian varieties, Welters' conjecture in the flex line case is then transformed into the following statement whose proof will be our main goal in the rest of these lecture notes.
\begin{theorem}\label{thm:flextheta}
An indecomposable principally polarized abelian variety $(A, \Theta)$ is the Jacobian of a curve if and only if there exist $0\ne U, V\in\CC^g$, such that for each $Z \in \CC^g \setminus \Sigma$ equation \eqref{eq:ddottildex} for the function $\tau(x, y) = \theta(Ux + Vy + Z)$ is satisfied, i.e.~\eqref{eq:T} is valid on the theta divisor~$\Theta$.
\end{theorem}

Since our starting point is equation \eqref{eq:L}, let's consider abstractly the properties of solutions of this equation. We assume that the potential $u(x, y)$ has the form $u(x, y) = 2 \partial_x^2 \ln \tau(x, y)$, where $\tau(x, y)$ is some entire function of the complex variable $x$ smoothly depending on a parameter $y$. Here we do not assume that~$\tau(x, y)$ has anything to do with the theta function to begin with, morally we are thinking of~$\tau$ as a function of~$x$ with parameter~$y$, and represent $\tau$ as an infinite product
\[\tau(x, y) = c(y) \prod \limits_i (x - x_i(y))\,,\]
where $x_i(y)$ are the individual values of $x$ for which $\tau$ is zero --- these are functions of~$y$. The point of this assumption is that $u(x, y)$ then has the following expansion around a simple zero~$x_i(y)$ of $\tau(x, y)$:
\begin{equation}\label{eq:uxiexpansion}
u(x, y) = \frac{-2}{(x - x_i(y))^2} + v_i(y) + w_i(y)(x - x_i(y)) + \dots
\end{equation}
Crucially, note that the residue of $u$ at $\tilde{x} = x_i(y)$ is automatically zero because $u$ is the $x$ derivative of a meromorphic function.
Then the corresponding solution of equation \eqref{eq:L} must have the form
\begin{equation}\label{eq:psixiexpansion}
\psi = \frac{\alpha}{x - \tilde{x}} + \beta + \gamma (x - \tilde{x}) + \delta (x - \tilde{x})^2 + \dots.
\end{equation}
We want to consider a formal solution of \eqref{eq:L} of the form
\begin{equation}\label{eq:wavepsixyk}
\psi(x, y, k) = e^{kx + (k^2 + b)y} \left(1 + \sum \limits_{s = 1}^{\infty} \xi_s(x, y) k^{-s}\right).
\end{equation}
(Such solutions are typically called `wave solutions' in the integrable systems literature, but we will avoid this terminology here, as we do not explain the motivation).
Then we will show that the assumptions in \Cref{thm:flextheta} are necessary and sufficient for the local existence of solutions of \eqref{eq:L} such that
\[\xi_s(x,y) = \frac{\tau_s(Ux + Vy + Z, y)}{\theta(Ux + Vy + Z)}, \qquad Z \not\in \Sigma\,,\]
where $\tau_s(Z, y)$, as a function of $Z$, is holomorphic in some open domain in $\CC^g$. In order for $\psi$ given by \eqref{eq:wavepsixyk} to satisfy \eqref{eq:L}, the functions $\xi_s$ must be defined recursively by the equation
\begin{equation}\label{eq:xisrecursion}
2\partial_x \xi_{s + 1} = \partial_y \xi_s - \partial_x^2 \xi_s - (u - b)\xi_s = \partial_y \xi_s - \partial_x^2 \xi_s + 2(\partial_x\xi_1)\xi_s\,.
\end{equation}
Therefore, the global existence of $\xi_s$, which is equivalent to the global existence of the function $\tau_s(Z, y)$, is controlled by the cohomology group $H^1(\CC^g \setminus \Sigma, \cV)$ as discussed above. Solutions of the recursion relations \eqref{eq:xisrecursion} involve some arbitrariness. At the local level the main problem is then to find a translation invariant normalization of $\xi_s$ such that the resulting wave solution $\psi$ given by \eqref{eq:wavepsixyk} is unique up to a $D_1$-invariant factor.

We now proceed to implement this outline of the proof. Substitute $u(x, y)$ into \eqref{eq:L}, and note that in our derivation of\eqref{eq:ddottildex}, we only used the fact that $u$ and $\psi$ have the form \eqref{eq:uxiexpansion} and \eqref{eq:psixiexpansion} respectively and that they are solutions of \eqref{eq:L}. So we still have
\begin{equation}\label{eq:D}
\ddot{\tilde{x}} = -2w\,,
\end{equation}
where $w$ is one of the coefficients in the expansion of $u$. This is an equation for every zero of a function in two variables. We now prove the following 

\begin{lemma}\label{lem:xisimplepole}
Suppose that equations \eqref{eq:D} for the zeros of $\tau(x, y)$ hold. Then there exist meromorphic solutions of \eqref{eq:L} that have simple poles at $x_i$ and are holomorphic everywhere else.
\end{lemma}

This is a very technical looking statement. But surprisingly enough, it's a key to everything in the proof. Let us explain what happens here. We have  \eqref{eq:L} where $u$ has a second order pole at $x=x_i(y)$, so  $\psi$ is going to have a first order pole there. We want to know that we can write the expansions of $u$ and $\psi$ and that this all works infinitely far in the power series, crucially with {\em only first order poles}. Typically what might have happened is that we have $\psi$, which is an infinite series with coefficients $\xi_s$ that are obtained recursively by solving differential equations. Note that if $\xi_s$ has a pole of order $a$, and if we differentiate it with respect to $x$, then the order of pole is typically going to increase to $a+1$. So one would expect that when one tries to solve for~$\xi_s$ recursively, they would need to have higher and higher order poles. For that not to happen requires some miraculous cancellations, so that the pole orders stay bounded. This will eventually allow us to stay in some finite-dimensional space, and obtain solutions there.

\begin{proof}[Proof of \Cref{lem:xisimplepole}]
The proof is a direct computation, verifying that the recursive expressions for the coefficients of higher order poles of $\xi_s$ do indeed cancel at every step in the recursive derivation. Indeed, substituting \eqref{eq:wavepsixyk} into \eqref{eq:L} yields a recursive system of equations
\begin{equation}\label{eq:xisrecurrencesystem}
2\xi'_{s + 1} = \partial_y\xi_s - (u-b)\xi_s - \xi''_s\,. 
\end{equation}
We will prove by induction that this system has meromorphic solutions with simple poles at all the zeros $x_i$ of $\tau$. For that let us expand $\xi_s$ (for a fixed given~$s$) around $x_i$ as
\[\xi_s = \frac{r_{s,-1}}{x - x_i} + r_{s,0} + r_{s,1}(x - x_i) + \dots.\]
Suppose that $\xi_s$ are defined and that equation \eqref{eq:xisrecurrencesystem} has a meromorphic solution. What would then prevent $\xi_{s + 1}$ from having the same form? We need to make sure that (1) $\xi_{s + 1}$ would not have higher order poles; and (2) $\xi_{s + 1}$ does not have a log term, which means that if the right-hand side of \eqref{eq:xisrecurrencesystem} has zero residue at $x = x_i$, i.e.~if
\[\Res_{x_i}\left(\partial_y \xi_s - (u-b) \xi_s - \xi''_s\right) = \dot{r}_{s,-1} - (v_i - b)r_{s,-1} + 2r_{s,1} = 0\,,\]
then the residue of the next equation also vanishes.

For (1) we note that only the second term $-(u-b)\xi_s$ and third term $-\xi''_s$ on the right-hand side of \eqref{eq:xisrecurrencesystem} would have a third order pole, and it is easy to check that these third order poles precisely cancel each other. Thus $\xi_{s+1}'$ has at most a second order pole, and thus $\xi_{s+1}$ has at most a simple pole.

To prove (2) we proceed as follows. From \eqref{eq:xisrecurrencesystem} it follows that the coefficients of the Laurent expansion for $\xi_{s + 1}$ are equal to
\begin{align*}
2r_{s + 1,-1} & = -\dot{x}_i r_{s,-1} -2r_{s, 0}\,,\\
2r_{s+1, 1} & = \dot{r}_{s,0} - \dot{x}_ir_{s,1} - w_ir_{s,-1} - (v_i - b)r_{s, 0}\,.
\end{align*}
These equations imply
\[
\begin{aligned}\dot{r}_{s + 1, -1} &- (v_i - b)r_{s + 1,-1} + 2r_{s+1, 1}\\ = &-\frac{1}{2}(\ddot{x}_i + 2w_i)r_{s,-1} - \frac{1}{2}\dot{x}_i[\dot{r}_{s,-1} - (v_i - b)r_{s,-1} + 2r_{s,1}] = 0\,,\end{aligned}\]
and thus the next term on the right-hand side of \eqref{eq:xisrecurrencesystem} has no residue, and the lemma is proved.
\end{proof}

Recall that our secret goal is to construct a spectral curve by finding another differential operator $L_2$ so that $\psi$ is also its eigenfunction ----- so that $L_1$ would commute with~$L_2$. Not every solution of equation \eqref{eq:L} arises from an algebraic curve, but we know that Baker-Akhizer functions have the desired property! So somehow we need to impose some extra conditions to normalize solutions of equation \eqref{eq:L} such that under these conditions the normalized solution is more or less unique. As the formal power series part of $\psi(x, y, k)$ is expected to come from theta functions, it is natural to impose some quasi-periodicity conditions on the coefficients $\xi_s$.

\subsection{$\lambda$-periodic solutions of \eqref{eq:L}}
In the next step we are going to fix a translation invariant normalization of $\xi_s$, called $\lambda$-periodicity, which will then define solutions of \eqref{eq:L} uniquely up to multiplication by an $x$-independent factor; these will be the so-called $\lambda$-periodic (wave) solutions of~\eqref{eq:L}. As mentioned before, direction~$U$ is most important, for possible existence of anomalous solutions. We will thus require~$\psi$ to be invariant under translation by~$U$, calling such solutions $\lambda$-periodic.
We will then construct $L_2$ for which the $\lambda$-periodic solutions will be eigenfunctions, and while these $L_2$ will only be locally defined, by the suitable uniqueness of normalization they will turn out to be independent of the $x$-independent multiplicative factor, as we will eventually see. 

Let $A_U \coloneqq \langle Uz\rangle \subseteq A$ be the Zariski (or analytic) closure in $A$ of the set $\{Uz\}_{z \in \CC}$. In other words, since we are taking the closure of an abelian subgroup, this is the minimal abelian subvariety of $A$ containing $\{Uz\}$, which a priori can be of arbitrary dimension. Morally we want ``$\psi$ to be a solution on the abelian subvariety $A_U$'', though this is not exactly what we are going to be able to say. Note that if this can be done and if the dimension of $A_U$ is $g$, then we are done because then we will have constructed a global function $\psi$, and our machinery will quickly work. So the difficulty happens if the span of $U$ is a much smaller abelian variety $A_U\subsetneq A$. In particular, much of the argument below becomes simply unnecessary if~$A$ is assumed to be simple, i.e.~to not contain any proper nontrivial abelian subvariety. We encourage a novice reader to first think through the argument below assuming that~$A$ is simple, and that thus $A=A_U$, and only then at the second reading focus on the extra complications present in the general case.

In the construction, we will need to avoid the bad locus $\Sigma$ defined before. As already discussed, as the abelian variety is assumed to be indecomposable, it follows that $\dim\Theta_1=g-2$. Furthermore, we want~$A_U$ not to be contained in $\Sigma$, which is achieved simply by taking a generic translate of $A_U$ (and we'll abuse notation to denote $A_U$ that translate). Then furthermore for any fixed sufficiently small $y$, we will also have $A_U + Vy \not\subset \Sigma$. Let $\pi: \CC^g \to A = \CC^g \slash \Lambda$ be the universal cover of $A$, and let $\CC^d\coloneqq \pi^{-1}(A_U)$. Recall that \Cref{lem:xisimplepole} is valid for any $\tau$ holomorphic with simple poles satisfying equation \eqref{eq:D}, where we have not made any use of the assumption in \Cref{thm:Kricheverflex} that $u(x, y)$ is expressed in terms of theta function --- and now we use this generality. Note that if we restrict the theta function to the affine subspace $\CC^d + Vy$, and define
\[\tau(Z, y) \coloneqq \theta(Z + Vy), \qquad Z \in \CC^d\,,\]
then the function $u(Z, y) \coloneqq 2 D_1^2\ln\tau(Z, y)$ is periodic with respect to the lattice $\Lambda_U \coloneqq \Lambda \cap \CC^d$, and has a double pole along the divisor $\Theta^U(y) = (\Theta - Vy) \cap \CC^d$. Solutions of equation \eqref{eq:L} are then expected to be quasi-periodic with respect to~$\Lambda_U$. For $\lambda \in \Lambda_U$ and a function $f:\CC^d\to\CC$ we say that it is $\lambda$-periodic if $f(Z + \lambda) = f(Z)$ for all~$Z \in \CC^d$. We first prove

\begin{lemma}\label{lem:phizykseries}
Let equation \eqref{eq:D} for $\tau(Ux + Z, y)$ hold and let $\lambda$ be a vector of the sublattice $\Lambda_U = \Lambda \cap \CC^d \subset \CC^g$. Then
\begin{enumerate}
\item equation \eqref{eq:L} with the potential $u(Ux + Z, y)$ has a solution of the form $\psi = e^{kx + k^2y}\phi(Ux + Z, y, k)$ such that the coefficients $\xi_s(Z, y)$ of the formal series
\begin{equation}\label{eq:formalseriesphizyk}
\phi(Z, y, k) = e^{by} \left(1 + \sum \limits_{s = 1}^{\infty} \xi_s(Z, y) k^{-s}\right)
\end{equation}
are $\lambda$-periodic meromorphic functions of the variable $Z \in \CC^d$ with a simple pole along the divisor $\Theta^U(y)$;
more precisely
\[\xi_s(Z + \lambda, y) = \xi_s(Z, y) = \frac{\tau_s(Z, y)}{\tau(Z, y)},\]
where $\tau_s$ are holomorphic functions of the variable $Z \in \CC^d$.
\item $\phi(Z, y, k)$ is unique  up to a factor $\rho(Z, k)$ that is $D_1$-invariant and holomorphic in $Z$, that is any other such solution has the form
\begin{equation}\label{eq:rhozkfactor}
\phi_1(Z, y, k) = \phi(Z, y, k)\rho(Z, k), \qquad \hbox{ with }D_1 \rho = 0.
\end{equation}
\end{enumerate} 
\end{lemma}
This uniqueness of~$\phi$ will be important in constructing a suitably globally well-defined differential operator for which it is an eigenfunction.
\begin{proof}
Recall that the functions $\xi_s(Z)$ are defined recursively by equations \eqref{eq:xisrecurrencesystem}.
We want to show that there are more or less unique solutions of these equations that are $\lambda$-periodic. We do this by induction on~$s$.

For the base case we note that a particular solution of the first equation $2\partial_x \xi_1 = b - u$ is given by the formula
\[2\xi_1^0 = -2\partial_x \ln \tau + (l, Z)b\,,\]
where $(l, Z)$ is a linear form on $\CC^d$ given by the scalar product of $Z$ with a vector $l \in \CC^d$ such that $(l, U) = 1$ and $(l, \lambda) \ne 0$. The constant $b$ is then fixed by the periodicity condition for $\xi_1^0$,
\[(l, \lambda)b = 2\partial_x \ln \tau(Z + \lambda, y) - 2 \partial_x \ln \tau(Z, y)\,.\]
Now by \Cref{lem:xisimplepole}, equation \eqref{eq:D} ensures the local solvability of \eqref{eq:xisrecurrencesystem} in any domain where $\tau(Z + Ux, y)$ has simple zeros, i.e.~ outside of the set $\Theta_1^U (y) \coloneqq (\Theta_1 - Vy) \cap \CC^d$ where the $x$-derivative of $\tau$ (along the $U$-direction) also vanishes. We want to claim that the normalized solution actually exists globally on~$\CC^d$ or on~$A_U$. For this let us analyze the singular locus $\Theta_1^U = \Theta \cap D_1\Theta$. Note that this set does not contain a $D_1$-invariant line because any such line is dense in $A_U$. Therefore, the sheaf $\cV_0$ of $D_1$-invariant meromorphic functions on $\CC^d \setminus \Theta_1^U(y)$ with poles along the divisor $\Theta^U(y)$ coincides with the sheaf of meromorphic $D_1$-invariant functions on $\CC^d$. Thus once we show the existence of $D_1$-invariant meromorphic functions $\xi_s^0$, they exist globally with simple pole along the divisor $\Theta^U(y)$ which is equivalent to the vanishing of the cohomology group $H^1(\CC^d \setminus \Theta_1^U(y), \cV_0)$. What is then the freedom we have in the choice of~$\xi_s$? Once we fix a particular $\lambda$-periodic solution $\xi_s^0$, the general global meromorphic solution is given by $\xi_s = \xi_s^0 + c_s$, where the constant of integration $c_s(Z, y)$ is a holomorphic $D_1$-invariant function of~$Z$.

Now let us proceed to the step of induction. The setup will be a little complicated, as we will be using \eqref{eq:xisrecurrencesystem} for~$s$ and~$s-1$, to both derive the next $\xi_{s+1}^0$ and simultaneously fix the constant of integration $c_s$. More precisely, the inductive assumption is that we have constructed the normalized $\lambda$-periodic solution $\xi_{s-1}$ and that we know one $\lambda$-periodic solution $\xi_s^0$ of \eqref{eq:xisrecurrencesystem} for $s-1$. Then as a step of induction we need to construct one $\lambda$-periodic solution $\xi_{s+1}^0$, and its existence will allow us to fix the constant of integration in determining the normalized $\lambda$-periodic solution~$\xi_s=\xi_s^0+c_s(Z,y)$. 

Indeed, let $\xi_{s + 1}^*$ be a solution of \eqref{eq:xisrecurrencesystem} for the given~$\xi_s^0$, which is not necessarily $\lambda$-periodic. Then it is straightforward to check that the function
\[\xi_{s+1}^0(Z, y) = \xi_{s+1}^*(Z, y) + c_s(Z, y)\xi_1^0(Z, y) + \frac{(l, Z)}{2}\partial_yc_s(Z, y)\]
is a solution of \eqref{eq:xisrecurrencesystem} for $\xi_s = \xi_s^0 + c_s$ (Here we use that $c_s(Z, y)$ is holomorphic and $D_1$-invariant). An arbitrary choice of a $\lambda$-periodic $D_1$-invariant function $c_s(Z, y)$ does not affect the periodicity of $\xi_s$, but it does affect the periodicity of~$\xi_{s + 1}^0$. So in order to make $\xi_{s + 1}^0$ periodic, the function $c_s(Z, y)$ has to satisfy the linear differential equation 
\[(l, \lambda)\partial_y c_s(Z, y) = 2 \xi_{s + 1}^*(Z + \lambda, y) - 2 \xi_{s + 1}^*(Z, y)\,.\]
This equation, together with an initial condition $c_s(Z) = c_s(Z, 0)$, uniquely defines $c_s(Z, y)$. The induction step is thus completed. 

At last we need to show that the $\lambda$-periodic solution of equation \eqref{eq:L} is unique up to some $y$-independent $D_1$-invariant factor. If $\phi_1$ is another $\lambda$-periodic formal series such that $e^{kx + k^2y}\phi_1$ is a solution of \eqref{eq:L} with the same $u(Ux + Z, y)$, then by our construction above the ratio of the two $\lambda$-periodic formal series $\phi_1$ and $\phi$ is $y$-independent and $D_1$-invariant (the ratio of two solutions of the same differential equation should be a constant with respect to $x$). Therefore, we have equation \eqref{eq:rhozkfactor}, where $\rho(Z, k)$ could be obtained by the evaluation of the two sides at $y = 0$. The lemma is thus proved.
\end{proof}

In \Cref{lem:phizykseries} the point $\lambda \in \Lambda_{U}$ is an arbitrary fixed lattice vector. Applying it repeatedly, we can construct a solution of equation \eqref{eq:L} that is quasi-periodic with respect to all of the lattice~$\Lambda_U$. This will further restrict the freedom of choice of~$\phi$.  

\begin{corollary}\label{cor:quasiperiodicphi}
Let $\lambda_1, \dots, \lambda_d\in\Lambda_U$ be linearly independent, and let $Z_0\in\CC^d$ be arbitrary. Then, under the assumptions of \Cref{lem:phizykseries}, there is a unique solution of equation \eqref{eq:L} such that the corresponding formal series $\phi(Z, y, k; Z_0)$ is quasi-periodic with respect to all of the lattice~$\Lambda_U$, i.e.~ 
\begin{equation}\label{eq:phiquasiperiodicity}
\phi(Z + \lambda, y, k; Z_0) = \phi(Z, y, k; Z_0)\mu_{\lambda}(k)\qquad \forall\lambda\in\Lambda_U
\end{equation}
and satisfies the normalization conditions
\begin{equation}\label{eq:mulambdainormalization}
\mu_{\lambda_i}(k) = 1, \quad\forall i=1,\dots,d;\qquad \text{ and }\phi(Z_0, 0, k; Z_0) = 1.
\end{equation}
\end{corollary}
For the proof of this corollary, we will use the following reformulated simplified version of \cite[Lem.~12(b)]{Shiota1986}.
\begin{lemma}\label{lem:Shiota}
Let $A_U=\CC^d/\Lambda_U$ be an abelian variety. If $F$ is a holomorphic function on $\CC^d$ such that
\[F(Z + \gamma) = F(Z)\sigma(\gamma, Z),\]
for all $Z\in\CC^d$ and all $\gamma\in\Lambda_U$, where the multiplicative factors $\sigma(\gamma, Z)$ are $D_1$-invariant holomorphic functions of $Z$, then there exists a $D_1$-invariant function $h(Z)$ and $\tilde{\sigma}$ such that the equation
    \[h(Z + \gamma)\sigma(\gamma, Z) = h(Z)\tilde{\sigma}(\gamma),\]
    holds. The ambiguity in the choice of $h$ and $\tilde{\sigma}$ corresponds to the multiplication by the exponent of a linear form in $Z$ vanishing on $U$, i.e.~
    \[h'(Z) = h(Z) e^{(b, Z)}, \qquad \tilde{\sigma}'_{\lambda} = \tilde{\sigma}_{\lambda}e^{(b, \lambda)}, \qquad (b, U) = 0\,,\]
    where $b$ is orthogonal to $U$.
\end{lemma}

\begin{proof}
We denote by $\cO = \cO_{\CC^d}$ the structure sheaf of $\CC^d$, and by $\cO' = \cO'_{\CC^d} \coloneqq\text{Ker}(D_1: \cO \to \cO)$, the $D_1$-invariant sheaf on $\CC^d$. For $F(Z)$ as a function of $Z$ on $\CC^d$, assume that for any $\gamma \in \Lambda_U$,
\[F(Z + \gamma) = F(Z)\sigma(\gamma, Z),\]
where the multiplicative factors $\sigma(\gamma, Z)$ are $D_1$-invariant holomorphic functions on $Z$. Then we have $\sigma(\gamma, Z) \in H^0(\CC^d, \cO') \subset H^0(\CC^d, \cO)$. Hence the crossed homomorphism $\sigma: \gamma \mapsto \sigma(\gamma)$ defines a group cohomology class $\bar{\sigma} \in H_{gr}^1(\Lambda_U, H^0(\CC^d, \cO))$. We have an exact sequence
\[0 \to H^1(\Lambda_U, H^0(\CC^d, \cO)) \to H_{\Lambda_U}^1(\CC^d, \cO) \to H^1(\CC^d, \cO)^{\Lambda_U} = 0\]
by the Hochschild-Serre spectral sequence for equivariant cohomology \cite{Grothendieck1957}. Here $H_{\Lambda_U}^1(\CC^d, \cO)$ is the equivariant sheaf cohomology group, which in our case coincides with the ordinary cohomology $H^1(A, \cO_A)$ of the quotient space $A = \CC^d \slash \Lambda_U$ since the action of $\Lambda_U$ on $\CC^d$ is fixed point free. By the Hodge spectral sequence we have an exact sequence
\[\begin{tikzcd}
0 \ar[r] & H^0(A, \Omega_A^1) \ar[r] & H^1(A, \CC) \ar[r] \ar[d, equal] & H^1(A, \cO_A) \ar[r] \ar[d, equal] & 0\\
& & H_{gr}^1(\Lambda_U, \CC) \ar[r, two heads] & H_{gr}^1(\Lambda_U, H^0(\CC^d, \cO)).
\end{tikzcd}\]
Hence there exists $\tilde{\sigma} \in H_{gr}^1(\Lambda_U, \CC) = \text{Hom}(\Lambda_U, \CC)$ which is mapped to $\bar{\sigma} \in H_{gr}^1(\Lambda_U, H^0(\CC^d, \cO))$. Then there exists $h \in H^0(\CC^d, \cO)$ such that
\[\sigma(\gamma) = \tilde{\sigma}(\gamma) h(\cdot + \gamma) h^{-1}.\]
Since $\sigma$ as a function of $Z$ is $D_1$-invariant, we have $D_1h \in H^0(\CC^d, \cO)^{\Lambda_U} = \CC$. The ambiguity in the choice of $h$ and $\tilde{\sigma}$ corresponds to the multiplication by the exponent of a linear form in $Z$ vanishing on $U$, i.e.~
\[h'(Z) = h(Z) e^{(b, Z)}, \qquad \tilde{\sigma}'(\gamma) = \tilde{\sigma}(\gamma)e^{(b, \gamma)}, \qquad (b, U) = 0\,,\]
where $b$ is orthogonal to $U$. The lemma is proved.
\end{proof}
Now we proceed to prove \Cref{cor:quasiperiodicphi} of \Cref{lem:phizykseries}.

\begin{proof}[Proof of \Cref{cor:quasiperiodicphi}]
By Lemma \eqref{lem:phizykseries} for the given~$\lambda_1 \in \Lambda_U$ there exist wave solutions corresponding to $\phi$ which are $\lambda_1$-periodic. Moreover, from statement (2) of \Cref{lem:phizykseries} it follows that for any $\lambda \in \Lambda_U$,
\[\phi(Z + \lambda, y, k) = \phi(Z, y, k)\rho_{\lambda}(Z, k)\,,\]
where the multiplicative factors~$\rho_{\lambda}$ are $D_1$-invariant holomorphic functions. Then by \Cref{lem:Shiota}
there exists a $D_1$-invariant series $f(Z, k)$ with holomorphic in $Z$ coefficients and formal series $\mu_{\lambda}(k)$ with constant coefficients such that the equation
\[f(Z + \lambda, k)\rho_{\lambda}(Z, k) = f(Z, k)\mu_{\lambda}(k)\]
holds. The ambiguity in the choice of $f$ and $\mu$ corresponds to the multiplication by the exponent of a linear form in $Z$ vanishing on $U$, i.e.~
\[f'(Z, k) = f(Z, k) e^{(b(k), Z)}, \qquad \mu'_{\lambda}(k) = \mu_{\lambda}(k)e^{(b(k), \lambda)}, \qquad (b(k), U) = 0\,,\]
where $b(k) = \sum_s b_s k^{-s}$ is a formal series with vector-coefficients that are orthogonal to $U$. The vector $U$ is in general position with respect to the lattice. Therefore, the ambiguity can be uniquely fixed by imposing $(d - 1)$ normalizing conditions $\mu_{\lambda_i}(k) = 1, i > 1$. (Recall that $\mu_{\lambda_1}(k) = 1$ by construction.)

The formal series $f\phi$ is quasi-periodic and its multipliers satisfy \eqref{eq:mulambdainormalization}. Then, by these properties it is defined uniquely up to a factor which is constant in $Z$ and $y$. Therefore, for the unique definition of $\phi_0$, it is enough to fix its evaluation at $Z_0$ and $y = 0$. The corollary is proved.
\end{proof}

So what have we done so far? The first step was to find a formal solution that has at most simple poles, and the second step was to  find a formal $\lambda$-periodic solution, also with at most simple poles. It may seem that we have not done much yet, but at the next step we will already be able to construct an operator for which~$\psi$ is an eigenfunction. 

\subsection{Spectral curves}
We next claim that there exists a unique {\em pseudo}-differential operator~$\cL$ for which the $\lambda$-periodic solution of equation \eqref{eq:L}, with $u$ given by \eqref{eq:potentialuintheta}, is an eigenfunction. Of course what we really want is a differential operator, but this will come later.
\begin{lemma}\label{lem:pseudoL}
Under the assumptions of \Cref{thm:Kricheverflex}, there exists a unique pseudo-differential operator 
\[\cL(Z) = \partial_x + \sum \limits_{s = 1}^{\infty} u_s(Z)\partial_x^{-s}, \qquad Z \in A\,,\]
such that
\begin{equation}\label{eq:pseudoLeigenvalue}
\cL(Ux + Vy + Z)\psi = k\psi,
\end{equation}
where $\psi$ is the $\lambda$-periodic wave solution of \eqref{eq:L} corresponding to the unique $\phi$ constructed in \Cref{cor:quasiperiodicphi}.
\end{lemma}

As usual, we will construct $\cL$ simply by computing its coefficients $u_s$ term by term. A priori, $u_s$ are functions of both $x$ and $y$, as $\psi$ depends on $x$ and $y$, but here we would like to have~$u_s$ depend only on~$Ux + Vy + Z$, not on~$x$ and~$y$ individually. Note that a priori the~$u_s$ are meromorphic functions of~$Z \in \CC^g \setminus \Sigma$, but as discussed before, the singular locus~$\Sigma$ is of codimension at least two (remember that this comes from the assumption that the abelian variety is indecomposable), so that Hartogs' theorem applies and we obtain a global meromorphic function on $\CC^g$, which furthermore will automatically be periodic, and will thus descend to~$A$. The proof below makes this idea rigorous.

\begin{proof}
First we define~$\cL$ as the pseudo-differential operator with arbitrary coefficients $u_s(Z, y)$, which are functions of $Z$ and $y$.

Let $\psi$ be a $\lambda$-periodic wave solution. Substituting \eqref{eq:formalseriesphizyk} into \eqref{eq:pseudoLeigenvalue} gives a system of equations that recursively define $u_s(Z, y)$ as differential polynomials in $\xi_s(Z, y)$. Note that the coefficients of $\psi$ are local meromorphic functions of $Z$, but the coefficients of $\cL$ are well-defined global meromorphic functions on $\CC^g \setminus \Sigma$, because different $\lambda$-periodic wave solutions differ by multiplication  by a $D_1$-invariant factor, which does not affect the construction of~$\cL$. Since $\Sigma$ has codimension $\ge 2$, by Hartogs' theorem $u_s(Z, y)$ can be extended to a global meromorphic function on $\CC^g$.

The translation invariance of $u$ implies the translation invariance of the $\lambda$-periodic wave solutions. Indeed, for any constant $s$ the two series $\phi(Vs + Z, y - s, k)$ and $\phi(Z, y, k)$ yield $\lambda$-periodic solutions of the same equation. Therefore, they are equal, up to multiplication by some $D_1$-invariant factor. This factor does not affect $\cL$. Hence $u_s(Z, y) = u_s(Vy + Z)$.

The $\lambda$-periodic wave functions corresponding to $Z$ and $Z + \lambda'$ for any $\lambda' \in \Lambda$ also differ by multiplication by a $D_1$-invariant factor:
\begin{equation}\label{eq:lambdaperiodicwave}
D_1\left(\phi(Z + \lambda', y, k)\phi^{-1}(Z, y, k)\right) = 0.
\end{equation}
Hence $u_s$ are periodic with respect to $\Lambda$ and therefore are meromorphic functions on the abelian variety $A$. The lemma is proved.
\end{proof}

\begin{remark}
\begin{enumerate}
\item The existence and uniqueness of $\cL$ is by computation for any $Z \in \CC^g \setminus \Sigma$. Since in solving for~$u_s$ term by term, at some point we have to divide by the theta function, the product $u_s \cdot \theta$ will be holomorphic on $\CC^g \setminus \Sigma$, but then this product extends to $\CC^g$ by Hartogs' theorem. This is one place where somehow the bad locus is avoided. 
\item The second, more important, remark is that the $\cL$ we obtained is globally well-defined. Indeed, $\cL$ does not depend on the choice of the $\lambda$-periodic wave function $\psi$ because different $\psi$ differ by multiplication by a factor independent of~$x$, which does not influence the construction of~$\cL$.

\item One might start feeling a bit more optimistic because we have now constructed a pseudo-differential operator $\cL$ such that $\psi$ is its eigenfunction. However, for our framework to proceed, we need an actual {\em differential} operator, in particular so that eventually we could deduce some commutation relation from the finite-dimensionality trick, arguing that the kernel of a given differential operator is finite-dimensional. So how do we get a differential operator? Of course one can just chop off the `pseudo' part of $\cL$, but this would simply give $\partial_x$, which is certainly not what we want. However, the differential operator part $(\cL^m)_+$ of $\cL^m$, for~$m\in\NN$ (that is, we take all the terms of $\cL^m$ that have a non-negative power of $\partial_x$), is much more interesting, and will be what we use. The differential operators $(\cL^m)_+$ are of course related to the powers $\cL^m$, but in principle they are independent for different values of~$m$: when we compute $(\cL^m)_+$, all the terms up to $u_m$ will enter into the expression, so more and more of the expansion of $\cL$ will show up.
\end{enumerate}
\end{remark}

Now let's compute. We know that
\[[\partial_y - \partial_x^2 - u, \cL]\psi = 0\,,\]
which implies
\[[\partial_y - \partial_x^2 - u, \cL^m]\psi = 0\,.\]
We can then compute the commutator 
\begin{equation}\label{eq:partialyLmcommutator}
[\partial_y - \partial_x^2 - u, (\cL^m)_+]\psi = 2[\partial_x(\Res_{\partial_x}\cL^m)]\psi\,,
\end{equation}
where $\Res_{\partial_x}$ is simply defined to be the coefficient of $\partial_x^{-1}$. Thus to get a commutation relation we need to adjust our constructions to cancel these `residue' terms on the right-hand-side. We denote~$F_m(Z) \coloneqq \Res_{\partial_x}\cL^m$;  these are meromorphic functions of $Z\in A$, and crucially we prove

\begin{lemma}\label{lem:Fmsecondorderpole}
All functions $F_m$ have at most a second order pole along the theta divisor~$\Theta\subset A$. 
\end{lemma}

This is a result of a direct computation that also involves some miraculous cancellation.
\begin{proof}
Note that for $\psi$ in \Cref{lem:pseudoL}, there exists a unique pseudo-differential operator $W$ such that
\begin{equation} \label{eq:psiPhi}
\psi = W e^{kx + k^2 y}, \qquad W = 1 + \sum \limits_{s = 1}^{\infty} w_s(Ux + Z_0, y)\partial_x^{-s}\,,
\end{equation}
and by a direct computation one checks that~$\cL$ can then be expressed explicitly as $\cL = W \partial_x W^{-1}$. The coefficients of $W$ are universal differential polynomials of $\xi_s$. Therefore $w_s(Z + Z_0, y)$ is a global meromorphic function of $Z \in \CC^d$ and a local meromorphic function of $Z_0 \notin \Sigma$.

Recall that the left action of a pseudo-differential operator is the formal adjoint action, so that  the left action of $\partial_x$ on a function $f$ is $(f\partial_x) = -\partial_x f$. Consider the dual function defined by the left action of the operator $W^{-1}$, that is define $\psi^+ \coloneqq (e^{-kx - k^2 y})W^{-1}$. If $\psi$ is a formal solution of \eqref{eq:L}, then $\psi^+$ is a solution of the adjoint equation
\[(\partial_y + \partial_x^2 + u)\psi^+ = 0\,.\]
The same arguments as before prove that if equations \eqref{eq:D} for poles of $u$ hold, then $\xi_s^+$ have at worst simple poles at the poles of $u$. Therefore, if $\psi$ is as in \Cref{lem:phizykseries}, then the dual solution is of the form $\psi^+ = e^{-kx - k^2y}\phi^+(Ux + Z, y, k)$, where the coefficients $\xi_s^+(Z + Z_0, y)$ of the formal series
\[\phi^+(Z + Z_0, y, k) = e^{-by} \left(1 + \sum \limits_{s = 1}^{\infty} \xi_s^+(Z + Z_0, y)k^{-s}\right)\]
are $\lambda$-periodic meromorphic functions of the variable $Z \in \CC^d$ with a simple pole along the divisor $\Theta^U(y)$.

The ambiguity in the definition of $\psi$ does not affect the product
\begin{equation}\label{eq:adjpsipsi}
\psi^+\psi = (e^{-kx - k^2y}W^{-1})(W e^{kx + k^2 y})\,.
\end{equation}
Therefore, although each factor is only a local meromorphic function on $\CC^g \setminus \Sigma$, the coefficients $J_s$ of the product
\[\psi^+ \psi = \phi^+(Z, y, k)\phi(Z, y, k) =: 1 + \sum \limits_{s = 2}^{\infty} J_s(Z, y)k^{-s}\,.\]
are global meromorphic functions of $Z$. Moreover, the translation invariance of $u$ implies that the functions~$J_s$ have the form $J_s(Z, y) = J_s(Z + Vy)$. Each of the factors in the left-hand side of \eqref{eq:adjpsipsi} has a simple pole along $t_{Vy}\Theta$. Hence $J_s(Z)$ is a meromorphic function on $A$ with a second-order pole along $\Theta$.

From the definition of $\cL$, it follows that
\[\Res_k(\psi^+(\cL^n\psi)) = \Res_k(\psi^+ k^n \psi) = J_{n + 1}\,.\]
On the other hand, using the identity
\[\Res_k(e^{-kx}\cD_1)(\cD_2 e^{kx}) = \Res_{\partial_x}(\cD_2 \cD_1)\,,\]
which holds for any two pseudo-differential operators, we get
\begin{equation} \label{eq:respartialLn}
\Res_k(\psi^+ \cL^n \psi) = \Res_k(e^{-kx}W^{-1})(\cL^n W e^{kx}) = \Res_{\partial_x} \cL^n = F_n.
\end{equation}
Therefore, $F_n = J_{n + 1}$ and the lemma is proved.
\end{proof} 

What's the use of \Cref{lem:Fmsecondorderpole}? Well, it means that we can think of $F_m$ as a section of the line bundle~$2\Theta$, that is $F_m \in H^0(A, 2\Theta)$. This is valid for any $m$, but the dimension of the space of sections is $\dim H^0(A,2\Theta)=2^g$, as we already know. Let then $\hat{\mathbf{F}}\subseteq H^0(A,2\Theta)\simeq \CC^{2^g}$ be the linear subspace generated by $\{F_m, m = 0, 1, \dots\}$, where we set $F_0 \coloneqq 1$ to be the constant function. Then for all but $\hat{g} \coloneqq \dim\hat{\mathbf{F}}$ positive integers $m$, there exist constants $c_{m, i}$ such that
\begin{equation}\label{eq:cmi}
F_m(Z) + \sum \limits_{i = 0}^{m - 1} c_{m, i}F_{i}(Z) = 0\,, 
\end{equation}
and then we can take the same linear combination of linear differential operators $(\cL^i)_+$ and obtain
\[[\partial_y - \partial_x^2 - u, (\cL^m)_+ + c_{m, m-1}(\cL^{m-1})_+ + \dots + c_{m, 0}(\cL^{0})_+] = 0\,.\]
Let $I$ denote the set of those $\hat{g}$ integers $m$ for which the $F_m(Z)$ is linearly independent from the previous $F_i$, i.e.~such that \eqref{eq:cmi} does not hold for any $c_{m,i}$. Then we have
\begin{lemma}\label{lem:eigenLm}
Let $\cL$ be the pseudo-differential operator corresponding to a $\lambda$-periodic function $\psi$ constructed above. Then for any $m\notin I$ the differential operator
\[L_m \coloneqq (\cL^m)_+ + \sum \limits_{i = 0}^{m - 1} c_{m, i}(\cL^{i})_+\]
has $\psi$ as an eigenfunction:
\begin{equation}\label{eq:Lneigenproblem}
L_m \psi = a_m(k)\psi, \qquad \hbox{ where }a_m(k) = k^m + \sum \limits_{s = 1}^{\infty} a_{m,s}k^{m - s}
\end{equation}
for some constants $a_{m, s}$.
\end{lemma}

\begin{proof}
First, note that from \eqref{eq:partialyLmcommutator}, it follows that for any $n \notin I$
\[[\partial_y - \partial_x^2 - u, L_n] = 0\,.\]
Hence if $\psi$ is a $\lambda$-periodic solution of \eqref{eq:L} corresponding to $Z \notin \Sigma$, then $L_n\psi$ is also a formal solution of the same equation. This implies the equation $L_n\psi = a_n(Z, k)\psi$, where $a$ is $D_1$-invariant. The ambiguity in the definition of $\psi$ does not affect $a_n$. Therefore, the coefficients of $a_n$ are well-defined global meromorphic functions on $\CC^g \setminus \Sigma$. The $D_1$-invariance of $a_n$ implies that $a_n$ is a holomorphic function of $Z\in\CC^g \setminus \Sigma$. Hence it has an extension to a holomorphic function of $Z\in\CC^g$. Equation \eqref{eq:lambdaperiodicwave} implies that $a_n$ is periodic with respect to the lattice $\Lambda$. Hence by Liouville's theorem~$a_n$ is independent of~$Z$. The lemma is proved.
\end{proof}

\begin{remark}
    Note that $a_{s, n} = c_{s, n}$ for $s \le n$. From  equation \eqref{eq:pseudoLeigenvalue}, we see that $\psi$ is an eigenfunction of $\cL$, and thus $\psi$ is an eigenfunction of $\cL^m + \sum \limits_{i = 0}^{m - 1} c_{m, i}(\cL^{i})$ with eigenvalue $k^m + \sum_{i  = 0}^{m - 1}c_{m, i}k^i$. Since the differential part of $(\cL^i)_+$ is the same as that of $\cL^i$, we only need some $a_{m, s},\, s > m$ to compensate the negative powers of $\partial_x$ to make $\psi$ an eigenfunction of $L_m$ while keeping $a_{m, s}=c_{m,s}$ for $s \le m$.
\end{remark}

For $m\notin I$ the operator $L_m$ can be regarded as a $(Z \in A \setminus \Sigma)$-parametric family of ordinary differential operators $L_m^Z$ whose coefficients have the form
\[L_m^Z = \partial_x^m + \sum \limits_{i = 1}^{m} u_{m, i}(Ux + Z)\partial_x^{m - i}\,.\]

We have the following easy
\begin{corollary}
The operators $L_m^Z$ and $L_n^Z$ for $m,n\notin I$ commute with each other, that is
\begin{equation}\label{eq:LnzLmzcommute}
[L_n^Z, L_m^Z] = 0, \qquad \forall Z \in A \setminus \Sigma.
\end{equation}
\end{corollary}
\begin{proof}
Indeed, from \eqref{eq:Lneigenproblem} it follows that the function $\psi$ for any parameter~$k$ lies in the kernel of the commutator $[L_n^Z, L_m^Z]$. This commutator is an ordinary differential operator, and thus by the finite-dimensionality trick it must be zero.
\end{proof}

We have finally obtained a commutative ring of differential operators! The following lemma is standard in the integrable systems literature.

\begin{lemma}[\cite{Krichever1977, Mumford1978, Segal-Wilson1985}] \label{lem:ztoAzcorrespondence}
Let $\cA^Z, Z \in A \setminus \Sigma$, be the commutative ring of ordinary differential operators generated by the operators $L_n^Z$. Then there is an irreducible algebraic curve $\hat C$ of arithmetic genus $\hat{g} = \dim\hat{\mathbf{F}}$ such that $\cA^Z$ is isomorphic to the ring $\cM(\hat C, p_{\infty})$ of meromorphic functions on $\hat C$ with the only pole (of an arbitrary order) at a smooth point $p_{\infty}$. The correspondence $Z \to \cA^Z$ defines a holomorphic embedding of $A \setminus \Sigma$ into the compactified Jacobian of $\hat C$, that is the set of torsion-free rank one sheaves $\cF$ on $\hat C$:
\begin{equation}\label{eq:XinPicGamma}
j: A \setminus \Sigma \hookrightarrow \overline{\Pic}(\hat C).
\end{equation}
\end{lemma}
\begin{remark}
We have not discussed singular curves in these lecture notes, nor the notion of the compactified Jacobian for such a singular curve. Dealing with singular curves is a major technical difficulty in this subject, and much of the relevant machinery in this direction was developed by Mumford, who in particular in \cite{Mumford1978} proved the lemma above in the generality we stated it.

To keep our presentation manageable, we will {\em not} endeavor to discuss singular curves, torsion-free rank one sheaves, etc, and will restrict the arguments below to the special case when the spectral curve~$\hat C$ happens to be smooth. This will help to elucidate the main remaining ideas in Krichever's proof, while avoiding the important and difficult technical issues of dealing with compactified Jacobians of singular curves. The case of the above lemma for smooth spectral curves was proved by Krichever \cite{Krichever1977}, and we restrict outselves to this situation.
\end{remark}

\begin{proof}[Proof of \Cref{lem:ztoAzcorrespondence} under the assumption that~$\hat C$ is smooth]
Recall that \Cref{thm:spectralcurve} provides the bijection 
\begin{equation}\label{eq:Krichevercorrespondence}
        \cA \longleftrightarrow \{\hat C, p_{\infty}, [k^{-1}]_1, \mathscr{F}\}
\end{equation}
between commutative rings $\cA \subseteq \CC[\![x]\!][\tfrac{d}{dx}]$ of ordinary linear differential operators containing a pair of monic operators of coprime orders, and sets of algebraic data $\{\hat C, p_{\infty}, [k^{-1}]_1, \mathscr{F}\}$, where $\hat C$ is an algebraic curve which is the spectral curve of $\cA$, $p_{\infty} \in \hat C$ is a smooth point, $[k^{-1}]_1$ is the first jet of a local coordinate $k^{-1}$ on $\hat C$ around $p_\infty$, and  $\mathscr{F}$ is a torsion-free rank one sheaf on $\hat C$ (which is a line bundle when $\hat{C}$ is smooth) such that
\begin{equation}\label{eq:H0H1}
    H^0(\hat C, \mathscr{F}) = H^1(\hat C, \mathscr{F}) = 0.
\end{equation}
The correspondence becomes one-to-one if the rings $\cA$ are considered up to conjugation $\mathscr{A}' = g(x)\cA g(x)^{-1}$ for $g(x) \in \CC[\![x]\!], g(0) \ne 0$ (see \cite{Mumford1978}).

It remains to understand how the data on the right-hand side of \eqref{eq:Krichevercorrespondence} depends on $Z\in A\setminus\Sigma$? Let $\hat C^Z$ be the spectral curve corresponding to $\cA^Z$. The eigenvalues $a_m(k)$ of the operators $L_m^Z$ defined in \eqref{eq:Lneigenproblem} coincide with the Laurent expansions at $p_{\infty}$ of the meromorphic functions $a_m \in A(\hat C^Z, p_{\infty})$. As remarked in the proof of \Cref{lem:eigenLm} these eigenvalues are $Z$-independent. Hence the spectral curve $\hat C = \hat C^Z$ is $Z$-independent as well. 
    
The construction of the correspondence \eqref{eq:Krichevercorrespondence} implies that if the coefficients of the operators $\cA$ holomorphically depend on parameters then the algebro-geometric data also depends holomorphically on the parameters. In the construction, the spectral curve does not depend on the choice of the initial point $x_0 = 0$, but the sheaf $\mathscr{F} = \mathscr{F}_{x_0}$ does. Hence the map~$j$ is holomorphic away from $\Theta$. The correspondence \eqref{eq:Krichevercorrespondence} can be extended to the cases when the coefficients of the differential operators in $\cA$ are meromorphic at $x = x_0$ by a shift of the argument $x$ which corresponds to sheaves that violate condition \eqref{eq:H0H1}. Under such extension, $j$ extends holomorphically over $\Theta \setminus\Sigma$, as well. The lemma is proved.
\end{proof}
How do we know that the spectral curve $\hat C$ we have constructed is the curve we want? A priori, the genus $\hat{g}$ of $\hat C$ could be as large as $2^g$! To finish the proof of the main result we will need to show that the map~$j$ in~\eqref{eq:XinPicGamma} is in fact an isomorphism. Note that $a_m(k)$ in equation \eqref{eq:Lneigenproblem} are actually functions on the curve $\hat C$, and the index set $I$ is the so-called Weierstrass gap sequence of the point $p_{\infty}\in\hat C$.

A commutative ring $\cA$ of linear ordinary differential operators is called maximal if it is not contained in any bigger commutative ring. We want to argue that for a generic $Z$ the ring $\cA^Z$ is maximal. Suppose this is not the case, then there exists $\alpha \in I$, such that for each $Z\in A\setminus \Sigma$ there exists an operator $L_{\alpha}^Z$ of order $\alpha$ which commutes with all~$L_m^Z, m \notin I$. We claim that $L_{\alpha}^Z$ then also commutes with $\cL$. For $n \notin I$, since $[L_{\alpha}^Z, L_n^Z] = 0$, by \Cref{thm:L12commute}, $\psi$ is an eigenfunction of $L_{\alpha}^Z$ with some eigenvalue $A_\alpha^Z(k)$. But by construction of $\cL$, i.e.~\Cref{lem:pseudoL}, $\psi$ is a formal eigenfunction of $\cL$ with eigenvalue $k$, where $k^{-1}$ is the local coordinate around $p_{\infty} \in \hat C$. Then heuristically we may argue that we could formally express $L_{\alpha}^Z$ as $A_{\alpha}^Z(\cL)$, which obviously commutes with the pseudo-differential operator~$\cL$ itself.

Note that a differential operator commuting with $\cL$ up to order $O(1)$ can be represented in the form $L_{\alpha} = \sum_{i < \alpha} c_{\alpha, i}(Z)\cL_+^i$, where $c_{i, \alpha}(Z)$ are some $D_1$-invariant functions of $Z$. This is because all differential operators that commute with $\cL$ up to order $O(1)$ can be determined similarly to the computation in exercise \Cref{ex:L2L3commute}. Such differential operators form a vector space; note that by definition $L_+^i$ commutes with $\cL$ up to order $O(1)$, and since we have one such operator in each degree~$i$, they constitute a basis for this vector space, so that any $L_{\alpha}$ commuting with $\cL$ up to order $O(1)$ is a linear combination of $L_+^i$ with some coefficients that are independent of $x$, i.e.~with $D_1$-invariant coefficients.
$L_{\alpha}$ commutes with $\cL$ if and only if
\begin{equation}\label{eq:commutantofL}
F_{\alpha}(Z) = \sum \limits_{i = 0}^{\alpha - 1}c_{\alpha, i}(Z)F_i(Z) \qquad \hbox{ for }D_1 c_{i, \alpha} = 0.
\end{equation}
Note the difference between \eqref{eq:cmi} and \eqref{eq:commutantofL}: in \eqref{eq:cmi} the coefficients $c_{i, n}$ are constants. The $\lambda$-periodic wave solution of equation \eqref{eq:L} is a common eigenfunction of all commuting operators, i.e.~ $L_{\alpha}\psi = a_{\alpha}(Z, k)\psi$, where $a_{\alpha} = k^{\alpha} + \sum_{s = 1}^{\infty} a_{s, \alpha}(Z)k^{\alpha - s}$ is $D_1$-invariant. The same arguments as those used in the derivation of \eqref{eq:Lneigenproblem} show that the eigenvalue $a_{\alpha}$ is $Z$-independent and $a_{\alpha, s} = c_{\alpha, s}$, for $ s \le \alpha$. Therefore, the coefficients in \eqref{eq:commutantofL} are $Z$-independent which contradicts the assumption that $\alpha \notin I$, so that a linear relation must exist.

\smallskip
So far the $\lambda$-periodic wave function we have constructed is only a formal power series multiplied by an exponential function. Now we can finally prove the global existence of the wave function and this is achieved by comparing with the Baker-Akhiezer function thanks for the existence of the algebraic curve $\hat C$.

\begin{lemma}\label{lem:globalization}
Under the assumptions of \Cref{thm:Kricheverflex}, there exists a common eigenfunction $\psi = e^{kx}\phi(Ux + Z, k)$ of the commuting operators $L_m^Z$ such that the coefficients of the formal series
\[\phi(Z, k) = 1 + \sum \limits_{s = 1}^{\infty} \xi_s(Z)k^{-s}\]
are {\em global} meromorphic functions of $Z\in A$ with a simple pole along the divisor~$\Theta$.
\end{lemma}
As before, we will only prove this lemma under the assumption that the spectral curve $\hat C$ constructed in \Cref{lem:ztoAzcorrespondence} is smooth; the proof for the general case of singular $\hat C$ would follow the same lines as below if we had available a theory of Baker-Akhiezer functions on singular curves. There are some recent developments in this direction by Abenda-Grinevich \cite{Abenda-Grinevich2018}, Nakayashiki (also with Bernatska and Enolski) \cite{Nakayashiki2019, Bernatska-Enolski-Nakayashiki2020, Nakayashiki2020, Nakayashiki2024}, Kodama-Xie \cite{Kodama-Xie2021}, Agostini-Fevola-Mandelshtam-Sturmfels \cite{Agostini-Fevola-Mandelshtam-Sturmfels2023} from various perspectives. Still, the theory of Baker-Akhiezer functions on singular curves is quite difficult, and we avoid discussing it or giving the analytic arguments to bypass it. We thus only prove the lemma under the additional assumption that $\hat C$ is smooth.
\begin{proof}
Indeed, for smooth $\hat C$ there is a unique Baker-Akhiezer function \eqref{eq:Baker-Akhiezer1} normalized by the condition $\hat{\psi}_0 \vert_{x = 0} = 1$, which is of the form
\begin{equation}\label{eq:hatpsitheta}
\hat{\psi}_0 = \frac{\hat{\theta}(\hat{A}(p) + \hat{U}x + \hat{Z})\hat{\theta}(\hat{Z})}{\hat{\theta}(\hat{U}x + \hat{Z})\hat{\theta}(\hat{A}(p) + \hat{Z})}e^{x \int \Omega(p)}\,,
\end{equation}
where $\hat{A}: \hat C \hookrightarrow \Jac(\hat C)$ is the Abel map (to the usual Jacobian of a smooth curve), and we use hats to remind ourselves that all of this happens on the curve $\hat C$ (which so far is unrelated to the abelian variety $A$ that we started from).

Then the function
\begin{equation}\label{eq:hatpsiBA}
\hat{\psi}_{\BA} \coloneqq \frac{\hat{\theta}(\hat{A}(p) + \hat{U}x + \hat{Z})}{\hat{\theta}(\hat{U}x + \hat{Z})}e^{x \int \Omega(p)}
\end{equation}
is also a common eigenfunction of the commuting operators since the rest of the expression is independent of $x$.

In the neighborhood of $p_{\infty}$ the function $\hat{\psi}_{\BA}$ has the form
\[\hat{\psi}_{\BA} = e^{kx} \left(1 + \sum \limits_{s = 1}^{\infty} \frac{\tau_s(\hat{Z} + \hat{U}x)}{\hat{\theta}(\hat{U}x + \hat{Z})}k^{-s}\right), \]
where $\tau_s(\hat{Z})$ are some global holomorphic functions.

By \Cref{lem:ztoAzcorrespondence}, we have a holomorphic embedding $\hat{Z} = j(Z)$ of $A \setminus\Sigma$ into $\Jac(\hat C)$. We can then use this map to pull back the Baker-Akhiezer function on $\Jac(\hat C)$ to the abelian variety $A$, denoting $\psi \coloneqq j^*\hat{\psi}_{\BA}$. This is a globally well-defined function away from $\Sigma$. Note that if $Z \notin \Theta$, then $j(Z) \notin \hat{\Theta}$, hence the coefficients of $\psi$ are regular away from $\Theta$. Since the singular locus is at least of codimension $2$, once again by  Hartogs' theorem we can extend $\psi$ to all of $A$. This finishes the proof in the case of smooth  spectral curve~$\hat C$.
\end{proof}

Recall that by \Cref{rmk:KPgeometry}, the KP flows are linearized on the Jacobian of the spectral curve, and they are actually tangent to the image of the spectral curve under the Abel-Jacobi map. To prove that the map $j$ in \eqref{eq:XinPicGamma} is actually an isomorphism, our strategy is to show that KP flows are actually induced by abelian functions on $A$, and this is the content of the next lemma.

\begin{lemma}\label{lem:hatFspannedbyHi}
The vector space $\hat{\mathbf{F}}\subset H^0(A,2\Theta)$ generated by the abelian functions $\{F_0, F_m = \Res_{\partial_x}\cL^m\}$, is contained in the vector space $\mathbf{H}$ generated by $F_0$ and by the abelian functions $H_i \coloneqq \partial_x \partial_{Z_i}\ln \theta(Z)$.
\end{lemma}

\begin{proof}
First we note that in order to prove the statement of the lemma, it suffices to show that for any $s>0$, $F_s = \partial_x Q_s$ for some meromorphic function~$Q_s$ with a pole along $\Theta$. Indeed, since $F_s$ is a meromorphic function on $A$ with a second order pole on $\Theta$, the function $Q_s$, as a function of $Z\in\CC^{\hat n}$ must have a simple pole along~$\Theta$. Moreover, since $\dim H^0(A,\Theta)=1$, there is no meromorphic function on~$A$ with a first order pole along~$\Theta$. Thus considered as a function of $Z\in\CC^{\hat n}$, $Q_s$ must be quasiperiodic, that is we must have $Q_s(Z + \lambda) = Q_s(Z) + c_{s, \lambda}(Z)$ for any $\lambda\in\Lambda$. Since $F_s=\partial_x Q_s$ is a function on~$A$, i.e.~periodic on $\CC^{\hat n}$, it follows that $\partial_x c_{s,\lambda}(Z)=0$, and thus $c_{s,\lambda}(Z)=c_{s,\lambda}$ must be independent of~$Z$. Consider now the function $G_s(Z)\coloneqq Q_s(Z)\cdot \theta(Z)$. Since we have canceled the simple pole of $Q_s(Z)$ along~$\Theta$, it follows that $G_s$ is a holomorphic function of $Z\in\CC^{\hat n}$. Moreover, since $c_{s,\lambda}\cdot \theta(Z)$ vanishes along the theta divisor, it follows that restricting $G_s$ to the theta divisor gives a holomorphic function on $\Theta\subset A$ (that is, $G_s|_\Theta$ is periodic and not just quasi-periodic). Thus $G_s\in H^0(\Theta,\Theta|_\Theta)$, where this means the restriction of the theta as a line bundle to theta as a divisor. A standard computation gives $\dim H^0(\Theta,\Theta|_\Theta)=\hat n$ \footnote{We encourage you to do this as an exercise in algebraic geometry techniques using the restriction exact sequence for $\Theta\subset A$, but is easy to see this by observing that $\Theta$ is ample on~$A$, and thus its restriction is also ample, and thus $H^i(\Theta,\Theta|_\Theta)=0$ for any $i>0$, while the top self-intersection number is $\langle (\Theta|_\Theta)^{\hat n-1}\rangle_\Theta=\langle \Theta^{\hat n}\rangle_A=\hat n!$, so that finally
$$
 \chi(\Theta,\Theta|_\Theta)=\dim H^0(\Theta,\Theta|_\Theta)=\frac{1}{(\hat n-1)!}\langle (\Theta|_\Theta)^{\hat n-1}\rangle_\Theta=\hat n
$$
by the Hirzebruch-Riemann-Roch formula.}. Moreover, for any section $s$ of any line bundle $L$ on any variety $X$, the partial derivatives of $s$ are sections of $L$ restricted to the zero locus of $s$ (if one na\"\i vely writes down the transformation formula, the derivative of the transition function is multiplied by~$s$ and vanishes on its zero locus), and thus the $\hat n$ partial derivatives of the theta function with respect to all coordinates $Z_i$ on $\CC^{\hat n}$ generate the $\hat n$-dimensional space $H^0(\Theta,\Theta|_\Theta)$. Thus $G_s$ must be a linear combination of $\partial_{Z_i}\theta$, and thus $Q_s=G_s/\theta$ is a linear combination of logarithmic derivatives.

Let $\psi(x, Z, k)$ be the formal Baker-Akhiezer function \eqref{eq:Baker-Akhiezer1}. Then by \Cref{lem:globalization} the coefficients $w_s(Z)$ of the corresponding operator $W$ defined in \eqref{eq:psiPhi} are global meromorphic functions with poles along $\Theta$.

Note that the left and right actions of pseudo-differential operators are formally adjoint, i.e.~ for any two pseudo-differential operators $\cD_1,\cD_2$ the equality $(e^{-kx}\cD_1)(\cD_2 e^{kx}) = e^{-kx}(\cD_1\cD_2 e^{kx}) + \partial_x(e^{-kx}(\cD_3 e^{kx}))$ holds, where $\cD_3$ is some pseudo-differential operator whose coefficients are differential polynomials in the coefficients of $\cD_1$ and $\cD_2$. Therefore, if we take $\cD_1 = W^{-1}$ and $\cD_2 = W$, then from \eqref{eq:adjpsipsi} and \eqref{eq:respartialLn} it follows that
\[\psi^+ \psi = 1 + \sum \limits_{s= 2}^{\infty} F_{s - 1}k^{-s} = 1 + \partial_x \left(\sum \limits_{s = 2}^{\infty} Q_s k^{-s}\right)\,.\]
The coefficients~$Q_s$ are differential polynomials in the coefficients $w_s$ of the wave operator. Therefore, they are global meromorphic functions of $Z$ with poles along $\Theta$, and their $\partial_x$ derivatives give $F_s$.
\end{proof}

We are now ready to complete the proof of the main result, Krichever's characterization of Jacobians by the existence of a flex line of the Kummer variety. Recall that we continue to work under the essential technical assumption that $\hat C$ is smooth.
\begin{proof}[Proof of \Cref{thm:flextheta}]
It was a long way to go from the flex line condition to the global existence of an $\lambda$-periodic solution of equation \eqref{eq:L} and then to existence of a collection of commuting differential operators $L_n^Z$ indexed by $n \notin I$, and now it is the time to harvest the results. From the previous discussion (more precisely, in \Cref{lem:ztoAzcorrespondence}), we have constructed the holomorphic embedding $A \backslash \Sigma\hookrightarrow\Jac(\hat C)$ (we continue to assume that $\hat C$ is smooth, sacrificing many essential technical details for the clarity of the exposition of the idea) . The genus of the spectral curve $\hat C$ can a priori be up to $2^g$; to finally finish the proof of \Cref{thm:flextheta} we need to show that $A$ is in fact isomorphic to $\hat C$, for which we need to show that the map $j$ in \Cref{lem:ztoAzcorrespondence} is an isomorphism. The idea is to construct the inverse of $j$ as follows: from Krichever's construction of an algebro-geometric solution of the KP hierarchy using the Baker-Akhiezer functions (see \Cref{sec:BAfunction}), we see that $\Jac(\hat C)$ is precisely the orbits of the KP flows\footnote{By orbits of the KP flows or KP orbits we mean the trajectories traced by solutions of the KP hierarchy under its infinite set of commuting time evolutions, which could either be finite-dimensional or infinite-dimensional. For example, the orbits for the soliton the solutions are given by a finite-dimensional torus $(\CC^*)^{n}$ in a finite-dimensional Grassmannians \cite{Kodama2017}, the orbits of algebro-geometric (quasi-periodic) solutions are given by Jacobians of spectral curves, etc.} which control the deformation of the pseudo-differential operator $\cL$ such as the one in \Cref{lem:pseudoL}, and by \Cref{lem:hatFspannedbyHi} these flows are actually induced by abelian functions on $A$, which finally induces an embedding of $\Jac(\hat{C})$ into $A$, providing an inverse to~$j$. In the following we give a more detailed explanation of this brief discussion of the idea.

Recall that by the correspondence in the proof of \Cref{lem:ztoAzcorrespondence} between the commutative rings of differential operators and the geometric data on the spectral curve, the KP flows discussed in \Cref{rmk:KPgeometry} define deformations of the commutative rings $\cA$ of ordinary differential operators, and the spectral curve~$\hat C$ is invariant under these flows. The orbits of the KP flows applied to a given initial solution of the KP hierarchy are then isomorphic to the Jacobian $J(\hat C)$ (see \Cref{rmk:KPgeometry} and \Cref{rem:flows2}).

Recall that by \Cref{rmk:KPgeometry} for any first order pseudo-differential operator $\cL$ like the one we have obtained in \Cref{lem:pseudoL}, the KP hierarchy is nothing but an infinite set of deformations of the pseudo-differential operator $\cL$ given by the equations
\begin{equation}\label{eq:KPSatoform}
\partial_{t_n}\cL = [(\cL^n)_+, \cL].
\end{equation}
If the operator $\cL$ is defined by $\lambda$-periodic solutions of equation \eqref{eq:L} as above, then equations \eqref{eq:KPSatoform} are equivalent to the equations
\begin{equation}\label{eq:KPtn}
\partial_{t_n} u = \partial_x F_n,
\end{equation}
where the first two times of the hierarchy are identified with the variables $x\coloneqq t_1$ and $y\coloneqq t_2$.

Recall that in \Cref{rem:flows2} we have shown that $U^{(n)}$ are the directions of the KP flows induced on $\Jac(\hat C)$. As $u$ depends only on the sum $\sum U^{(n)}t_n$ and not on the individual times $t_n$ (see \eqref{eq:potentialuintheta}), from \eqref{eq:12abeliandifferential2} we see that $\partial_{t_n}$ is the direction tangent to the orbit of the KP flow. Equations \eqref{eq:KPtn} then identify the space $\hat{\mathbf{F}}_1$ generated by the functions $D_1 F_n$ with the tangent space of the KP orbits at $\cA^Z$. Then from \Cref{lem:hatFspannedbyHi} it follows that the tangent space of the KP orbit is a subspace of the tangent space of the abelian variety $A$ (generated by all directions $Z_i$). Hence for any $Z \in A\setminus \Sigma$, the orbit of the KP flows of the ring $\cA^Z$ is contained in $A$, i.e.~this defines a holomorphic embedding:
\begin{equation}\label{eq:JGammainX}
i_Z: J(\hat C) \hookrightarrow A.
\end{equation}
Then \eqref{eq:XinPicGamma} implies that $i_Z$ is an isomorphism. A posteriori it turns out that for the Jacobians of smooth algebraic curves the bad locus $\Sigma$ is empty \cite{Shiota1986}, i.e.~the embedding $j$ in \eqref{eq:XinPicGamma} is defined everywhere on $A$ and is inverse to $i_Z$.
\end{proof}

\appendix
\section{Hints for exercises}\label{sec:hints}
\begin{hint}[Hint for \Cref{ex:1}]
The differential operator $L$ can be brought into a normal form with leading coefficient $u_n$ to be $1$ and subleading coefficients $u_{n-1}$ to be $0$ by transforming the independent and dependent variables, respectively. Here we take the second order differential operator
\[L = u_2(x)\frac{d^2}{dx^2} + u_1(x) + u_0(x)\]
and the corresponding equation $L\psi = 0$ as an example to illustrate the method, and the general case is similar. To normalize the leading coefficient, we make the change of variable $t = t(x)$ and note that
\[\frac{d}{dx} = \frac{dt}{dx}\frac{d}{dt}, \qquad \frac{d^2}{dx^2} = \left(\frac{dt}{dx}\right)^2\frac{d^2}{dt^2} + \frac{d^2t}{dx^2}\frac{d}{dt}.\]
Substituting these into $L$ we obtain
\[L = u_2(x)\left({t'}^{2}\frac{d^2}{dt^2} + t''\frac{d}{dt}\right) + u_1(x)t'\frac{d}{dt} + u_0(x).\]
To make the leading coefficient to be $1$, we set
\[u_2(x){t'}^2 = 1,\]
which leads to
\[t(x) = \int \frac{dx}{\sqrt{u_2(x)}}.\]
Now the differential operator $L$ becomes
\[L = \frac{d^2}{dt^2} + p(t)\frac{d}{dt} + q(t),\]
where
\[p(t) = \frac{u_2(x)t'' + u_1(x)t'}{{t'}^2} = \frac{u_1(x)}{\sqrt{u_2(x)}} - \frac{u'_2(x)}{2\sqrt{u_2(x)}}, \quad q(t) = \frac{u_0(x)}{{t'}^2}.\]
To further remove the subleading coefficient $p(t)$, we make the change of dependent variable $\psi(t) = \varphi(t)\exp\left(-\frac{1}{2}\int p(t)dt\right)$ which corresponds to conjugating the differential operator $L$ by the function $\exp\left(\frac{1}{2}\int p(t)dt\right)$. Substituting $\psi(x)$ into the equation $L\psi = 0$ we obtain
\[\frac{d^2}{dt^2}\varphi + u(t)\varphi = 0,\]
where
\[u(t) = \frac{u_0(x(t))}{u_2(x(t))} - \frac{p(t)^2}{4} - \frac{p'(t)}{2}.\]
The above procedure to bring a second-order differential operator into its normal form is called Kummer-Liouville transformation (See \cite[p.~292, Eqs.~(20),(20a)]{Courant-Hilbert1953}, or \cite[p.~320, Eqs.~(33),(34)]{Birkhoff-Rota1989}).
\end{hint}
\begin{hint}[Hint for \Cref{ex:2}]
We first check that for a given $L_1$ with leading coefficient~$1$, any $L_2$ that commutes with it must have leading order term a constant times the $m$'th derivative. This is something we did not address properly in the main text: if we have two commutating differentials operators $L_1, L_2$, we can put $L_1$ in the canonical form by some transformations, but why can we achieve this for $L_2$ simultaneously? Then once we do that, we can check that the dimension of the space of $L_2$ that commute with $L_1$ is at most $m$ or $m + 1$ depending on what one does with the coefficients. We can see that we can eliminate many terms by using the commutation equation. Then we try to see how the polynomial ring in two variables grows and then we will convince ourselves that we have a polynomial relation between $L_1$ and $L_2$.
\end{hint}
\begin{hint}[Hint for \Cref{ex:g=1}]
A complex torus $\CC \slash \Lambda$ in dimension $1$ is defined by a lattice $\Lambda = (\omega_1, \omega_2)$ with $\omega_i \in \CC^*$, and two lattices are isomorphic if and only if they are homothetic, i.e.~there exists $a \in \CC^*$ such that $\Lambda_1 = a \Lambda_2$. Thus the lattice $\langle \omega_1, \omega_2\rangle$ is isomorphic to the lattice $\langle 1, \tau \rangle = \langle 1, \omega_1 \slash \omega_2\rangle$, and we can make the choice such that $\Im (\omega_1 \slash \omega_2) > 0$ (otherwise just switch $\omega_1$ and $\omega_2$). The complex tori defined by the lattices $\langle 1,\tau\rangle$ and $\langle 1,\tau'\rangle$ are then isomorphic if and only if it is the same lattice, up to scaling, which is equivalent to $\tau'=(a\tau+b)/(c\tau+d)$ for some $\left(\begin{smallmatrix} a&b\\ c&d\end{smallmatrix}\right)\in\SL_2(\ZZ)$. A fundamental domain for this action of $\SL(2,\ZZ)$ on the upper half-plane~$\HH$ can be given as $\{\tau \in \HH : -\half\le \Re (\tau)\le \half, |\tau| \ge 1\}$. So complex tori with nontrivial automorphism correspond to the points in the fundamental domain which have nontrivial stabilizer in $SL_2(\ZZ)$ and this corresponds to $\tau = i$, $\tau = e^{2\pi i\slash 3}\simeq e^{\pi i \slash 3}$ in the fundamental domain. We now encourage you to think about the endomorphisms.
\end{hint}
\begin{hint}[Hint for \Cref{ex:3}]
For $g = 1$, the theta divisor is zero dimensional, and by calculating the degree of $\theta$ we know that the theta divisor is just one point. Riemann theta singularity theorem says that at this point the corresponding line bundle has a one-dimensional space of sections. This is the trivial line bundle which has constant functions as sections. For $g = 2$, the theta divisor is one-dimensional, and is exactly the image of the curve under the Abel-Jacobi map discussed below, so in this case the theta divisor is smooth. For $g = 3$, at the smooth point~$D$ of the theta divisor $\dim H^0(C, [D])=1$, while if $D\in\Sing\Theta$, then $\dim H^0(C, [D]) \ge 2$, which means, since $\deg D=g-1=2$, that there exists a nonconstant function on $C$ with two poles. This implies that $C$ is a degree two cover of $\CC\PP^1$, so $C$ is a hyperelliptic curve. So we have two quite different cases here: either the theta divisor is smooth and the curve is non-hyperelliptic, or the theta divisor is singular (in fact in only one point), and the curve is hyperelliptic.
\end{hint}

\bibliographystyle{plain}
\bibliography{GrushevskyXie}
\end{document}